\documentclass[]{article}
\usepackage[utf8]{inputenc}
\usepackage[a4paper,top=2.5cm,bottom=2.5cm,left=3cm,right=3cm]{geometry}
\usepackage{graphicx} 
\usepackage{float}
\usepackage{textcomp}
\usepackage{xcolor}
\usepackage{amsmath}
\usepackage{amssymb}
\usepackage{amsthm}
\usepackage{amsfonts}
\usepackage[english]{babel}
\usepackage{verbatim}
\usepackage{array}
\usepackage{hyperref}
\usepackage{cases}
\usepackage{afterpage}
\usepackage{url}    
\usepackage{esint}    
\usepackage{caption}
\usepackage{subcaption}                                                                                          
\newcommand{\authoraddress}[2]{%
	\textsc{#1} \textit{E-mail address:} \protect\url{#2}%
}                                                                                                                     
\newcommand{\AuthorAddressone}{                                                                                          
	\authoraddress{Institute for Applied Mathematics, University of Bonn, 53115 Bonn, Germany.}{dematte@iam.uni-bonn.de}%
} 
\newcommand{\AuthorAddresstwo}{                                                                                          
	\authoraddress{Institute for Applied Mathematics, University of Bonn, 53115 Bonn, Germany.}{velazquez@iam.uni-bonn.de}%
} 
\numberwithin{equation}{section}
\newtheorem{theorem}{Theorem}[section]
\newtheorem{lemma}{Lemma}[section]
\newtheorem{corollary}{Corollary}[section]
\newtheorem{prop}{Proposition}[section]

\theoremstyle{definition}

\theoremstyle{remark}
\newtheorem*{remark}{Remark}

\newcommand{\rchi}{\protect\raisebox{2pt}{$ \chi $}}
\newcommand{\ou}{\overline{u}}

\newcommand{\ov}{\overline{v}}
\newcommand{\uv}{\underline{v}}
\newcommand{\og}{\overline{g}}

\newcommand{\oU}{\overline{U}}
\newcommand{\oV}{\overline{V}}
\newcommand{\oH}{\overline{H}}

\newcommand{\oLL}{\overline{\mathcal{L}}}
\newcommand{\tu}{\tilde{u}}

\newcommand{\F}{\mathcal{F}}
\newcommand{\LL}{\mathcal{L}}
\newcommand{\Ss}{\mathbb{S}}
\newcommand{\RR}{\mathbb{R}}
\newcommand{\Rot}{\mathcal{R}}
\newcommand{\eps}{\varepsilon}

\newcommand{\dist}{\text{dist}}
\newcommand\Bnu[1]{B_\nu\left(T\left({#1}\right)\right)}
\newcommand\LpR[1]{L^{{#1}}\left(\RR\right)}
\newcommand\Lprr[1]{L^{{#1}}\left(\RR^3\right)}
\newcommand{\intnu}{\int_{0}^\infty d\nu}
\newcommand{\intS}{\int_{\Ss^2}dn}
\newcommand{\intnN}{\int_{n\cdot N<0}dn}
\newcommand{\intR}{\int_{\RR_+\times\RR^2}}
\newcommand{\intup}{\int_0^\infty}
\newcommand{\intbig}{\int_{-\infty}^\infty}
\newcommand{\intdown}{\int_{-\infty}^0}
\newcommand{\gradx}{\nabla_x}
\newcommand{\grady}{\nabla_y}
\newcommand{\alnu}{\alpha_\nu}
\newcommand{\Inux}{I_\nu\left(x,n\right)}
\newcommand{\Inuy}{I_\nu\left(y,n\right)}
\newcommand{\Bnux}{B_\nu\left(T\left(x\right)\right)}
\newcommand{\Bnuy}{B_\nu\left(T\left(y\right)\right)}
\newcommand{\Bnueta}{B_\nu\left(T\left(\eta\right)\right)}
\newcommand{\sgn}{\text{sgn}}

\newcommand{\diam}{\textnormal{diam}}
\newcommand{\bnd}{\partial\Omega}
\newcommand{\eptw}{\eps^{\frac{1}{2}+2\delta}}
\newcommand{\epth}{\eps^{\frac{1}{2}+3\delta}}
\newcommand{\epfo}{\eps^{\frac{1}{2}+4\delta}}
\newcommand{\epsi}{\eps^{\frac{1}{2}+6\delta}}
\newcommand\Kern[1]{\frac{e^{-\frac{\left|{#1}\right|}{\eps}}}{4\pi\eps\left|{#1}\right|^2}}
\newcommand\PlanarKern[1]{K\left(#1\right)}
\newcommand\Epskern[1]{K_\eps\left(#1\right)}
\newcommand\deps[1]{d_\eps\left(#1\right)}

\DeclareMathOperator*{\Div}{div}
\title{On the diffusion approximation of the stationary radiative transfer equation with absorption and emission}
\author{Elena Demattè\thanks{\AuthorAddressone}, Juan J.L. Velázquez\thanks{\AuthorAddresstwo}}

\begin{document}
	
	\maketitle
\begin{abstract}
	We study the situation in which the distribution of temperature a body is due to its interaction with radiation. We consider the boundary value problem for the stationary radiative transfer equation under the assumption of the local thermodynamic equilibrium. We study the diffusion equilibrium approximation in the absence of scattering. We consider absorption coefficient independent of the frequency $ \nu $ (the so-called Grey approximation) and the limit when the photons' mean free path tends to zero, i.e. the absorption coefficient tends to infinity. We show that the densitive of radiative energy $ u $, which is proportional to the fourth power of the temperature due to the Stefan-Boltzmann law, solves in the limit an elliptic equation where the boundary value can be determined uniquely in terms of the original boundary condition. We derive formally with the method of matched asymptotic expansions the boundary condition for the limit problem and we prove rigorously the convergence to the solution of the limit problem with a careful analysis of some non-local integral operators. The method developed here allows to prove all the results using only maximum principle arguments.

\end{abstract}
\textbf{Acknowledgments:} The authors gratefully acknowledge the financial support of the collaborative research centre \textit{The mathematics of emerging effects} (CRC 1060, Project-ID 211504053) and Bonn International Graduate School of Mathematics (BIGS) at the Hausdorff Center for Mathematics founded through the Deutsche Forschungsgemeinschaft (DFG, German Research Foundation) under Germany’s Excellence Strategy – EXC-2047/1 – 390685813. \\


\textbf{Statements and Declarations:} The authors have no relevant financial or non-financial interests to disclose.\\

\textbf{Data availability:} Data sharing not applicable to this article as no datasets were generated or analysed during the current study.
	\tableofcontents
	\section{Introduction}
	The radiative transfer equation is the kinetic equation which describes the distribution of energy and direction of motions of a set of photons, which can be absorbed and scattered by a medium. This equation can be used to describe the transfer of heat in a material due to radiative processes. The radiative transfer equation can be written in its more general form as
	\begin{equation}\label{rte}
	\frac{1}{c}\partial_t I_\nu(x,n,t)+n\cdot\nabla_xI_\nu(x,n,t)=\alpha_\nu^e-\alpha_\nu^aI_\nu(x,n,t)-\alpha_\nu^sI_\nu(x,n,t)+\alpha_\nu^s\int_{\Ss^2}K(n,n')I_\nu(x,n',t)dn'.
	\end{equation} 
	We denote by $ I_\nu(x,n,t) $ the intensity of radiation (i.e. radiating energy) of frequency $ \nu $ at position $ x\in\Omega $ and in direction $ n\in\Ss^2 $ and at time $ t\geq0 $. The coefficients $ \alpha_\nu^a $, $ \alpha_\nu^e $ and $ \alpha_\nu^s $ are respectively the absorption, the emission and the scattering coefficient. In the scattering term the kernel is normalized such that $ \int_{\Ss^2}K(n,n')dn'=1 $. The speed of light is indicated by $ c $.\\
	In this paper we focus on the stationary problem and on processes, where the scattering is negligible. Therefore the equation we will study reduces to
	\begin{equation}\label{stationary rte}
	n\cdot \nabla_x I_\nu \left(x,n\right)= \alpha_\nu^e-\alpha^a_\nu I_\nu \left(x,n\right).
	\end{equation}
	
	In this article we consider the situation of local thermal equilibrium (LTE), which means that at every point $ x\in\Omega $ there is a well-defined temperature $ T(x)\geq 0 $. This yields according to the Kirchhoff's law (cf. \cite{Zeldovic}) the following relation for the absorption and emission coefficient
	\begin{equation*}\label{emission}
	\alpha_\nu^e(x)=\alpha_\nu^a(x) B_\nu(T(x)),
	\end{equation*}
	where $ B_\nu(T(x))=\frac{2h\nu^3}{c^2}\frac{1}{e^{\frac{h\nu}{kT}}-1} $ is the Plank emission of a black body at temperature $ T(x) $ and $ k $ the Boltzmann constant. Moreover, it is well-known that
	\begin{equation}\label{sigma}
	\int_0^\infty B_\nu\left(T(x)\right)\;d\nu=\sigma T^4(x),
	\end{equation} where $ \sigma=\frac{2\pi^4k^4}{15 h^3 c^2} $ is the Stefan-Boltzmann constant. We will denote for simplicity from now on as the absorption coefficient $ \alpha_\nu^a $ as $ \alpha_\nu $.\\
	
	The solution $ I_\nu(x,n) $ of \eqref{stationary rte} can be used to compute the flux of energy at each point $ x\in\Omega $ of the material, which is given by
	\begin{equation*}\label{flux}
	\F(x):=\int_{0}^{\infty}d\nu\int_{\Ss^2}dn\;n\:I_\nu\left(x,n\right).
	\end{equation*}
	In the stationary case, if the temperature is independent of time at every point, the total incoming and outgoing flux of energy should balance. In mathematical terms this can be formulated by the condition for the flux of energy to be divergence free, i.e.
	\begin{equation*}\label{div free}
	\nabla_x\cdot \F(x)=0.
	\end{equation*}
	This situation is denoted in the physical literature by pointwise radiative equilibrium.\\

	We study the situation when the radiation is coming from a very far source of infinite distance. This can be formalized in mathematical terms by means of the boundary condition
	\begin{equation}\label{boundary}
	I_\nu \left(x,n\right)=g_\nu\left(n\right)\geq 0
	\end{equation}
	if $ x\in\partial\Omega  $ and $ n\cdot N_x<0 $ for $ N_x $ the outer normal vector of the boundary at point $ x\in\bnd $. Throughout this paper we will consider $ \Omega\subset\RR^3 $ to be a bounded convex domain with $ C^3 $-boundary.\\
	
	We are concerned in this paper with the study of the diffusion approximation that arises in optically thick media. This means, that we consider the case when the optical depth is very large compared to the characteristic length of the system. Hence, we rescale and for $ \eps\ll 1 $ we consider the following boundary value problem 
		\begin{equation}\label{bvpnoscattering}
	\begin{cases}
	n\cdot \nabla_x I_\nu \left(x,n\right)= \frac{\alpha_\nu(x)}{\eps}\left(B_\nu\left(T\left(x\right)\right)-I_\nu \left(x,n\right)\right) & x\in\Omega,\\
	\nabla_x\cdot \F=0 & x\in\Omega,\\
	I_\nu \left(x,n\right)=g_\nu\left(n\right)& x\in\partial\Omega \text{ and }n\cdot N_x<0.
	\end{cases}
	\end{equation}
	For the solution to this equation we will prove that the intensity of radiation $ I_\nu(x,n) $ is approximately the Plank distribution $ B_\nu(T(x)) $ with the local temperature at each point $ x\in\Omega $, i.e. we will show
	\begin{equation}\label{result}
	I_\nu^\eps(x,n)\to B_\nu(T(x))\;\;\;\;\;\;\;\;\text{ as }\eps\to 0.
	\end{equation}Notice however, that this approximation cannot be expected for points $ x $ that are close to the boundary $ \bnd $.
	The situation in which \eqref{result} holds is denoted in the physical literature as the diffusion equilibrium approximation (see e.g. \cite{mihalas} and \cite{Zeldovic}). More precisely, we will consider the limit problem when $ \eps\to 0 $ and we will rigorously prove that it is given by a Dirichlet problem for the heat equation of the temperature with boundary value uniquely determined by the incoming source $ g_\nu(n) $ and the outer normal $ N_x $ for $ x\in\bnd $. The main result we will prove in this paper is for the so called Grey approximation, i.e. the case when the absorption coefficient is independent of the frequency $ \nu $. The main reason for that is that some of the estimates are already in this case very technical. Hopefully, the type of methods we are developing in this paper can be extended to the non-Grey case.
	\begin{theorem}\label{main result}
Let $ \alpha_\nu(x)=\alpha(x) $ independent of $ \nu $, $ \alpha\in C^3\left(\Omega\right) $, $ g_\nu\geq 0 $ with $ \int_0^\infty g_\nu(n)\;d\nu\in L^{\infty}\left(\Ss\right) $ in \eqref{bvpnoscattering}, $ \Omega $ bounded convex with $C^3 $-boundary and strictly positive curvature. Let $ T_\eps $ be the temperature associated to the intensity $ I_\nu $ solution to the initial value problem \eqref{bvpnoscattering}. Then there exists a functional $ T_\Omega:L^\infty\left(\Ss^2,L^1\left(\RR_+\right)\right)\to C\left(\bnd\right) $ which maps $ g_\nu $ to a continuous function $ T_\Omega[g_\nu](p) $ on the boundary $ p\in\bnd $ such that
\begin{equation*}
T_\eps(x)\to T(x)
\end{equation*}
uniformly in every compact subset of $ \Omega $, where $ T $ is the solution to the Dirichlet problem
\begin{equation*}
\begin{cases}
-\Div\left(\frac{\sigma 4 T^3}{\alpha}\nabla T\right)=0 &x\in\Omega,\\T(p)=T_\Omega[g_\nu](p)&p\in\bnd.
\end{cases}
\end{equation*}
	\end{theorem}
	\subsection{Motivation and previous results}
	The computation of the distribution of temperature of matter interacting with radiation is an important issue in many physical application and in addition it rises interesting mathematical questions. The kinetic equation describing the interaction of matter with radiation is the radiative transfer equation. A detailed explanation of its derivation and its main properties can be found in \cite{Chandrasekhar, mihalas, oxenius, Rutten, Zeldovic}. In particular, in \cite{mihalas, Zeldovic} there is an extensive discussion about the diffusion equilibrium approximation and the situations where this can be expected or not.
	
	Since the earlier result by Compton \cite{compton} in 1922 the interaction of a gas with radiation has been extensively studied. Milne for example studied a simplified model, where the radiation is monochromatic and the gas density depends only on one space variable (cf. \cite{Milne}). 
	
	A question which has been much studied in mathematical literature is the situation in which $ \alpha_\nu^e=\alpha_\nu^a=0 $ in  \eqref{rte}, i.e. the interaction between matter and radiation is due to scattering only. In this case the problem reduces to
	\begin{equation}\label{rtescattering}
		\frac{1}{c}\partial_t I_\nu(x,n,t)+n\cdot\nabla_xI_\nu(x,n,t)=-\alpha_\nu^s(x)I_\nu(x,n,t)+\alpha_\nu^s(x)\int_{\Ss^2}K(n,n')I_\nu(x,n',t)dn'.
	\end{equation}The same equation arises also in the study of neutron transport, a problem which has been extensively studied in mathematics.
	
	It turns out that in the Grey approximation the problem \eqref{bvpnoscattering} can be reduced exactly to the study of a particular neutron transport equation, namely the case when the kernel $ K $ is constant $ 1 $. Indeed, denoting by $ J(x,n)=\int_0^\infty I_\nu(x,n)\;d\nu $ and combining the first two equations of \eqref{bvpnoscattering} we obtain $ \int_0^\infty B_\nu(T(x))\;d\nu=\fint_{\Ss^2}J(x,n)\;dn=\frac{1}{4\pi}\int_{\Ss^2} J(x,n)\;dn$. Hence, equation \eqref{bvpnoscattering} is equivalent to the study of 
	\begin{equation}\label{equiv}
	\begin{cases}
	n\cdot \nabla_x J(x,n)=\frac{\alpha(x)}{\eps}\left(\fint_{\Ss^2} J(x,n)\;dn-J(x,n)\right)& \text{ if }x\in\Omega,\\J(x,n)=\int_0^\infty g_\nu(n)\;d\nu& \text{ if }x\in\bnd, \;n\cdot N_{x}<0.
	\end{cases}
	\end{equation}
	However, the equivalence between \eqref{bvpnoscattering} and \eqref{equiv} does not hold in the non-Grey case. 
	The properties of equation \eqref{equiv} as well as the diffusion approximation limit have been studied for a long time, starting with the seminal paper \cite{papanicolaou} of 1979, where the stationary version of \eqref{rtescattering} was studied. In that work the authors proved the diffusion approximation for the neutron transport equation using a stochastic method. The result they obtained for $ J $ would imply in particular our main Theorem \ref{main result}.  
	
	More recently, in a series of papers \cite{GuoLei2d,Leiunsteady,Lei3d,wuguo,annulus} Yan Guo and Lei Wu have studied the diffusion approximation of both the stationary and the time dependent neutron transport equation \eqref{rtescattering} when $ K\equiv 1 $ and $ \alpha^s_\nu(x)\equiv 1 $, independent of $ x $, for different classes of boundary conditions in $ 2 $ and $ 3 $ dimensions, in bounded convex domains or annuli (in 2D). In particular the result in paper \cite{Lei3d} imply again the main Theorem \ref{main result} when $ \alpha\equiv 1 $. Their proof relies on PDE methods and not on a stochastic approach. Moreover, they also computed the geometric approximation in the structure of the boundary layer.
	
	The main goal of this paper is to develop a method which allows to obtain diffusive limit approximations like the one in Theorem \ref{main result} for the radiative transfer equation \eqref{rte} using PDE methods that rely only in maximum principle tools. This tools are different from those used by Guo and Wu. Specifically, the method in \cite{GuoLei2d,Leiunsteady,Lei3d,wuguo,annulus} relies on the $ L^2 $-$ L^p $-$ L^\infty $ estimates that were introduced for the analysis of kinetic equations by Yan Guo in \cite{GuoL2}. In particular, the method is based on the estimates of the velocity distribution $ J $. Our approach is based on the direct derivation of estimates for the temperature $ T(x) $ associated to a given distribution of radiation $ I_\nu(x,n) $. More precisely, equation \eqref{bvpnoscattering} can be reformulated as a non-local integral equation for the temperature (c.f. \cite{jang}). In the case of the Grey approximation we have the following equation for $ u(x)=4\pi\sigma T^4(x) $
	\begin{equation}\label{integral}
	u(x)-\int_{\Omega} K_\eps(x,\eta)u(\eta)\;d\eta=S(x),
	\end{equation} 
	where the precise form of the kernel $ K_\eps $ and of the source $ S(x) $ are discussed in Sections 4 and 5.
	
	Equation \eqref{integral} can be thought as a non-local elliptic equation which in particular satisfies good properties, such as the maximum principle. Specifically, our proof relies only in finding suitable supersolutions and applying the maximum principle. The way in which we constructed these supersolutions is mimicking particular solutions of elliptic equations with constant coefficients. These supersolutions give also an insight of the behavior of the solution near the boundary $ \bnd $. Our hope is that the method developed in this paper could be extended to the non-Grey case, at least for some suitable choice of $ \alpha_\nu(x) $. One reason why this should be possible is that \cite{jang} shows how to solve the non-local equation \eqref{integral} for some class of non-Grey problems.
	
	Another type of diffusion approximation for \eqref{rte} is the one in \cite{guo2,guo1} in which it has been considered the situation when $ \alpha_\nu^s\to\infty $ while $ \alpha_\nu^e $ and $ \alpha_\nu^a $ remain bounded combined with the equation for balancing the energy either in the one dimensional case or in the whole space.
	
	The well-posedness and the diffusion approximation of the time dependent problem \eqref{rtescattering} in the frame work of $ L^1 $-functions using the theory of $ m $-accretive operators has ben studied in a series of papers \cite{Golse3,golse5}. Seemingly, although the techniques in these papers allow to develop a theory for the time dependent problem, they do not provide information about the stationary solution.
	
	Some versions of the stationary problem involving the radiative transfer equation can be found in \cite{finnland2,finnland1,porzio,finnland3}. The problems studied in these papers include also heat conduction and different type of boundary condition of our model (for a more detailed discussion see \cite{jang}).
	
	It is important to emphasize that equation \eqref{bvpnoscattering} is very different in the non-Grey case from the scattering problem \eqref{rtescattering}, in the sense that the system \eqref{bvpnoscattering} provides an equation for the temperature. Specifically, the equation $ \nabla_x\cdot \mathcal{F}=0 $ is automatic satisfied in the stationary version of \eqref{rtescattering}. Physically, this is due to the fact that the radiation arriving at every point is just deflected. Equation \eqref{bvpnoscattering} plays the same role as the Laplace equation in order to describe the stationary distribution of temperature in systems where the energy is transported by means of heat conduction. In the case of \eqref{bvpnoscattering} the energy is transported by means of radiation which results in non-locality for determining the temperature distribution. The fact that the determination of the temperature in a body where the energy is transported by radiation is non-local was first formulated in \cite{holstein}. Since the approximation \eqref{result} fails at the boundary, some boundary layers appears for which the intensity of radiation $ I_\nu^\eps $ differs from the Plank distribution $ B_\nu(T) $. Hence, a careful analysis must be made for these boundary layers where the radiation is out of equilibrium. This will be essential in order to determine the functional $ T_\Omega $ in Theorem \ref{main result}, which defines the temperature at every point of the boundary. 
	 
	Finally, we mention that one can consider more complicated interactions between radiation and matter. For instance when the matter that interacts with radiation is a moving fluid. (cf. \cite{Golse1, Golse2,mihalas, Zeldovic}). The case when the interacting medium is a Boltzmann gas whose molecules can be in different energy layers has been considered in \cite{dematte,paper,oxenius,rossani}.
	\subsection{Structure of the paper and notation}
	We aim to prove Theorem \ref{main result}. In Section 2 we will formally derive the limit problem using ideas from the theory of matched asymptotic expansions and from boundary layers theory. Sections 3 and 4 deal with the case of constant absorption coefficient $ \alpha\equiv 1 $, while Section 5 shows the diffusion approximation in the case of space dependent coefficient $ \alpha\in C^3(\Omega) $. In Section 3 we will study the properties of the solution for the non-local integral equation describing the distribution of energy at the boundary layer. In particular, this will allow us to construct in the case of $ \alpha\equiv 1 $ the functional $ T_\Omega $  of Theorem \ref{main result} which assigns the boundary value of the limit solution at every point of the boundary. Important tools we will use for the well-posedness are the maximum principle and a combination of Fourier methods with the classical tools of sub- and super-solution which resembles the Perron method for the Laplace equation. For the asymptotic theory we use the theory of Fourier transform for distributions. Section 4 deals with the rigorous proof of the convergence to the diffusion equilibrium approximation for the limit problem in the constant absorption coefficient case. Here the main tool is the maximum principle for the non-local integral operator we can construct for the boundary value problem \eqref{bvpnoscattering}. Finally, in Section 5 we consider the more general Grey approximation in which $ \alpha\in C^3(\overline{\Omega}) $ is not constant. We will derive formally the limit problem and then we will extend the theory developed in Sections 3 and 4 for this case. We will hence prove again by means of the maximum principle and suitable supersolutions the convergence to the diffusion equilibrium approximation.  \\
	
	We introduce here some notation we will use throughout this paper. First of all, $ \Omega\subset \RR^3 $ is an open bounded convex domain with $ C^3 $-boundary and strictly positive curvature. In order to avoid meaningless constants we assume without loss of generality that $ 0\in\overline{\Omega} $. $ N_x $ indicates always the outer normal vector for a point $ x\in\bnd $. 
	
	We assume $ \Omega $ to be convex in order to simplify some geometrical argument. First of all this assumption implies that for every point $ p\in\bnd $ the tangent plane to the boundary at $ p $ divided the space $ \RR^3 $ in two disjoint half-spaces, one of them containing the whole domain $ \Omega $. This will be used several times in the definition for every point $ p\in\bnd $ of the isometric transformation mapping $ p $ to $ 0 $ and $ \Omega $ in the positive half-space $ \RR_+\times\RR^2 $. The assumption of convexity can be relaxed and the geometrical estimates should still hold, but we would need a more careful analysis of the geometry of the problem.\\
	Moreover, for $ g_\nu(n)\geq 0 $ with $ \int_0^\infty g_\nu(n)\;d\nu\in L^\infty\left(\Ss^2\right) $ we define the norms
	\begin{equation}\label{norm1}
	\Arrowvert g\Arrowvert_1:=\int_0^\infty \int_{\Ss^2} g_\nu(n)\;d\nu\;dn
	\end{equation}
	and 
		\begin{equation}\label{norminfty}
	\Arrowvert g\Arrowvert_\infty:=\sup\limits_{n\in\Ss^2}\left(\int_0^\infty g_\nu(n)\;d\nu\right).
	\end{equation}
	\begin{remark} The reason why we are assuming the seemingly restrictive boundary condition \eqref{boundary} is because we are supposing that the source of radiation is placed at infinity. We can obtain analogous results to the one of the paper if we consider the more general boundary condition $ g_\nu(n,x) $ depending also on $ x\in\bnd $. In addition to the assumption above we need to require $ g_\nu(n,x) $ to be a $ C^1 $-function with respect to $ x\in\bnd $.
	\end{remark}
\begin{figure}[H]\centering
\includegraphics[height=5cm]{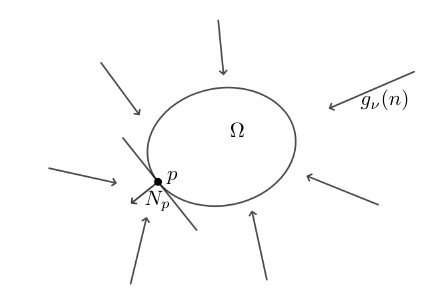}\caption{Representation of the boundary value problem.}
\end{figure}
	For any point $ p\in\bnd $ we choose a fixed isometry mapping $ p $ to $ 0 $ and the vector $ N_p $ to $ -e_1 $. We will denote this rigid motion by $ \Rot_p:\RR^3\to\RR^3 $ with the following properties
	\begin{equation}\label{rotp}
	\Rot_p(p)=0\;\;\;\;\;\;\;\text{ and }\;\;\;\;\;\;\;\Rot_p(N_p)=-e_1.
	\end{equation} 
	Finally, we define by 
	\begin{equation}\label{proj}
	\begin{split}
	 \pi_{\bnd}:\left\{x\in\RR^3:\dist(x,\bnd)<\delta\right\}&\to\bnd\\x&\mapsto\pi_{\bnd}(x) \end{split}
	\end{equation} the projection to the unique closest point in the boundary $ \bnd $. This function is continuous and well-defined in small neighborhood of $ \bnd $, i.e. for $ \delta>0 $ small enough. 
	\section{Derivation of the limit problem}

	\subsection{Formal derivation of the limit problem in the diffusive equilibrium approximation}
	We first remind how to obtain formally the equation for the interior in the limit problem. First of all we expand the intensity of radiation
	\begin{equation}\label{intensityexpanded}
	I_\nu \left(x,n\right)=f_\nu^0\left(x,n\right)+\eps f_\nu^1\left(x,n\right)+\eps^2...
	\end{equation}
	Substituting it in the first equation of \eqref{bvpnoscattering} and identifying the terms containing $ \eps^{-1} $ and $ \eps^0 $ we see
	\begin{equation*}\label{f0}
	f_\nu^0\left(x,n\right)=B_\nu\left(T\left(x\right)\right)
	\end{equation*}
	and 
	\begin{equation*}\label{f1}
	f_\nu^1\left(x,n\right)= -\frac{1}{\alpha_\nu(x)}n\cdot \nabla_xB_\nu\left(T\left(x\right)\right)
	\end{equation*}
	Using the second equation in \eqref{bvpnoscattering} and the expansion in \eqref{intensityexpanded} we deduce
	\begin{equation*}\label{interior1}
	\begin{split}
	0&=\intnu\intS \;n\cdot\gradx\Inux\\
	&=\Div\left[\intnu\intS \;n\Bnux\right]-\eps\Div\left[\intnu\intS\:\left(n\otimes n\right)\frac{1}{\alpha_\nu(x)}\gradx\Bnux\right]\\
	&=-\eps\frac{4}{3}\pi\Div\left(\left(\intnu\frac{1}{\alpha_\nu(x)}\gradx \Bnux\right)\right),
	\end{split}
	\end{equation*}
	where we used 
	\begin{equation*}
	\intS (n\otimes n)=\frac{4}{3}\pi\text{Id}\;\;\;\;\;\text{ and }\;\;\;\;\;\int_{\ss^2}dn\;n=0.
	\end{equation*}
	Therefore,
	\begin{equation}\label{nongreyinterior}
	\Div\left(\kappa\left(T\right)\gradx T\right)=0,
	\end{equation}
	where $ \kappa\left(T\right):=\intnu \frac{\partial_T \Bnux}{\alnu(x)} $.
	In the particular case of the Grey approximation when $ \alnu(x)=1 $ we have $ \kappa(T)= 4\sigma T^3(x) $. Then defining $ u(x):=4\pi\sigma T^4(x) $ we obtain the following equation
	\begin{equation*}\label{greyinterior}
	\Delta u=0.
	\end{equation*}
	This is the limit  problem we will study.
	\subsection{Formal derivation of boundary condition for the limit problem in the diffusive equilibrium approximation}
	In order to obtain the intensity of radiation closed to the boundary of $ \Omega $ we derive a boundary layer equation, whose solution will be used to determine the value of the temperature at the boundary by means of a matched argument. Suppose that $ x_0\in\bnd $, without loss of generality we can assume $ x_0=0 $ and $ N_{x_0}=N=-e_1 $ using the rigid motion $ \Rot_{x_0} $ defined in \eqref{rotp} and putting $ \og_\nu(n):=g_\nu\left(\Rot^{-1}_{x_0}(n)\right) $. We rescale $ x=\eps y $, where $ y\in\frac{1}{\eps}\Omega $. Thus, at the leading order as $ \eps\to 0 $ we obtain $ \alpha_\nu(x)=\alpha_\nu(\eps y)=\alpha_\nu(0)+\mathcal{O}(\eps) $. Taking $ \eps\to 0 $ we obtain that the intensity of radiation satisfies
	\begin{equation}\label{scaledbvpproblem}
	\begin{cases}
	n\cdot \grady\Inuy= \alnu(0)\left(\Bnuy-\Inuy\right) & y\in\RR_+\times\RR^2\\
	\grady\cdot \F=0 & y\in\RR_+\times\RR^2\\
	\Inuy=\og_\nu\left(n\right)& y\in \{0\}\times\RR^2 \text{ and }n\cdot N<0
	\end{cases}
	\end{equation}
	The first equation can be solved for $ I_\nu $ using the method of characteristics. Given $ y\in \RR_+\times\RR^2 $ and $ n\in\Ss^2 $ with $ n\cdot N<0 $ we call $Y(y,n)$ the unique point belonging to $\partial\left(\RR_+\times\RR^2\right)=\{0\}\times\RR^2 $ such that 
	\begin{equation*}
	y=Y(y,n)+s(y,n)n,
	\end{equation*}
	where $ s(y,n)=\left|y-Y(y,n)\right| $. Notice that $ s(y,n) $ is the distance to the first intersection point of the boundary $ \{0\}\times\RR^2 $ with the half line $ \{y-tn:t>0\} $. For $ n\cdot N\geq 0 $ we define $ s(y,n)=\infty $. Solving the equation by characteristics we obtain
	\begin{equation*}\label{boundary1}
	\begin{split}
	\Inuy=& \og_\nu(n)e^{-\alnu(0) s(y,n)}\rchi_{n\cdot N<0}+\int_{0}^{s(y,n)}\;e^{-\alnu(0) t}\alnu(0) \Bnu{y-tn} dt.
	\end{split}
	\end{equation*}
	Using the second equation in the rescaled problem \eqref{scaledbvpproblem} we calculate
	\begin{equation}\label{boundary2}
	\begin{split}
	0=&\Div\left[\intnu\intnN\;n\og_\nu(n)e^{-\alnu(0) s(y,n)}+\intnu\intS\int_{0}^{s(y,n)}dt\;ne^{-\alnu(0) t}\alnu \Bnu{y-tn}\right]\\
	=& -\intnu\intnN \;\og_\nu(n)\alnu n\cdot\grady s(y,n) e^{-\alnu(0) s(y,n)}\\& +\Div\left(\intnu\intR d\eta\; \frac{y-\eta}{\left|y-\eta\right|^3} e^{-\alnu(0)\left|y-\eta\right|}\alnu(0) \Bnu{\eta}\right)\\
	=&-\intnu\intnN \;\og_\nu(n)\alnu(0) e^{-\alnu(0) s(y,n)}+4\pi\intnu(0) \;\alnu \Bnuy\\
	&-\intnu\intR d\eta\;\frac{\alnu^2(0)}{\left|y-\eta\right|^2}e^{-\alnu(0)\left|y-\eta\right|}\Bnu{\eta}.
	\end{split}
	\end{equation}
	The second equality holds via the spherical change of variable
	\begin{equation*}
	\begin{split}
	\Ss^2\times \RR_+&\to \RR_+\times\RR^2\\
	(n,t)&\mapsto \eta=y-tn
	\end{split}
	\end{equation*}
	so that $ n=\frac{y-\eta}{\left|y-\eta\right|}$. For the third inequality we use on the one hand that $ \Div_y\left(\frac{y-\eta}{\left|y-\eta\right|^3}\right)=4\pi\delta(y-\eta) $ and on the other hand that $ n\cdot\grady s(y,n)=1 $. The latter can be seen by the fact that 
	\begin{equation*}
	Y(y,n)+s(y+tn,n)n=y+tn=Y(y+tn,n)+s(y+tn,n)n.
	\end{equation*}
	This implies that $ Y(y+tn,n) $ is $ t $-constant and therefore $ 1=\partial_t s(y+tn,n)=\left(\grady s(y+tn,n)\right)\cdot n $. We assume now that the temperature depends only on the first variable. This can be expected because we are considering limits for $ \eps\ll1 $ and hence the temperature can be considered to depend only on the distance to the point $ x_0 $, which is approximated by the first variable in this setting. After the change of variables $ \xi=(y_2+\eta_2,y_3-\eta_3) $ and calling $ y-\eta:=y_1-\eta_1 $ the last integral in \eqref{boundary2} can be written as
	\begin{equation*}\label{boundary3}
	\intnu\int_{\RR_+}d\eta\int_{\RR^2}d\xi \;\alnu^2(0)\frac{e^{-\alnu(0)\sqrt{(y-\eta)^2+|\xi|^2}}}{(y-\eta)^2+|\xi|^2}\Bnueta.
	\end{equation*}
	Using polar coordinates we obtain
	\begin{equation}\label{kernel}
	\begin{split}
	\int_{\RR^2}d\xi \;\frac{e^{-\alnu(0)\sqrt{(y-\eta)^2+|\xi|^2}}}{(y-\eta)^2+|\xi|^2}=&\pi \int_{|y-\eta|^2}^\infty dx\; \frac{e^{-\alnu(0)\sqrt{x}}}{x}= 2\pi \int_{\alnu(0)|y-\eta|}^\infty dt\; \frac{e^{-t}}{t}= 4\pi K(\alnu(0) |y-\eta|),
	\end{split}
	\end{equation}
	where we will denote $ K(x)=\frac{1}{2}\int_{|x|}^\infty dt \;\frac{e^{-t}}{t} $ as the normalized exponential integral.
	
	Notice that $ s(y,n)= \frac{y_1}{\left|n\cdot N\right|} $ if $ n\cdot N<0 $. We can summarize the equation the temperature satisfies in the non-Grey approximation as follows
	\begin{equation}\label{bvpnongreyboundary}
	\begin{split}
	&\intnu \;\alnu(0) B_\nu\left(T(y_1)\right)-\intnu \int_{0}^\infty d\eta\; \alnu^2(0) K\left(\alnu(0)\left|y_1-\eta_1\right|\right)B_\nu\left(T\left(\eta_1\right)\right)\\
	=& \intnu\intnN \;\og_\nu(n)\alnu(0) e^{-\alnu(0) \frac{y_1}{\left|n\cdot N\right|}}.
	\end{split}
	\end{equation}
	In the particular case of the Grey approximation when $ \alpha\equiv 1 $ using that $ u(y)= 4\pi\sigma T^4(y) $ we can simplify equation \eqref{bvpnongreyboundary} by property \eqref{sigma}
	\begin{equation}\label{bvpgreyboundary}
	u(y_1)-\int_0^\infty d\eta\;K(y_1-\eta)u(\eta)= \intnu\intnN \;\og_\nu(n)e^{-\frac{y_1}{\left|n\cdot N\right|}}.
	\end{equation}
	In some occasions, when the dependence of the boundary layer function $ u $ on the point $ p\in\bnd $ is needed, we will use the notation $ \ou(y_1,p) $, where this function solves according to the rigid motion $ \Rot_p $ in \eqref{rotp}
	\begin{equation}\label{bvpgreyboundary1}
	\ou(y_1,p)-\int_0^\infty d\eta\;K(y_1-\eta)u(\eta,p)= \intnu\int_{n\cdot N_p<0}dn \;g_\nu(n)e^{-\frac{y_1}{\left|n\cdot N_p\right|}}.
	\end{equation}
	For the rest of Section 2 and Section 3 we will focus on the study of $ \ou(y_1,p) $ for an arbitrary given $ p\in\bnd $, hence we will call $ u(y_1)=\ou(y_1,p) $ and $ N=N_p $. In order to simplify the reading from now on we set $ G(x)=\intnu\intnN \;\og_\nu(n)e^{-\frac{x}{\left|n\cdot N\right|}}\rchi_{\{x>0\}} $ and if we want to stress out the dependence on $ p\in\bnd $ we write $ G_p(x)= \intnu\int_{n\cdot N_p<0} \;g_\nu(n)e^{-\frac{x}{\left|n\cdot N_p\right|}}\rchi_{\{x>0\}} $.\\

	From now on until Section 5 we consider the case of constant absorption coefficient $ \alpha\equiv 1 $.
	\subsection{Some properties of the kernel}
	We consider the kernel $ K $ introduced in Section 2.2. We remark that $ K(x)=\frac{1}{2}E_1(|x|)$, where $ E_1 $ is the standard exponential integral function. See \cite{exponentialintegral}. We collect some properties of the normalized exponential integral.
	\begin{prop}\label{propK}
The function $ K $ satisfies $ \intbig dx\;K(x)=1 $, $ K\in\LpR{1}\cap\LpR{2} $ and the following estimate holds
\begin{equation*}\label{Kestimate}
\frac{1}{4}e^{-|x|}\ln(1+\frac{2}{|x|})\leq K(x)\leq \frac{1}{2}e^{-|x|}\ln(1+\frac{1}{|x|}).
\end{equation*}
Moreover, the Fourier transform of $ K $ is $ \hat{K}(\xi)= \frac{1}{\sqrt{2\pi}}\frac{\arctan(\xi)}{\xi} $.
\begin{proof}
Since $ K $ is even and non negative we can calculate, applying Tonelli's Theorem
\begin{equation*}\label{integralK}
\intbig K(s)\;ds=2\intup K(s)\;ds=\intup\int_s^\infty\frac{e^{-t}}{t}\;dt\;ds=\intup\frac{e^{-t}}{t}\int_0^t\;ds\;dt=\intup e^{-t}\;dt=1.
\end{equation*}
This proves also that $ K\in\LpR{1} $. 

For the square integrability we refer to equation 5.1.33 in \cite{exponentialintegral} and see $ \int_\RR \left|K(x)\right|^2\;dx=\ln(2) $. Estimate 5.1.20 in \cite{exponentialintegral} also implies $ \frac{1}{4}e^{-|x|}\ln(1+\frac{2}{|x|})\leq K(x)\leq \frac{1}{2}e^{-|x|}\ln(1+\frac{1}{|x|}) $.

We now move to the computation of the Fourier transform of the kernel $ K $. The kernel is an even function, hence we compute
\begin{equation*}\label{Fourier transform}
\begin{split}
\hat{K}(\xi)=&\frac{1}{\sqrt{2\pi}}\intbig e^{-i\xi x}K(x)\;dx=\frac{1}{\sqrt{2\pi}}\intup \cos\left(\xi x\right)\int_x^\infty \frac{e^{-t}}{t}\;dt\;dx\\=&\frac{1}{\sqrt{2\pi}}\frac{1}{\xi}\intup\frac{e^{-t}}{t}\sin(\xi t)\;dt=\frac{1}{\sqrt{2\pi}}\frac{\arctan(\xi)}{\xi}.
\end{split}
\end{equation*}
The last identity can be justified noticing that $ F(\xi)= \intup\frac{e^{-t}}{t}\sin(\xi t)\;dt$ has derivative $ F'(\xi)=\frac{1}{\xi^2+1} $.
\end{proof}
	\end{prop}
The following calculation will also be very useful in the next section.
\begin{prop}\label{propK1}
	Let $ x>0$. Then we can compute
	\begin{equation}\label{intK1}
	\int_{-x}^\infty K(s)\;ds= 1-\frac{e^{-x}}{2}+xK(x);
	\end{equation}
	\begin{equation}\label{intK2}
	\int_{x}^\infty K(s)\;ds= \frac{e^{-x}}{2}-xK(x);
	\end{equation}
	\begin{equation}\label{intK3}
	\int_{-x}^\infty sK(s)\;ds=\int_{x}^\infty sK(s)\;ds= \frac{xe^{-x}}{4}+\frac{e^{-x}}{4}-\frac{x^2}{2}K(x);
	\end{equation}
	\begin{proof}
	The proof relies on basic integral computations. We have to compute several integrals changing the order of integration applying Tonelli's Theorem and integrating by parts. We assume $ x>0 $. We prove only \eqref{intK1}, since all other formulas can be obtained in a similar way.
	\begin{equation*}\label{intK1calc}
	\begin{split}
	\int_{-x}^\infty K(s)\;ds=& \frac{1}{2}\int_{-x}^0 \int_{|s|}^\infty\frac{e^{-t}}{t}\;dt\;ds+\frac{1}{2}\intup\int_s^\infty\frac{e^{-t}}{t}\;dt\;ds\\
	=&\frac{1}{2}\int_0^x\frac{e^{-t}}{t} \int_0^t\;ds\;dt+\frac{1}{2}\int_x^\infty\frac{e^{-t}}{t} \int_0^x\;ds\;dt+\frac{1}{2}\\
	=&1-\frac{e^{-x}}{2}+xK(x).
	\end{split}
	\end{equation*}
	\end{proof}
\end{prop}

	\section{The boundary condition for the limit problem}
	We now start with the boundary layer analysis. This boundary layer problem, know in the literature as Milne problem, was studied with different approaches, e.g.  \cite{Golse3,golse5,Davison,golse6,Hopf}. We will present another proof of the boundary layer analysis for the equation \eqref{bvpgreyboundary} of the temperature. The proof uses a combination of comparison arguments and Fourier analysis. In addition, instead of considering the intensity of radiation the analysis is made directly for the temperature.\\
	
	Our aim is now to solve equation \eqref{bvpgreyboundary}. Indeed, according to the method of matched asymptotic expansions we expect the boundary condition for the limit problem to be the limit of $ u $ as $ y\to\infty $ for every point $ x\in\partial\Omega $. In order to simplify the notation we call $ \LL\left(u\right)(x):=u(x)-\int_0^\infty dy\; K\left(x-y\right)u(y)$ and $ \oLL \left(u\right)(x):=u(x)-\int_{-\infty}^\infty dy\; K\left(x-y\right)u(y) $.
	\subsection{The homogeneous equation}
	We start with the study of the homogeneous equation, i.e. \eqref{bvpgreyboundary} with $ G(x)\equiv 0 $. We will show using maximum principle that any bounded solution is the trivial solution $ u\equiv 0 $. We will use the following version of the maximum principle for the non-local operator $ 	\LL $.
	\begin{lemma}\label{maxprinciple}
		Let $ \ou\in C\left([0,\infty)\right) $ with $ \lim\limits_{x\to\infty}\ou(x)\in[0,\infty] $ be a supersolution of \eqref{bvpgreyboundary}, i.e. 
		\begin{equation*}\label{supersolution}
		\begin{cases}
		\ou (x)-\int_0^\infty dy\;K\left(x-y\right)\ou(y)\geq 0 & x>0\\\ou(x)=0&x<0
		\end{cases}
		\end{equation*}
		Then $ u\geq 0 $ for all $ x\geq 0 $.
		\begin{proof}Let us assume the contrary, i.e. that there exists some $ x\in[0,\infty] $ such that $ \ou(x)<0 $. By assumption $ x\in[0,\infty) $.
		Since $ \ou $ is continuous in $ [0,\infty) $ and it has non-negative limit at infinity which is bounded or infinity, $ u $ attains its global minimum in $ [0,\infty) $, i.e. there exists some $ x_0\in[0,\infty) $ such that $ \ou(x_0)=\inf_{x\in[0,\infty)}\ou(x)<0 $. Since $ \ou $ is a super solution we can calculate
		\begin{equation*}\label{maxprinciple1}
		\begin{split}
		0\leq& \LL(\ou)(x_0)=\ou(x_0)-\intup dy\; K\left(x_0-y\right)\ou(y)\\
		=& \intbig dy\;K\left(x_0-y\right)\ou(x_0)-\intup dy\; K\left(x_0-y\right)\ou(y)\\
		=& \intdown dy\;K\left(x_0-y\right)\ou(x_0)+\intup dy\; K\left(x_0-y\right)\left(\ou(x_0)-\ou(y)\right)\\
		<&0,
		\end{split}
		\end{equation*}
		where we used the positivity of $ K\left(x_0-y\right) $, the fact that the integral of the kernel $ K $ is $ 1 $ and the fact that $ \ou(x_0) $ is the minimum of $ \ou $ and it is strictly negative. This leads to a contradiction and thus we conclude the proof.
		\end{proof}
	\end{lemma}
With the maximum principle we can now show the following theorem on the triviality of the solution to the homogeneous equation.
	\begin{theorem}\label{homogeneoussolution}
	Assume $ u $ is a bounded solution to
	\begin{equation}\label{homogeneous}
\oLL(u)(x)=0
	\end{equation}	with $ u(x)\equiv 0 $ for $ x<0 $.
	Then $ u=0 $ for almost every $ x\in\RR $.
	\begin{proof}
	We will construct a supersolution $ \ou $ which converges to infinity and we will apply Lemma \ref{maxprinciple} to the supersolutions $ \ou-u $ and $ u+\ou $. First of all we see that for $ x>0 $ the bounded solution $ u $ is continuous, indeed $ u(x)=K*u (x) $. Since $ K\in \LpR{1} $ and $ u\in\LpR{\infty} $ then the convolution is a continuous bounded function. Moreover we can extend continuously $ u $ in $ 0 $. Indeed, we define
	\begin{equation*}
	u(0)=\lim\limits_{x\to 0}\left[G(x)+\intup dy\;K\left(x-y\right)u(y)\right].
	\end{equation*}
	This limit exists because $ G $ is continuous in $ [0,\infty) $ and for the integral term we can apply the generalized dominated convergence theorem using that the sequence $ K\left(x-y\right)\to K\left(y\right) $ as $ x\to0 $ pointwise and in $ \LpR{1} $. 
	
	We consider now the function
	\begin{equation*}\label{supersol}
	\ou(x)=	\begin{cases} 1+x & x\geq 0\\0&x<0	\end{cases}
	\end{equation*}
	$ \ou $ is a supersolution. It is indeed possible to calculate $ \LL(\ou)(x) $. Let $ x\geq 0 $. Then $ \LL(\ou)=\LL(Id)+\LL(1) $. By a simple calculation we get on the one hand
	\begin{equation*}\label{Lx}
	\begin{split}
	\LL(Id)(x)=& x-\intup dy\; K\left(x-y\right)y=x-\int_{-x}^\infty dy\;K\left(y\right)(x+y)=\frac{x}{4}e^{-x}-\frac{e^{-x}}{4}-\frac{x^2}{2}K(x)	\end{split}
	\end{equation*}
	and on the other hand
	\begin{equation*}\label{L1}
	\begin{split}
	\LL(1)(x)=& 1-\intup dy\; K\left(x-y\right)=1-\int_{-x}^\infty dy\;K\left(y\right)=\frac{e^{-x}}{2}-xK(x).	\end{split}
	\end{equation*}
	Therefore we want to show that the function $ f(x):=\LL(\ou)(x)= \frac{e^{-x}}{4}(1+x)-\frac{x}{2}K(x)(2+x) $ is non-negative for all $ x\geq 0 $. It is not difficult to see that $ f(0)=\frac{1}{4}>0 $ and that $ \lim\limits_{x\to\infty}f(x)=0 $. Moreover, we can consider the derivative
	\begin{equation*}\label{supersol1}
	\begin{split}
	f'(x)=& \frac{1}{2}\left(e^{-x}-K(x)(2x+2)\right)\leq\frac{1}{2}\left(e^{-x}-\frac{e^{-x}}{2}\ln\left(1+\frac{2}{x}\right)(x+1)\right)
	\leq 0.
	\end{split}
	\end{equation*}
	The first inequality is given by the estimate of Proposition \ref{propK} and the second one is due to the well-know estimate $ \ln(1+x)\geq \frac{2x}{2+x} $.
	The non-positivity of the derivative implies that $ f $ is monotonously decreasing, and therefore $ \LL(\ou)(x)=f(x)\geq 0 $ for all $ x\geq 0 $.
	
	Let now $ \eps>0 $ arbitrary. We know that $ u $ is bounded and $ \ou $ converges to infinity, moreover both $ u $ and $ \ou $ are continuous in $ [0,\infty) $. 
	Also $ u $ is a homogeneous solution of \eqref{bvpgreyboundary} and the operator $ \LL $ is linear. Therefore we can apply Lemma \ref{maxprinciple} to the supersolutions $ \eps\ou-u $ and $ u+\eps\ou $ and get that the $ \inf_{x\in[0,\infty)}\left[\eps\ou(x)-u(x)\right]\geq 0 $ and $ \inf_{x\in[0,\infty)}\left[\eps\ou(x)+u(x)\right]\geq0 $. This implies that for any $ x\in\RR $ the following holds
	\begin{equation*}
	-\eps\ou(x)\leq u(x)\leq\eps\ou(x)
	\end{equation*}
	Since $ \eps $ was arbitrary we conclude $ u(x)=0 $  for all $ x\in\RR $.
	\end{proof}
	\end{theorem}
\subsection{Well-posedness theory for the inhomogeneous equation}
We can now move to the well-posedness theory for the inhomogeneous equation, for which the next theorem is the main result.
	\begin{theorem}\label{existence} Let  $ H:\RR_+\to\RR_+ $ be a continuous function bounded by an exponential function, i.e. $ |H(x)|\leq Ce^{-Ax}\rchi_{\{x>0\}} $ for $ C,A>0 $. Then there exists a unique bounded solution to the equation
	\begin{equation}\label{inhomogeneousH}
	\begin{cases}
	u (x)-\int_0^\infty dy\;K\left(x-y\right)u(y)= H(x) & x>0,\\u(x)=0&x<0.
	\end{cases}
	\end{equation}
	Moreover, $ u $ is continuous on $ (0, \infty) $.
	\begin{proof}The assumption on the exponential decay of $ H $ yields $ H\in L^1\left(\RR\right)\cap L^2\left(\RR\right)\cap L^\infty\left(\RR\right) $.
	In order to find a bounded solution for \eqref{inhomogeneousH} we will follow several steps. We will look for functions $ \tu $ and $ v $ solutions of the following equations
	\begin{equation*}\label{eqtu}
	\tu(x)-\intbig K(x-y)\tu(y)\;dy= \oH(x):= H(x)-H(-x)\;\;\;\;\;\;\;\;x\in\RR
	\end{equation*}
	and
	\begin{equation}\label{eqv}
	\begin{cases}
	v(x)-\intbig K(x-y) v(v)\;dy=0, & x>0\\ v(x)=-\tu(x)& x<0.
	\end{cases}
	\end{equation}
	Then $ u=\tu+v $ will be the desired solution.\\
	
	\textbf{Step 1: }Construction of $ \tu $.\\
	We can construct the solution $ \tu $ via Fourier method. First of all we notice that any affine function is a solution to the homogeneous equation in the whole space $ \RR $. This is because $ \intbig K(x)\;dx=1 $ and $ \intbig xK(x)\;dx=0 $. Since by assumption $ H\in\LpR{2} $ also $ \oH\in\LpR{2} $. We define for an integrable function $ f $ the kth-moment as $ m_k\left(f\right)=\intbig x^kf(x)\;dx $ assuming it exists. Then clearly by construction $ m_0\left(\oH\right)=0 $ while $ m_1\left(\oH\right)=\frac{2C}{A}>0 $. Moreover, since $ \oH $ has exponential decay, all moments $ m_k\left(\oH\right)<\infty $ are bounded.
	
	We define also the function $ F(x)=\oLL\left(\frac{3}{2}\sgn\right)(x) $. It can be compute that \begin{equation*}
	F(x)= \frac{3}{2}\left(\sgn(x)-\intbig K(x-y)\sgn(y)\;dy\right)=\frac{3}{2}\left(\sgn(x)-\int_{-x}^x K(y)\;dy\right).
	\end{equation*} It is not difficult to see that $ F(0)=0 $, $ \lim\limits_{|x|\to\infty}F(x)=0 $ and that $ F $ is a stepwise continuous function with the discontinuity in $ 0 $. Therefore $ F(x) $ is bounded. We proceed with the construction of $ \tu $. We can write it as $ \tu=u^{(1)}+u^{(2)}+a+bx $, where $ u^{(1)}(x)= m_1\left(\oH\right)\frac{3}{2}\sgn(x) $ solves the equation
	\begin{equation*}\label{u1}
	\LL\left(u^{(1)}\right)(x)=m_1\left(\oH\right)F(x) \;\;\;\;\;\;\;\;\;x\in\RR
	\end{equation*}
	and $ u^{(2)} $ solves
	\begin{equation}\label{u2}
	\LL\left(u^{(2)}\right)(x)=\oH(x)-m_1\left(\oH\right)F(x) \;\;\;\;\;\;\;\;\;x\in\RR.
	\end{equation}
	applying now the Fourier transform to the equation \eqref{u2}, recalling the convolution rule and the Fourier transforms of the kernel $ K $ and of the $ \sgn $ function we get first in distributional sense
	\begin{equation}\label{u21}
	\hat{u}^{(2)}(s)\left(\frac{s-\arctan(s)}{s}\right)= \mathcal{F}(\oH)(s)+\frac{3m_1\left(\oH\right)}{\sqrt{2\pi}}\frac{i}{s}\frac{s-\arctan(s)}{s}.	
	\end{equation}
	The Fourier transform of $ \oH $ is in $ C^\infty $, since $ \oH $ has exponential decay and therefore it has all kth-moment finite. Therefore there exists a function $ \tilde{H} $ such that $ \tilde{H}(0)=\tilde{H}'(0)=0 $ such that $ \mathcal{F}(\oH)(s)=-\frac{i}{\sqrt{2\pi}}m_1\left(\oH\right)s+\tilde{H}(s) $, since $ m_0\left(\oH\right)=0 $ and by definition $ \mathcal{F}(\oH)'(s)\big|_{s=0}=-\frac{i}{\sqrt{2\pi}}m_1\left(\oH\right) $. We can therefore find first formally $ u^{(2)} $ analyzing its Fourier transform
	\begin{equation}\label{u22}
	\begin{split}
		\hat{u}^{(2)}(s)=& \frac{s}{s-\arctan(s)}\mathcal{F}(\oH)(s)+\frac{3m_1\left(\oH\right)}{\sqrt{2\pi}}\frac{i}{s}\\
		=& -\frac{is^2}{s-\arctan(s)}\frac{m_1\left(\oH\right)}{\sqrt{2\pi}}+\frac{3m_1\left(\oH\right)}{\sqrt{2\pi}}\frac{i}{s}+\tilde{H}(s)\frac{s}{s-\arctan(s)}\\
		=&\mathbb{H}(s).
	\end{split}
	\end{equation}
	It is important to notice that $ \lim\limits_{s\to0}\frac{s^2}{s-\arctan(s)}-\frac{3}{s}=0 $, since $ \frac{s}{s-\arctan(s)}= \frac{3}{s^2}+\frac{9}{5}+O(s^2) $ near zero. Using L'Hôpital rule we see also that $ \lim\limits_{s\to0}\tilde{H}(s)\frac{s}{s-\arctan(s)} $ is finite. On the other hand $ \frac{s}{s-\arctan(s)} $ is bounded for $ |s|>1 $. Since $ \mathcal{F}(\oH)(s) $ and $ \frac{1}{s} $ are both square integrable functions and since $ \mathbb{H} $ is bounded near $ 0 $ we conclude that $ \mathbb{H}\in\LpR{2} $. Therefore also the in \eqref{u22} defined $ \hat{u}^{(2)} $ is square integrable. We can hence invert it 
	\begin{equation*}\label{u23}
	u^{(2)}(x):=\mathcal{F}^{-1}\left(\mathbb{H}\right)(x)\in\LpR{2}.
	\end{equation*}
	Since this function solves \eqref{u21} not only in distributional sense but also pointwise almost everywhere, we can conclude rigorously that indeed the function in \eqref{u22} is the desired $ u^{(2)} $ solving \eqref{u2}.
	Moreover, $ u^{(2)}= K*u^{(2)}+\oH-F$ and since both $ K $ and $ u^{(2)} $ itself are square integrable and both $ H $ and $ F $ are bounded, then also $ u^{(2)} $ is bounded. We can conclude this step therefore defining 
	\begin{equation}\label{tu}
	\tu(x)=\frac{3}{2}m_1\left(\oH\right)\sgn(x)+a+bx+u^{(2)}(x).	
	\end{equation}
	\textbf{Step 2: }Construction of $ v $.\\
	
	We recall that the equation $ v $ shall solve \eqref{eqv}
	As we found out in the first step, $ \tu=\frac{3}{2}m_1\left(\oH\right)\sgn(x)+a+bx+u^{(2)}(x)  $. As we already pointed out, affine solutions are always solution of the homogeneous equation in the whole space $ \RR $. Therefore, we shall look for a function of the form 
	\begin{equation}\label{v1}
	v(x)= \frac{3}{2}m_1\left(\oH\right)-a-bx+v^{(2)}(x)
	\end{equation}
	where $ v^{(2)} $ solves similarly as above
	\begin{equation}\label{eqv2}
	\begin{cases}
	v^{(2)}(x)-\intbig K(x-y) v(v)\;dy=0 & x>0,\\ v^{(2)}(x)=-u^{(2)}(x)& x<0.
	\end{cases}
	\end{equation}
	We proceed now iteratively constructing the desired solution. We call $ B>0 $ the constant such that $ \left\Arrowvert u^{(2)}\right\Arrowvert_\infty\leq B $ and we define $ \ov=B $ and $ \uv=-B $. Inductively we define $ v_0:=\uv $ and for $ k\geq1 $ we set
	\begin{equation*}\label{vk}
	v_k(x)=	\begin{cases} -u^{(2)}(x)&x<0,\\\intbig K\left(x-y\right)v_{k-1}(y)\;dy&x>0.\\	
	\end{cases}
	\end{equation*}
	We claim that $ \uv=v_0\leq v_1\leq v_2\leq ...\leq v_k\leq v_{k+1}\leq ... $ and that $ v_k\leq \uv $ for all $ k\in\mathbb{N} $. Clearly for $ k=0 $ both statements hold. On the one hand  since $ \intbig K\left(x-y\right)v_0(y)=-B $ we see that 
		\begin{equation*}\label{v0}
	v_1(x)-v_0(x)=	\begin{cases} -u^{(2)}(x)+B\geq 0&x<0,\\0&x>0,\\	
	\end{cases}
	\end{equation*}
	on the other hand per definition we have $ \ov-\uv=2B\geq 0 $. We see also that $ \ov-v_1\geq 0 $, indeed
	\begin{equation*}\label{v11}
	\ov(x)-v_1(x)=	\begin{cases} B+u^{(2)}(x)\geq 0&x<0,\\2B&x>0.\\	
	\end{cases}
	\end{equation*}
	We now prove inductively that $ v_k\geq v_{k-1} $ and $ \ov\geq v_k $. Hence, we assume that these inequalities are satisfied for $ k $ and we prove them for $ k+1 $. Indeed this just follows from the identities
	\begin{equation*}\label{v12}
	v_{k+1}(x)-v_k(x)=	\begin{cases} 0&x<0,\\\intbig K\left(x-y\right)\left(v_k(y)-v_{k-1}(y)\right)\geq 0&x>0,\\	
	\end{cases}
	\end{equation*}
		\begin{equation*}\label{v13}
	\ov(x)-v_{k+1}(x)=	\begin{cases} B+u^{(2)}(x)\geq 0&x<0,\\\intbig K\left(x-y\right)\left(B-v_k(y)\right)\geq 0&x>0,\\	
	\end{cases}
	\end{equation*}
	where we used again that the integral in the whole line of the kernel $ K $ is $ 1 $. Therefore the sequence $ v_k(x) $ is increasing and bounded. This means that there exists a pointwise limit. By the dominated convergence theorem and by construction this will be also the desired solution of \eqref{eqv2}, i.e.
	\begin{equation*}\label{v2}
	v^{(2)}(x):=\lim\limits_{k\to\infty}v_k(x)
	\end{equation*}
	solves the equation \eqref{eqv2} and it is by construction bounded.\\
	
	\textbf{Step 3: }Properties of $ u $.\\
	
	Now we are ready to write down the whole solution. As we remarked at the beginning $ u= \tu+v $, where $ \tu $ solves as in Step 1 \eqref{eqtu} and $ v $ solves as in Step 2 \eqref{eqv}. Therefore by \eqref{v1} and by \eqref{tu}
	\begin{equation*}\label{defu}
	u(x)=\begin{cases}
	6 m_{1}\left(H\right)+u^{(2)}(x)+v^{(2)}(x)&x>0,\\ 0&x<0,
	\end{cases}
	\end{equation*}
	solves the initial problem \eqref{inhomogeneousH} and it is by construction bounded. Moreover, since $ K $ is integrable and $ H $ is continuous in $ [0,\infty) $ also $ u=K*u+H $ is continuous in $ [0,\infty) $.\\
	
	\textbf{Step 4: }Uniqueness.\\
	
	Let us assume that $ u_1 $ and $ u_2 $ are two bounded solution to the problem \eqref{inhomogeneousH}. Then $ u_1-u_2 $ will be a bounded continuous solution to the homogeneous problem \eqref{homogeneous}. Therefore by Theorem \ref{homogeneoussolution} $ u_1-u_2=0 $. Hence, there exists a unique bounded solution $ u $ to the inhomogeneous problem \eqref{inhomogeneousH}.	
	 This concludes the proof.
	\end{proof}
	\end{theorem}
\begin{corollary}\label{existenceG}
	Let  $ p\in\bnd $ and $ G_p(x) $ as defined in \eqref{bvpgreyboundary1}. Let $ g_\nu(n)\geq 0 $ and assume $ \intnu\;g_\nu(n)\in L^\infty\left(\Ss\right) $. Then there exists a unique bounded solution to the equation
	\begin{equation}\label{inhomogeneous}
	\begin{cases}
	u (x)-\int_0^\infty dy\;K\left(x-y\right)u(y)= G_p(x) & x>0,\\u(x)=0&x<0.
	\end{cases}
	\end{equation}
	Moreover, $ u $ is continuous on $ (0, \infty) $.
	\begin{proof}
	By assumption $ G_p $ is continuous for $ x>0 $ and $ |G_p(x)|\leq \Arrowvert g\Arrowvert_1e^{-y}\rchi_{\{x>0\}} $. Hence we can apply Theorem \ref{existence}.
	\end{proof}
\end{corollary}
It is also possible to show, that the bounded solution $ u $ is non-negative
\begin{lemma}\label{nonnegativity}
Let $ u $ be the unique bounded solution to \eqref{inhomogeneous}. Then $ u(x)\geq 0 $ for all $ x\in\RR $.
\begin{proof}
	The proof is very similar to the proof of Theorem \ref{homogeneoussolution}. We consider the supersolution 
	\begin{equation*}
	\ou(x)=	\begin{cases}1+x&x\geq 0\\0&x<0
	\end{cases}.
	\end{equation*}
	As we have seen before, $ u=K*u+G $ is continuous in $ [0,\infty) $. Moreover, since $ G>0 $ as $ x\geq 0 $, $ u $ is a supersolution too. Let now $ \eps>0 $ be arbitrary. Let us consider the supersolution $ \eps\ou+u $. This is continuous in $ [0,\infty) $ and since $ u $ is bounded it converges to infinity as $ x\to\infty $. Therefore Lemma \ref{maxprinciple} implies that there exists no $ x_0\in[0,\infty) $ such that
	\begin{equation*}
	\inf_{x\in[0,\infty)}\left(\eps\ou(x)+u(x)\right)=\eps\ou(x_0)+u(x_0)<0.
	\end{equation*}
	Hence $ u\geq -\eps\ou $ and since $ \eps>0 $ was arbitrary we conclude $ u\geq 0 $.
\end{proof}
\end{lemma}
\begin{remark}
	Theorem \ref{existence} can be proved also using the Wiener-Hopf solution formula for the problem \eqref{bvpgreyboundary} as given in \cite{Hochstadt}. It is true that in this way one obtains an explicit formula, which not only assures the well-posedness of the planar problem we are studying but also directly shows the existence of a limit for the solution $ u $ when $ x\to\infty $. However, the Wiener-Hopf method produces a complicate formula which requires a careful analysis with complex variables in order to be understood. We have preferred to use this soft method approach which in particular allows us to prove some relevant properties of the solution, such as the positivity.
\end{remark}

\subsection{Asymptotic behavior of the bounded solution of the inhomogeneous equation}
We were able to show that the equation for the boundary value in the Grey approximation has a unique bounded solution which is positive whenever $ G>0 $. As we anticipated at the beginning of this section, we would like to study the limit as $ x\to\infty $ of the solution $ u(x) $. We will show, that such limit exists and is uniquely characterized by $ g_\nu(n) $ and $ N $. To this end we first prove that the function $ u $ is uniformly continuous.
\begin{lemma}\label{unifcont}
	Let $ u $ be the unique bounded solution to the problem \eqref{bvpgreyboundary}. The $ u $ is uniform continuous on $ [0,\infty) $ and it satisfies for $ x,y\in[0,\infty) $
	\begin{equation}\label{ux-uy}
	\begin{split}
	\left|u(x)-u(y)\right|\leq
	& \left|G(x)-G(y)\right|\\+&\left\Arrowvert u\right\Arrowvert_{\infty}\left[\frac{\left|e^{-x}-e^{-y}\right|}{2}+2\left(1-e^{\frac{\left| x-y\right|}{2}}\right)+4\left|\frac{y-x}{2}\right|K\left(\frac{y-x}{2}\right)+\left|xK(x)-yK(y)\right|\right].
	\end{split}
	\end{equation}
	\begin{proof}
	This is a consequence of the uniform continuity of $ G $ and $ xK(x) $. Clearly, since $ u $ solves the problem \eqref{inhomogeneous}, we have the estimate
	\begin{equation}\label{unifcont1}
	\left|u(x)-u(y)\right|\leq \left|G(x)-G(y)\right|+\intup\left|K\left(\eta-x\right)-K\left(\eta-y\right)\right|u(\eta)\;d\eta.
	\end{equation}
	Since $ G $ is continuous on $ [0,1] $, and therefore uniformly continuous on $ [0,1] $ and since $ G $ is Lipschitz continuous in $ [1,\infty) $, $ G$ is uniform continuous in $ [0,\infty) $. The latter affirmation is true, since 
	\begin{equation*}\label{unifcont2}
	\sup_{x\geq 1}\left|G'(x)\right|\leq\intnu\intnN \;g_\nu(n)\frac{e^{-\frac{1}{|n\cdot N|}}}{|n\cdot N|}<\infty,
	\end{equation*}
	where the finiteness is due to the fact that $ \lim\limits_{|n\cdot N|\to0}\frac{e^{-\frac{1}{|n\cdot N|}}}{|n\cdot N|}=0 $.\\
	
	For the integral term in \eqref{unifcont1} we assume that $ x<y $. Then we can calculate using the fact that for positive arguments the kernel $ K $ is decreasing
	\begin{equation*}\label{unifcont3}
	\begin{split}
	\intup&\left|K\left(\eta-x\right)-K\left(\eta-y\right)\right|u(\eta)\;d\eta\\
	=&\int_{0}^{\frac{x+y}{2}}\left(K\left(\eta-x\right)-K\left(\eta-y\right)\right)u(\eta)\;d\eta+\int_{\frac{x+y}{2}}^{\infty}\left(K\left(\eta-y\right)-K\left(\eta-x\right)\right)u(\eta)\;d\eta\\
	\leq&\left\Arrowvert u\right\Arrowvert_{\infty} \left[\int_{0}^{\frac{x+y}{2}}\left(K\left(\eta-x\right)-K\left(\eta-y\right)\right)\;d\eta+\int_{\frac{x+y}{2}}^{\infty}\left(K\left(\eta-y\right)-K\left(\eta-x\right)\right)\;d\eta\right]\\
	\end{split}
	\end{equation*}
	We can calculate explicitly the last two integrals using the result of Proposition \ref{propK1}, indeed by a change of variable
	\begin{equation*}\label{unifcont4}
	\begin{split}
	\intup&\left|K\left(\eta-x\right)-K\left(\eta-y\right)\right|u(\eta)\;d\eta\\\leq&\left\Arrowvert u\right\Arrowvert_{\infty} \left[\int_{-x}^{\frac{y-x}{2}}K\left(\eta\right)\;d\eta-\int_{-y}^{\frac{x-y}{2}}K\left(\eta\right)\;d\eta+\int_{\frac{x-y}{2}}^{\infty}K\left(\eta\right)\;d\eta-\int_{\frac{y-x}{2}}^{\infty}K\left(\eta\right)\;d\eta\right]\\
	=&\left\Arrowvert u\right\Arrowvert_{\infty}\left[\frac{e^{-y}-e^{-x}}{2}+2\left(1-e^{\frac{ x-y}{2}}\right)+4\frac{y-x}{2}K\left(\frac{y-x}{2}\right)+xK(x)-yK(y)\right].
	\end{split}
	\end{equation*}
	Recalling that $ x<y $ we get the estimate \eqref{ux-uy}. From the well-known estimates $ \left|e^{-x}-e^{-y}\right|\leq \left|x-y\right| $ and $ \left|1-e^{\frac{ x-y}{2}}\right|\leq\frac{\left|x-y\right|}{2}  $ we see that we shall only consider the function $ f(x)=xK(x) $. Since $ f(0)=0 $ and $ f $ is continuous, $ f $ is uniformly continuous on $ [0,1] $, on the other hand $ f $ is Lipschitz continuous on $ [1,\infty] $. This is because 
	\begin{equation*}\label{unifcont5}
	\sup\limits_{x\geq 1} \left|f'(x)\right|=\sup\limits_{x\geq 1} \left|K(x)-e^{-x}\right|\leq \frac{1}{e}+K(1)<\infty.
	\end{equation*}
	Therefore $ f $ is uniform continuous on $ [0,\infty) $.
	By the continuity of $ f $ in $ 0 $ we also now that given an $ \eps>0 $ there exists some $ \delta $ such that $ \frac{y-x}{2}K\left(\frac{y-x}{2}\right)<\eps $ for all $ \left|x-y\right|<\delta $. Hence, we conclude that $ u $ is uniform continuous. 
	\end{proof}
\end{lemma}
We want now to show that the limit $ \lim\limits_{y\to\infty}u(y) $ exists. To this end we proceed again using Fourier methods. 
\begin{theorem}\label{thmgrey}
Let $ u $ be the unique bounded solution to the problem \eqref{inhomogeneous}. Then $ \lim\limits_{x\to\infty}u(x) $ exists and it is uniquely determined by $ G $ and $ u $ itself. Moreover, the limit is positive if \linebreak$ \left\{n\in\Ss:n\cdot N<0\text{ and }\int_0^\infty d\nu \og_\nu(n)\not\equiv 0\right\} $ is not a zero measure set.
\begin{proof}
Since $ u $ is the unique bounded solution, $ u $ solves for all $ x\in\RR $
\begin{equation}\label{bigthm 1}
u(x)-\int_{-\infty}^\infty K(y-x)u(y)\;dy=G(x)\;\rchi_{\{x>0\}}-\int_{0}^{\infty}K(y-x)u(y)\;dy\;\rchi_{\{x<0\}}\equiv W(x).
\end{equation}
Indeed, this is equivalent to \eqref{inhomogeneous}. This can be seen easily, since $ u $ solves for $ x<0 $ 
\begin{equation*}\label{bigthm 2}
u(x)-\int_{-\infty}^0 K(y-x)u(y)\;dy=0
\end{equation*}
and since $ u=0 $ for $ x<0 $ is a possible solution, by uniqueness, this is the only possible solution. It is not only true that $ W\in L^1\left(\RR\right)\cap L^2\left(\RR\right) $ but also that $ W $ has all moments bounded. This follows from the similar property of $ G $ (cf. Step 1 in Theorem \ref{existence}) as well as from the inequality $ 0\leq \int_{0}^{\infty}K(y-x)u(y)\;dy\rchi_{\{x<0\}}\leq \Arrowvert u\Arrowvert_{\infty}\; \rchi_{\{x<0\}} \left(\frac{e^{-|x|}}{2}-|x|K(x)\right) $. Notice that $ |x|K(x)\leq \frac{e^-|x|}{2} $. Hence, finite moments and Riemann-Lebesgue Theorem imply that $ W $ has a Fourier transform $ \hat{W}\in C_0\left(\RR\right)\cap C^\infty\left(\RR\right)\cap L^2\left(\RR\right) $. Moreover, looking at the left hand side of \eqref{bigthm 1} we recall as in \cite{reedsimon} that in distributional sense for all $ \phi\in\mathcal{S}\left(\RR\right) $
\begin{equation*}\label{bigthm 3}
\langle\hat{u}-\mathcal{F}\left(u*K\right),\phi \rangle:=\langle u-u*K, \hat{\phi}\rangle=\langle u, \mathcal{F}\left((1-\sqrt{2\pi}\hat{K})\phi\right)\rangle,
\end{equation*}
where the last equality is due to an elementary calculation involving the convolution and we define $ \langle f,g\rangle =\int_{\RR}f(x)g(x)\;dx $. We recall also that $ 1-\sqrt{2\pi}\hat{K}(\xi)=\frac{\xi-\arctan(\xi)}{\xi}:=F(\xi) $. Hence, for all $ \phi\in\mathcal{S}\left(\RR\right) $ we have
\begin{equation}\label{bigthm 4}
\langle u, \mathcal{F}(\phi F)\rangle=\langle \hat{W}, \phi\rangle.
\end{equation}
Now we consider for $ \eps>0 $ the sequence of standard mollifiers $ \phi_\eps(\xi):=\frac{1}{\eps}\phi\left(\frac{\xi}{\eps}\right)\in C_c^\infty\left(\RR\right)\subset \mathcal{S}\left(\RR\right)  $ such that in distributional sense $ \phi_\eps\rightharpoonup \delta $. The smoothness of $ \hat{W} $ implies $ \langle \hat{W}, \phi\rangle\to\hat{W}(0) $ as $ \eps\to 0 $. It is our first aim to show that $ \hat{W}(0) $ is zero. To this end we study the left hand side of \eqref{bigthm 4}. We calculate
\begin{equation*}\label{bigthm 5}
\begin{split}
\langle u,& \mathcal{F}(\phi_\eps F)\rangle=\frac{1}{\sqrt{2\pi}}\int_0^\infty dx\;u(x)\int_{\RR}d\xi\; \phi_\eps(\xi)F(\xi)e^{-i\xi x}\\
=&\frac{1}{\sqrt{2\pi}}\int_0^1 dx\;u(x)\int_{\RR}d\xi\; \phi_\eps(\xi)F(\xi)e^{-i\xi x}-\frac{1}{\sqrt{2\pi}}\int_1^\infty dx\;\frac{u(x)}{x^2}\int_{\RR}d\xi\; \left(\phi_\eps(\xi)F(\xi)\right)''e^{-i\xi x},
\end{split}
\end{equation*}
where for the last equality we integrated twice by parts in $ \xi $. By a change of coordinates and the dominated convergence theorem, since $ F(0)=0 $ and $ |F(\eps\xi)\phi(\xi)|\leq |\phi(\xi)| $ we see for the first term as $ \eps\to 0 $
\begin{equation*}\label{bigthm 6}
\left|\frac{1}{\sqrt{2\pi}}\int_0^1 dx\;u(x)\int_{\RR}d\xi\; \phi_\eps(\xi)F(\xi)e^{-i\xi x}\right|\leq \frac{1}{\sqrt{2\pi}}\int_0^1 dx\;u(x)\int_{\RR}d\xi\; |F(\eps\xi)\phi(\xi)|\to 0.
\end{equation*}
Thus, we shall consider only the second term. We use the following well-known estimate $ \left|e^{-i\xi x}-1\right|\leq 2|\xi|^\delta|x|^\delta $ for $ 0<\delta <1 $ and $ x\in\RR $. Then using $ \int_\RR \left(\phi_\eps F\right)''=0 $
\begin{equation*}\label{bigthm 7}
\begin{split}
\frac{1}{\sqrt{2\pi}}\int_1^\infty dx\;\frac{u(x)}{x^2}\int_{\RR}d\xi\; \left(\phi_\eps(\xi)F(\xi)\right)''e^{-i\xi x}=&\frac{1}{\sqrt{2\pi}}\int_1^\infty dx\;\frac{u(x)}{x^2}\int_{\RR}d\xi\; \left(\phi_\eps(\xi)F(\xi)\right)''\left(e^{-i\xi x}-1\right),
\end{split}
\end{equation*}
and hence
\begin{equation*}\label{bigthm 8}
\begin{split}
\left|\frac{1}{\sqrt{2\pi}}\int_1^\infty dx\;\frac{u(x)}{x^2}\int_{\RR}d\xi\; \left(\phi_\eps(\xi)F(\xi)\right)''e^{-i\xi x}\right|\leq \frac{1}{\sqrt{2\pi}}\int_1^\infty dx\;\frac{u(x)}{x^{2-\delta}}\int_{\RR}d\xi\; \left|\left(\phi_\eps(\xi)F(\xi)\right)''\right|2|\xi|^\delta.
\end{split}
\end{equation*}
Now we notice that $ \int_1^\infty dx\;\frac{u(x)}{x^{2-\delta}}<\infty $ and also we see that $ F(\xi)\simeq \frac{\xi^2}{3} $ as $ x\to 0 $, similarly as $ \xi\to 0 $ also $ F'(\xi)\simeq \frac{2}{3}\xi $ and $ F''(\xi)\simeq \frac{2}{3} $. Hence, with a change of variables we see that 
\begin{equation*}\label{bigthm 9}
\begin{split}
\int_{\RR}d\xi\;& \left|\left(\phi_\eps(\xi)F(\xi)\right)''\right||\xi|^\delta\\\leq&\int_{\RR}d\xi\;\left[|\phi(\xi)||F''(\eps\xi)|\eps^\delta|\xi|^\delta+2|\phi'(\xi)|\frac{|F'(\eps\xi)||\xi|^\delta}{\eps^{1-\delta}}+|\phi''(\xi)|\frac{|F(\eps\xi)||\xi|^\delta}{\eps^{2-\delta}}\right].
\end{split}
\end{equation*}
With the consideration above about $ F $ and since $ \phi\in C_c^\infty\left(\RR\right) $ we see that there exists a constant $ C=2\Arrowvert \phi\Arrowvert_{C_c^\infty\left(\RR\right)} \left(\max\limits_{\text{supp}\phi}|\xi|\right)^{2+\delta}<\infty$ such that 
\begin{equation*}
	|\phi(\xi)||F''(\eps\xi)|\eps^\delta|\xi|^\delta+|\phi'(\xi)|\frac{|F'(\eps\xi)||\xi|^\delta}{\eps^{1-\delta}}+|\phi''(\xi)|\frac{|F(\eps\xi)||\xi|^\delta}{\eps^{2-\delta}}\leq C\eps^\delta
	\end{equation*}  for any $ \xi\in\text{supp}(\phi) $. Thus, again with the dominated convergence theorem we conclude
\begin{equation*}\label{bigthm 10}
\left|\frac{1}{\sqrt{2\pi}}\int_1^\infty dx\;\frac{u(x)}{x^2}\int_{\RR}d\xi\; \left(\phi_\eps(\xi)F(\xi)\right)''e^{-i\xi x}\right|\to 0,
\end{equation*}
which implies the first claim, namely $ \hat{W}(0)=0 $. \\

As next step we prove that the limit $ \lim\limits_{x\to\infty}u(x) $ exists. First of all we know that in distributional sense $ \hat{u} $ solves the equation
\begin{equation}\label{distr}
F \hat{u}\overset{\mathcal{S}'}{=} \hat{W}.
\end{equation}
Given any distributional solution $ \hat{u} $ to \eqref{distr} also $ \hat{u}+\hat{u}_h $ is a solution, where $ \hat{u}_h $ is the homogeneous solution to $ F\hat{u}_h\overset{\mathcal{S}'}{=}0 $. Let us consider the tempered distribution given by $ \hat{u}_h $ and let $ \varphi\in\mathcal{S}(\RR) $ be any testfunction with support away from zero, i.e. $ \text{supp}(\varphi)\subset\RR\setminus\{0\} $. Since $ F(\xi)=0 $ if and only $ \xi=0 $ and since it is bounded, the function $ \frac{\varphi}{F}\in\mathcal{S}\left(\RR\right) $. Hence, $ \int_\RR \hat{u}_h \varphi=0 $. This implies (see \cite{reedsimon1}) that $ \hat{u}_h\overset{\mathcal{S}'}{=}\sum\limits_{0\leq\alpha<m}c_\alpha(D^\alpha\delta) $, for $ c_\alpha $ constants and a suitable $ m\in\mathbb{N} $. Since $ c_\alpha F (D^\alpha\delta) \not\equiv 0$ for any $ \alpha\geq2 $ we conclude 
\begin{equation*}\label{distr1}
\hat{u}_h=c_0\delta+c_1\delta'
\end{equation*}
for suitable constants $ c_0,c_1 $. Using the smoothness of $ \hat{W} $ we can write $ \hat{W}(\xi)=\hat{W}'(0)\xi+ H(\xi)$ where $ \hat{W}'(0)=\frac{m_1(W)}{\sqrt{2\pi}i} $ and $ H\in C^\infty\left(\RR\right) $ with $ H(0)=H'(0)=0 $. Let us consider the behavior of $ F $
\begin{equation}\label{distr2}
F(\xi)\simeq\begin{cases}
\frac{\xi^2}{3}-\frac{\xi^4}{5}+\mathcal{O}(\xi^6)&\xi\to. 0,\\1-\frac{\pi}{2\xi}+\mathcal{O}\left(\frac{1}{\xi^2}\right)&\xi\to \infty
\end{cases}
\end{equation}
Hence, \begin{equation}\label{def f}
f(\xi):=\hat{W}(\xi)-\frac{3m_1(W)}{\sqrt{2\pi}i}\frac{F(\xi)}{\xi}\in L^2\left(\RR\right)
\end{equation} and it also satisfies
\begin{equation}\label{behavior f}
f(\xi)\simeq H''(0)\xi^2+\mathcal{O}(\xi^3) \;\;\;\;\;\;\text{ as }\xi\to0 .
\end{equation} By the boundedness of $ F $ and given its behavior as in \eqref{distr2} we conclude that the function $ \hat{h}:=\frac{f}{F}\in L^2\left(\RR\right) $, in particular $\hat{h} $ is well-defined in zero. It is easy to see that $ \hat{u} $ solves
\begin{equation}\label{distr3}
F(\xi) \hat{u}(\xi)\overset{\mathcal{S}'}{=} \frac{3m_1(W)}{\sqrt{2\pi}i}\frac{F(\xi)}{\xi}+f(\xi).
\end{equation}
Therefore, since $ \hat{h}\in L^2(\RR) $ we have that $ \hat{u}(\xi)= \frac{3m_1(W)}{\sqrt{2\pi}i}PV\left(\frac{1}{\xi}\right)+\hat{h}(\xi) $ is a solution to \eqref{distr3}. We denote by $ PV(\cdot) $ the principal value. Thus, adding the homogeneous solution we conclude
\begin{equation*}\label{bigthm 13}
\hat{u}(\xi)\overset{\mathcal{S}'}{=}c_0\delta+c_1\delta'+\frac{3}{2i}m_1(W)\sqrt{\frac{2}{\pi}}PV\left(\frac{1}{\xi}\right)+\hat{h}(\xi),
\end{equation*}
which yields
\begin{equation*}\label{distr4}
u(x)\overset{\mathcal{S}'}{=}\frac{c_0}{\sqrt{2\pi}}-\frac{c_1 i}{\sqrt{\pi}}x+\frac{3}{2}m_1(W)\sgn(x)+h(x),
\end{equation*}
where $ h\in L^2(\RR) $ is the inverse transform of $ \hat{h} $. Since $ u $ is bounded and satisfies $ u(x)=0 $ for all $ x<0 $, we have in distributional sense
\begin{equation*}\label{bigthm 14}
u(x)=\frac{3}{2}m_1(W)+\frac{3}{2}m_1(W)\sgn(x)+h(x).
\end{equation*}
Hence for $ x>0 $ also $ u(x)=3m_1(W)+h(x) $ pointwise.  Lemma \ref{unifcont} implies also that $ h $ is uniformly continuous in the positive real line. Hence, $ \lim\limits_{x\to\infty}h(x)=0 $  and therefore the limit of $ u $ as $ x\to\infty $ exists and is uniquely determined by $ g_\nu(n) $ and $ N$. This is true since
\begin{equation*}\label{bigthm 15}
\lim\limits_{y\to\infty}u(y)=3m_1(W)=3\left(\int_0^\infty dx\;xG(x)-\int_{-\infty}^0dx\;x\int_0^\infty dy\;K(y-x)u(y)\right)\geq 0.
\end{equation*}
Also the positivity of the limit is guaranteed when $\left\{n\in\Ss:n\cdot N<0\text{ and }\int_0^\infty d\nu \og_\nu(n)\not\equiv 0\right\} $ is not a zero measure set.
\end{proof}
\end{theorem}
We will define $ \ou_\infty(p):=\lim\limits_{y\to\infty}\ou(y,p) $ for $ p\in\bnd $.
We can also show that $ \ou $ converges to $ \ou_\infty $ with exponential rate.
\begin{lemma}\label{exp rate}
Let $ u $ be the unique bounded solution to the problem \eqref{inhomogeneous} and $ u_\infty=\lim\limits_{x\to\infty}u(x) $. Then there exists a constant $ C>0 $ such that
\begin{equation*}
	|u-u_\infty|\leq Ce^{-\frac{|x|}{2}}.
\end{equation*}
\begin{proof}
We use the same notation as in Theorem \ref{thmgrey}. Hence, we know that
\begin{equation}\label{uhat}
\hat{u}(\xi)\overset{\mathcal{S}'}{=}\frac{u_\infty}{2}\sqrt{2\pi}\delta+\frac{u_\infty}{2}\sqrt{\frac{2}{\pi}}PV\left(\frac{1}{\xi}\right)+\hat{h}(\xi),
\end{equation}
with $ F(\xi)\hat{h}(\xi)=\hat{W}(\xi)-\frac{3m_1(W)}{\sqrt{2\pi}i}\frac{F(\xi)}{\xi} $. By the definition of $ W $ we see 
\begin{equation}\label{u0}
\lim\limits_{x\nearrow 0}W(x)-\lim\limits_{x\searrow 0}W(x)=W(0^+)-W(0^-)=u(0).
\end{equation}
We recall that $ W $ has exactly one discontinuity in $ x=0 $ and that $ W\rchi_{\{x<0\}}\in C^\infty\left(\RR_-\right) $ and $ W\rchi_{\{x>0\}}\in C^\infty\left(\RR_+\right) $. By the monotonicity of the two functions $  W\rchi_{\{x<0\}} $ and $ W\rchi_{\{x>0\}} $ and since $ W\in L^\infty\left(\RR\right) $ we see that $ W'\rchi_{\{x<0\}}\in L^1(\RR_-)$ and $W'\rchi_{\{x>0\}}\in L^1(\RR_+) $. Moreoevr, we have the asymptotics $ \hat{W}(\xi)\simeq \frac{u(0)}{\sqrt{2\pi}i\xi}+\mathcal{O}\left(\frac{1}{\xi^{1+\delta}}\right)$ as $ |\xi|\to\infty $ for $ 0<\delta<1 $. Indeed, integrating by parts and using that $ \lim\limits_{|x|\to\infty}W(x)=0 $ we compute
\begin{equation}\label{higher order}
\begin{split}
\sqrt{2\pi}\hat{W}(\xi)=&\int_{-\infty}^0 W(x)e^{-i\xi x}\;dx+\int_0^\infty W(x)e^{-i\xi x}\;dx\\
=&\frac{u(0)}{i\xi}+\frac{1}{i\xi}\left(\int_{-\infty}^0 W'(x)e^{-i\xi x}\;dx+\int_0^\infty W'(x)e^{-i\xi x}\;dx\right)\\
=&\frac{u(0)}{i\xi}-\frac{1}{i\xi}\left(\int_{-\infty}^{-1}dx\int_0^\infty dy\; \frac{e^{-(y-x)}u(y)}{2(y-x)}e^{-i\xi x}+\int_{-\infty}^0dx\int_1^\infty dy\; \frac{e^{-(y-x)}u(y)}{2(y-x)}e^{-i\xi x}\right)\\&-\frac{1}{i\xi}\left(\int_{-1}^{0}dx\int_0^1 dy\; \frac{e^{-(y-x)}u(y)}{2(y-x)}\frac{d}{dx}\frac{e^{-i\xi x}-1}{-i\xi}\right)\\&+\frac{1}{i\xi}\left(\int_1^\infty G'(x)e^{-i\xi x}\;dx+\int_0^1 G'(x)\frac{d}{dx}\frac{e^{-i\xi x}-1}{-i\xi}\right)
\end{split}	
\end{equation}
We conclude integrating by parts and applying the Riemann-Lebesgue Theorem in the following way.
First of all, the function $ \partial_x \frac{e^{-(y-x)}}{(y-x)} $ is integrable on $ (-\infty,1)\times \RR_+\cup \RR_-\times(1,\infty) $ and also  $ G''(x) $ is integrable in $ (1,\infty) $. Moreover, using $ \left|e^{-i\xi x}-1\right|\leq 2 |\xi|^\delta|x|^\delta $ for $ 0<\delta<1 $ we have
\begin{equation*}
\int_{-1}^{0}dx\int_0^1 dy\; \frac{e^{-(y-x)}u(y)}{(y-x)}|x|^\delta\leq C\int_{-1}^0 dx\;\left(|x|^{\delta-1}\right)<\infty
\end{equation*}
and 
\begin{equation*}
\int_0^1 dx\;G''(x)|x|^\delta\leq C\int_0^1 \frac{e^{-x}}{x^{1-\delta}}\;dx<\infty.
\end{equation*}
For this last estimate we also used that $ \frac{d}{d\theta}e^{\frac{x}{\cos(\theta)}}=x\frac{e^{\frac{x}{cos(\theta)}}}{\cos^2(\theta)}\sin(\theta) $, which implies $ \left|G''(x)\right|\leq 2\pi \Arrowvert g\Arrowvert_{\infty} \frac{e^{-x}}{x} $. Thus, by the definition of $ \hat{h} $ and using \eqref{distr2} we have
\begin{equation}\label{behavior u}
\hat{h}(\xi)\simeq\begin{cases}
\mathcal{O}(1)& |\xi|\to 0,\\\frac{u(0)}{\sqrt{2\pi}}\frac{1}{i\xi}-\frac{u_\infty}{\sqrt{2\pi}}\frac{1}{i\xi}+\mathcal{O}\left(\frac{1}{\xi^{1+\delta}}\right)& |\xi|\to\infty.\end{cases}
\end{equation}
By the definition of $ \hat{u} $ in \eqref{uhat} we see
\begin{equation}\label{vhat}
\hat{v}(\xi):= \hat{u}(\xi)-\frac{u_\infty}{2}\sqrt{2\pi}\delta-PV\left(\frac{1}{i\xi}\right)\left(\frac{u_\infty}{2}\sqrt{\frac{2}{\pi}}\frac{1}{1+\xi^2}+\frac{u(0)}{\sqrt{2\pi}}\frac{\xi^2}{1+\xi^2}\right)\in L^2(\RR).
\end{equation}
We claim that 
\begin{enumerate}
	\item[(i)] $ \hat{v} $ is analytic in the strip $ S=\{z\in\mathbb{C}: |\Im(z)|<\frac{3}{4}\} $;
\item[(ii)] $| \hat{v}(\xi)|\leq \frac{C}{|1+\xi^{1+\delta}|} $;
\item[(iii)] $ v(x)=u(x)-u_\infty+\frac{e^{-|x|}}{2}\left(u_\infty-u(0)\right) $ for $ x>0 $ and $v(x)=\mathcal{F}^{-1}(\hat{v})(x) $.
\end{enumerate}
A contour integral implies then the lemma. Indeed for $ x>0 $ we can compute
\begin{equation}\label{contour}
\begin{split}
\sqrt{2\pi}|v(x)|=&\lim\limits_{R\to\infty}\left|\int_{-R}^R \hat{v}(\xi)e^{i\xi x}\;d\xi\right|\\
\leq&\lim\limits_{R\to\infty}\left|i\int_0^{\frac{1}{2}}\hat{v}(R+it)e^{iRx}e^{-tx}\;dt\right|+\lim\limits_{R\to\infty}\left|i\int_0^{\frac{1}{2}}\hat{v}(-R+it)e^{-iRx}e^{-tx}\;dt\right|\\&+\lim\limits_{R\to\infty}\left|\int_{-R}^R\hat{v}\left(t+\frac{1}{2}i\right)e^{itx}e^{-\frac{x}{2}}\;dt\right|\\
\leq&e^{-\frac{x}{2}}\lim\limits_{R\to\infty}\int_{-R}^R\frac{C}{\frac{1}{2}+t^{1+\delta}}\;dt=\overline{C}e^{-\frac{x}{2}},
\end{split}
\end{equation}
where for the first inequality we used the triangle inequality and the analycity of $ \hat{v} $ by (i), the second inequality is due to dominate convergence and the claim (ii), finally the last integral is finite. Equation \eqref{contour} and claim (iii) imply $ |u(x)-u_\infty|\leq Ce^{-\frac{x}{2}} $ for $ x>0 $.\\

We prove now the claims. To prove claim (i) it is enough to show that $ \hat{h} $ is analytic in
$ S $. Then, \eqref{uhat} and \eqref{vhat} implies (i). First of all we recall that $ W $ has an exponential decay like $ |W(x)|\leq Ce^{-|x|} $, hence $ |W(x)|e^{\frac{3}{4}|x|}\in L^1\left(\RR\right) $ and therefore Paley-Wiener Theorem implies that $ \hat{W} $ is analytic in $ S $. Since $ \arctan(z)=\frac{1}{2i}\ln(\frac{1+iz}{1-iz}) $ is analytic in $ \{z\in\mathbb{C}:|\Im(z)|<1\} $ and since $ F(z)=\frac{z-\arctan(z)}{z} $ has exactly one zero in $ z=0 $, which is of degree $ 2$, the definition of $ \hat{h}=\frac{f}{F} $ together with \eqref{def f} implies that $ \hat{h} $ is analytic in $ S $ since \eqref{behavior f} implies that $ 0 $ is a removable singularity.\\

For claim (ii) we just put together equations \eqref{uhat}, \eqref{behavior u} and \eqref{vhat}. We notice also that the constant $ C>0 $ of claim (ii) depends only on $ \hat{W} $. \\

Claim (iii) is more involved. We have to consider again two different contour integrals in order to compute the inverse Fourier transform of $ \hat{v} $. We start with considering the function $ PV\left(f(\xi)\right)=PV\left(\frac{1}{i\xi(1+\xi^2)}\right) $. Let first of all $ x>0 $ and let $ \gamma^+_1 $ the path around $ i $ given as in the following picture.
\begin{figure}[H]\centering
	\includegraphics[height= 3 cm]{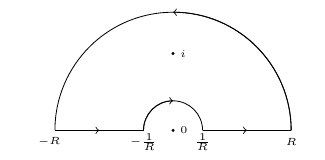}	\vspace{-.5cm}\caption{sketch of $ \gamma^+_1 $.}
\end{figure}
Hence, we compute 
\begin{equation*}
\begin{split}
\mathcal{F}^{-1}&\left(PV\left(\frac{1}{i\xi}\frac{1}{1+\xi^2}\right)\right)(x)=\frac{1}{\sqrt{2\pi}}\lim\limits_{R\to\infty}\left(\int_{-R}^{-\frac{1}{R}} f(\xi)e^{i\xi x}\;d\xi+\int_{\frac{1}{R}}^R f(\xi)e^{i\xi x}\;d\xi\right)\\
=&\frac{1}{\sqrt{2\pi}}\lim\limits_{R\to\infty}\left(\int_{\gamma^+_1}f(\xi)e^{i\xi x}d\xi\right)+\frac{1}{\sqrt{2\pi}}\lim\limits_{R\to\infty}\left(\int_0^\pi f\left(\frac{e^{i\theta}}{R}\right)\frac{ie^{i\theta}}{R}e^{-\frac{\sin(\theta)x}{R}}e^{\frac{i\cos(\theta)x}{R}}\;d\theta\right)\\
&-\frac{1}{\sqrt{2\pi}}\lim\limits_{R\to\infty}\left(\int_0^\pi f\left(Re^{i\theta}\right)Rie^{i\theta}e^{-R\sin(\theta)x}e^{iR\cos(\theta)x}\;d\theta\right)\\
=& \sqrt{\frac{\pi}{2}}\left(1-e^{-x}\right).
\end{split}
\end{equation*}
For the computation of these integrals we used the Cauchy's residue theorem and $ \text{Res}_i f(\xi)e^{i\xi x}=\frac{ie^{-x}}{2} $, the second integral converges to $ \pi $ as $ R\to \infty $ and the third converges to zero, both limit are due to the Lebesgue dominated convergence theorem. Denoting by $ \gamma^-_1 $ the mirrored path to $ \gamma^+_1 $ with respect to the real axis and arguing similarly we also get that for $ x<0 $ the inverse Fourier transformation is $ \mathcal{F}^{-1}\left(PV\left(\frac{1}{i\xi}\frac{1}{1+\xi^2}\right)\right)(x)= -\sqrt{\frac{\pi}{2}}\left(1-e^{-|x|}\right) $. Hence, 
\begin{equation}\label{inverse1}
\mathcal{F}^{-1}\left(PV\left(\frac{1}{i\xi}\frac{1}{1+\xi^2}\right)\right)(x)=\sgn(x)\sqrt{\frac{\pi}{2}}\left(1-e^{-|x|}\right).
\end{equation}
For the function $ g(x)(\xi)=\frac{\xi}{i(1+\xi^2)} $ we consider again first of all $ x>0 $ and the path $ \gamma^+_2 $ around $ i $ given by
\begin{figure}[H]\centering
	\includegraphics[height= 3 cm]{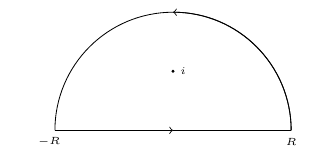}	\vspace{-.5cm}\caption{sketch of $ \gamma^+_2 $.}
\end{figure}
Hence, the Cauchy's residue theorem and the dominated convergence imply
\begin{equation*}
\begin{split}
\mathcal{F}^{-1}&\left(\frac{\xi}{i(1+\xi^2)}\right)(x)=\frac{1}{\sqrt{2\pi}}\lim\limits_{R\to\infty}\int_{-R}^{R} g(\xi)e^{i\xi x}\;d\xi\\
=&\frac{1}{\sqrt{2\pi}}\lim\limits_{R\to\infty}\left(\int_{\gamma^+_2}g(\xi)e^{i\xi x}d\xi\right)-\frac{1}{\sqrt{2\pi}}\lim\limits_{R\to\infty}\left(\int_0^\pi g\left(Re^{i\theta}\right)Rie^{i\theta}e^{-R\sin(\theta)x}e^{iR\cos(\theta)x}\;d\theta\right)\\
=& \sqrt{\frac{\pi}{2}}e^{-x},
\end{split}
\end{equation*}
where we also used that $ \text{Res}_i g(\xi)e^{-\xi x}=\frac{e^{-x}}{2i} $. Denoting similarly as before by $ \gamma^-_2 $ the mirrored path to $ \gamma^+_2 $ with respect to the real axis we obtain $  \mathcal{F}^{-1}(g)(x)=-\sqrt{\frac{\pi}{2}}e^{-|x|} $for $ x<0 $ and thus
\begin{equation}\label{inverse2}
\mathcal{F}^{-1}\left(\frac{\xi}{i(1+\xi^2)}\right)(x)=\sgn(x)\sqrt{\frac{\pi}{2}}e^{-|x|}.
\end{equation}
Hence, the definition of $ u $ and equations \eqref{inverse1}, \eqref{inverse2} imply claim (iii) for $ x>0 $ 
\begin{equation*}
v(x)=u(x)-u_\infty+\frac{e^{-|x|}}{2}\left(u_\infty-u(0)\right).
\end{equation*}
\end{proof}
\end{lemma}

There are still two important properties of $ \ou(y,p) $ we will need for the rest of the paper and which are explained in the next two Lemmas. First of all $ \ou(y,p) $ is uniformly bounded in both variables.

\begin{lemma}\label{unifbounded}
Let $ \ou(y,p) $ be the non-negative bounded solution to the problem \eqref{bvpgreyboundary1} for $ g_\nu(n) $ satisfying the assumption as in Theorem \ref{existence}. Then there exists a constant $ C $ such that 
\begin{equation*}\label{unifbounded1}
\sup\limits_{y\in\RR,\;p\in\bnd}\ou(y,p)\leq C<\infty.
\end{equation*}
\begin{proof}
By definition $ \ou $ satisfies $ \LL(\ou)(y)=G_p(y) $ for $ y>0 $ and $ \ou(y,p)=0 $ for $ y<0 $. Moreover, recalling the norm as in \eqref{norm1} the source can be estimated by
\begin{equation*}\label{unifbounded2}
0\leq G_p(y)\leq \Arrowvert g\Arrowvert_1 e^{-y},
\end{equation*}
since $ |n\cdot N_p|\leq1 $.

Theorem \ref{existence} assures us the existence of a unique bounded continuous (in the positive line) solution $ v $ of $ \LL(v)(y)=\Arrowvert g\Arrowvert_1 e^{-y} $ for $ y>0 $ and $ v(y)=0 $ for $ y<0 $. Hence, we can apply the maximum principle of Theorem \ref{maxprinciple} as we did in Lemma \ref{nonnegativity} to the function $ v-\ou(\cdot,p)\in C\left([0,\infty]\right) $ and we conclude
\begin{equation*}\label{unifbounded3}
0\leq \ou(y,p)\leq v(y)\leq \Arrowvert v\Arrowvert_{\infty}:=C<\infty
\end{equation*} for all $ y\in \RR $ and $ p\in\bnd $.
\end{proof}
\end{lemma}
Also, the rate of convergence of $ \ou(y,p) $ to $ \ou_\infty(p) $ can be bounded independently of $ p\in\bnd $.

\begin{corollary}\label{rate}
	There exists a constant $ C>0 $ independent of $ p\in\bnd $ such that $$ |\ou(y,p)-\ou_\infty(p)|\leq Ce^{-\frac{y}{2}} $$.
	
	\begin{proof}
		This is a consequence of Lemma \ref{exp rate} and Lemma \ref{unifbounded}.
		From Lemma \ref{unifbounded} we know that there exists a constant $ C>0 $ independent of $ p\in\bnd $ such that
		\begin{equation*}
		\left|W(x)\right|\leq C\left(e^{-|x|}+|x|K(x)\rchi_{\{x<0\}}\right)\in L^1(\RR)\cap L^2(\RR)\cap L^\infty(\RR),
		\end{equation*}
		where $ W $ is the function defined in \eqref{bigthm 1}. Since $ |x|K(x)\leq \frac{e^{-|x|}}{2} $ all moments of $ W $ are finite and for any $ n\in\mathbb{N} $ there exists a constant $ C_n>0 $ independent of $ p\in\bnd $ such that
		\begin{equation*}
		\left|m_n\left(W\right)\right|\leq C_n<\infty.
		\end{equation*} 
		Hence, $ \hat{W}\in C_0(\RR)\cap C^\infty(\RR)\cap L^2(\RR) $ and also all derivatives are uniformly bounded in $ p\in\bnd $ since $ \left|\hat{W}^{(n)}(\xi)\right|\leq \frac{C_n}{\sqrt{2\pi}} $. Thus, the function $ \hat{h} $ in \eqref{behavior u} defined using \eqref{behavior f} can be bounded independently of $ p\in\bnd $.
		
		Moreover, we notice that in \eqref{higher order} as $ |\xi|\to\infty $ we can bound $ \left|\hat{W}(\xi)-\frac{u(0)}{\sqrt(2\pi)i\xi}\right| $ by $ \frac{C}{\left|\xi^{1+\delta}\right|} $ with a constant $ C>0 $ independent of $ p\in\bnd $. Indeed, as we have seen in Lemma \ref{exp rate} we have $ \left|G''(x)\right|\leq 2\pi \Arrowvert g\Arrowvert_{\infty} \frac{e^{-x}}{x} $ and by Lemma \ref{unifbounded} we have also $ \left|\ou(y,p)\right|\leq C $.
		
		Hence, we conclude as we did in Lemma \ref{exp rate} that there exists a constant $ C>0 $ independent of $ p\in\bnd $ such that $ \left|\hat{v}(\xi)\right|\leq\frac{C}{|1+\xi^{1+\delta}|} $, where $ \hat{v} $ was defined in \eqref{vhat}. 
		
		Arguing now exactly as in Lemma \ref{exp rate} using also Lemma \ref{unifbounded} we conclude that there exists a constant $ C>0 $ independent of $ p\in\bnd $ such that $ |\ou(y,p)-\ou_\infty(p)|\leq Ce^{-\frac{y}{2}} $.
	\end{proof}
\end{corollary}
Next, using again the maximum principle we can also show that $ \ou(y,p) $ is Lipschitz continuous with respect to $ p\in\bnd $ uniformly in $ y $.
\begin{lemma}\label{ou continuous in p}
	Let $ g_\nu(n) $ be as in Theorem \ref{existence} and let $ \ou $ be the unique bounded solution to \eqref{bvpgreyboundary1}. Then $ \ou $ is uniformly continuous with respect the variable $ p\in\bnd $ uniformly in $ y $. More precisely, it is Lipschitz continuous, i.e. there exists a constant $ C>0 $ such that for every $ p,q\in\bnd $
	\begin{equation*}\label{ou cont p 1} 
	\sup_{y\geq 0}\left|\ou(y,p)-\ou(y,q)\right|\leq C|p-q|:= \omega_1(|p-q|).
	\end{equation*}
	\begin{proof}
		The proof is based on the maximum principle. We start taking $ 0<\tilde{\delta}<1 $ sufficiently small and we consider $ p,q\in\bnd $ with $ |p-q|<\tilde{\delta} $. We denote by $ S_p(q) $ the plane defined by the vector $ \overset{\rightharpoonup}{pq} $ and the unit vector $ N_p $. Given that $ \bnd $ is a $ C^3 $-surface we can define $ \rho_p $ to be the radius of curvature of the curve $ C_p(q):=S_p(q)\cap\bnd $ at $ p $. Since by assumption the curvature of $ \bnd $ is bounded from below by a positive constant, for $ \tilde{\delta} $ small enough we can estimate
		\begin{equation}\label{ou cont p}
		\frac{1}{2}\rho_p\theta_{pq}\leq |p-q|\leq 2 \rho_p\theta_{pq},
		\end{equation}
		where $ \theta_{pq} $ is the angle between $ N_p $ and $ N_q $. This is true, because for $ \tilde{\delta} $ sufficiently small the angle $ \theta_{pq} $ is not zero and it is approximately the central angle between the rays connecting $ p $ and $ q $ with the center of the circle with radius $ \rho_p $ tangent to $ p $. We denote by $ R $ the minimal radius of curvature of $ \bnd $, hence $ \rho_p\geq R $. Now we consider the operator $ \LL $ acting on the difference $ \ou(y,p)-\ou(y,q) $. We can estimate its absolute value by the sum of the following six terms
		\begin{equation}\label{ou cont p 2}
		\begin{split}
		\left|\LL\left(\right.\right.&\left.\left.\ou(y,p)-\ou(y,q)\right)\right|\leq\int_{A_1}\int_0^\infty g_\nu(n)e^{-\frac{y}{|n\cdot N_p|}}\;d\nu\;dn+\int_{A_2}\int_0^\infty g_\nu(n)e^{-\frac{y}{|n\cdot N_q|}}\;d\nu\;dn \\&+ \int_{A_3}\int_0^\infty g_\nu(n)\left|e^{-\frac{y}{|n\cdot N_p|}}-e^{-\frac{y}{|n\cdot N_q|}}\right|\;d\nu\;dn+\int_{A_4} \int_0^\infty g_\nu(n)\left|e^{-\frac{y}{|n\cdot N_p|}}-e^{-\frac{y}{|n\cdot N_q|}}\right|\;d\nu\;dn\\
		&+ \int_{A_5}\int_0^\infty g_\nu(n)\left|e^{-\frac{y}{|n\cdot N_p|}}-e^{-\frac{y}{|n\cdot N_q|}}\right|\;d\nu\;dn+\int_{A_6} \int_0^\infty g_\nu(n)\left|e^{-\frac{y}{|n\cdot N_p|}}-e^{-\frac{y}{|n\cdot N_q|}}\right|\;d\nu\;dn,
		\end{split}
		\end{equation}
		where we denote by $ A_i $ the following sets
		\begin{equation*}
			\begin{split}
			A_1&:=\left\{n\in\Ss^2: n\cdot N_p<0,\;n\cdot N_q\geq 0\right\},\;\;\;\;A_2:=\left\{n\in\Ss^2: n\cdot N_p\geq 0,\;n\cdot N_q< 0\right\},\\
			A_3&:=\left\{n\in\Ss^2: n\cdot N_p<0,\;n\cdot N_q< 0,\; |n\cdot N_p|\geq |n\cdot N_q|,|;|n\cdot N_p|> \frac{4}{R}|p-q|\right\},\\
			A_4&:=\left\{n\in\Ss^2: n\cdot N_p<0,\;n\cdot N_q< 0,\; |n\cdot N_p|\geq |n\cdot N_q|,|;|n\cdot N_p|\leq \frac{4}{R}|p-q|\right\},\\
			A_5&:=\left\{n\in\Ss^2: n\cdot N_p<0,\;n\cdot N_q< 0,\; |n\cdot N_q|\geq |n\cdot N_p|,|;|n\cdot N_q|> \frac{4}{R}|p-q|\right\} \text{ and }\\
			A_6&:=\left\{n\in\Ss^2: n\cdot N_p<0,\;n\cdot N_q< 0,\; |n\cdot N_q|\geq |n\cdot N_p|,|;|n\cdot N_q|\leq \frac{4}{R}|p-q|\right\}.\\
			\end{split}
		\end{equation*}
		By symmetry, we need to estimate only the first, the third and the fourth terms. We start with the first line of equation \eqref{ou cont p 2}. The set $ A_1 $ is contained by the set given by all the $ n $ such that their angle with $ N_p $ is in the interval $ (\frac{\pi}{2},\frac{\pi}{2}+\theta_{pq}) $. Using the fact that $ \frac{y}{|n\cdot N_p|}>y $, we estimate the exponential by $ e^{-y} $ and hence we see
		\begin{equation}\label{ou cont p 3}
		\int_{A_1} \int_0^\infty g_\nu(n)e^{-\frac{y}{|n\cdot N_p|}}\;d\nu\;dn\leq \Arrowvert g\Arrowvert_\infty 2\pi \theta_{pq} e^{-y}\leq \frac{4\pi}{R}\Arrowvert g\Arrowvert_\infty e^{-y}.
		\end{equation}
		The second term in \eqref{ou cont p 2} is estimated similarly.
		For the third term of equation \eqref{ou cont p 2} we estimate the difference of the exponential as follows, assuming $  |n\cdot N_p|\geq |n\cdot N_q|$
		\begin{equation*}\label{ou cont p 4}
		\left|e^{-\frac{y}{|n\cdot N_p|}}-e^{-\frac{y}{|n\cdot N_q|}}\right|\leq e^{-\frac{y}{|n\cdot N_p|}}y \left|\frac{1}{|n\cdot N_q|}-\frac{1}{|n\cdot N_p|}\right|\leq e^{-\frac{y}{|n\cdot N_p|}}y \left|\frac{|n\cdot N_p|-|n\cdot N_q|}{|n\cdot N_q||n\cdot N_p|}\right|,
		\end{equation*}
		where we used for $ x>0 $ the inequality $ 1-e^{-x}\leq x $. By definition $ |n\cdot (N_p-N_q)|\leq \theta_{pq}\leq \frac{2}{R}|p-q| $ which implies
		\begin{equation*}
		0\leq |n\cdot N_p|-|n\cdot N_q|=|n\cdot(N_q-N_p)|\leq \frac{2}{R}|p-q|. 
		\end{equation*} Since $ |n\cdot N_p|>\frac{4}{R}|p-q| $ we see also that
		\begin{equation*}\label{ou cont p 5}
		|n\cdot N_q|\geq |n\cdot N_p|-\frac{2}{R}|p-q|\geq \frac{|n\cdot N_p|}{2}.\end{equation*}
		Hence, 
		\begin{equation*}
		\left|e^{-\frac{y}{|n\cdot N_p|}}-e^{-\frac{y}{|n\cdot N_q|}}\right|\leq e^{-\frac{y}{|n\cdot N_p|}}y\frac{4 |p-q|}{R|n\cdot N_p|^2}.
		\end{equation*} Putting together these inequalities we compute
		\begin{equation}\label{ou cont p 6}
		\begin{split}
		\int_{A_3} \int_0^\infty g_\nu(n)&\left|e^{-\frac{y}{|n\cdot N_p|}}-e^{-\frac{y}{|n\cdot N_q|}}\right|\;d\nu\;dn\leq \frac{4|p-q|}{R}\Arrowvert g\Arrowvert_{\infty} \int_{A_3}dn\;e^{-\frac{y}{|n\cdot N_p|}}\frac{y}{|n\cdot N_p|^2}\\\leq&\frac{4|p-q|}{R}\Arrowvert g\Arrowvert_{\infty}4\pi\int_0^{\frac{\pi}{2}}
		e^{-\frac{y}{\cos(\theta)}}\frac{y\sin(\theta)}{\cos^2(\theta)}\;d\theta=\frac{16\pi|p-q|}{R}\Arrowvert g\Arrowvert_{\infty} e^{-y},
		\end{split}
		\end{equation}
		where we estimated the last integral in $ A_3 $ using polar coordinates in $ \Ss^2 $ using as reference $ N_p $.
		It remains to estimate the integral on $ A_4 $. For this term we use the inclusion
		\begin{equation*}
			\begin{split}
			 A_4\subset&\left\{n\in\Ss^2: n\cdot N_p<0,\;|n\cdot N_p|\leq \frac{4}{R}|p-q|\right\}\\\subset&\left\{(\varphi,\theta)\in[0,2\pi]\times[0,\pi]: \theta\in \left(-\frac{\pi}{2},-\frac{\pi}{2}+C(R)|p-q|\right)\cup\left(\frac{\pi}{2}-C(R)|p-q|,\frac{\pi}{2}\right)\right\},
			\end{split}
		\end{equation*}
		where the last inclusion is due to the smallness of $ \frac{4}{R}|p-q|<1 $ and the expansion of the arc-cosine. Moreover, $ C(R) $ is a constant depending only on $ R $. Hence, as we estimated in \eqref{ou cont p 3} we have
		\begin{equation}\label{ou cont p 9}
		\int_{A_4} \int_0^\infty g_\nu(n)\left|e^{-\frac{y}{|n\cdot N_p|}}-e^{-\frac{y}{|n\cdot N_q|}}\right|\;d\nu\;dn\leq C(R)4\pi\Arrowvert g\Arrowvert_{\infty}|p-q|.
		\end{equation}
		Now, with equations \eqref{ou cont p 3},\eqref{ou cont p 6} and \eqref{ou cont p 9} we estimate the operator by
		\begin{equation*}\label{ou cont p 7}
		\left|\LL\left(\ou(y,p)-\ou(y,q)\right)\right|\leq C(R)\Arrowvert g\Arrowvert_{\infty}|p-q|e^{-y},
		\end{equation*}
		where $ C(R)>0 $ is a constant depending only on the minimal radius of curvature $ R $.
		Theorem \ref{existence} and the maximum principle imply the existence of a unique non-negative bounded continuous function $ V $ solution to the equation $ \LL(V)(y)= e^{-y} $ for $ y\geq 0 $. Hence, we apply the maximum principle of Theorem \ref{maxprinciple} as in Lemma \ref{nonnegativity} to the continuous functions $ C(R)\Arrowvert g\Arrowvert_{\infty} |p-q| V -\left(\ou(y,p)-\ou(y,q)\right) $ and $ C(R)\Arrowvert g\Arrowvert_{\infty} |p-q| V -\left(\ou(y,q)-\ou(y,p)\right) $. We conclude the uniformly continuity of $ \ou(y,p) $ in $ p $ uniformly in $ y $
		\begin{equation*}\label{ou cont p 8}
		\left|\ou(y,p)-\ou(y,q)\right|\leq C(R)\Arrowvert g\Arrowvert_{\infty}|p-q|.
		\end{equation*}
		The modulus of continuity $ \omega_1 $ is hence defined by $ \omega_1(r)= C(R)\Arrowvert g\Arrowvert_{\infty} r $.
	\end{proof}
\end{lemma}
\begin{corollary}\label{ou infty continuous}
	The	limit $ \ou_\infty $ is Lipschitz continuous in $ p\in\bnd $.
	\begin{proof}
		This is a direct consequence of the previous Lemma \ref{ou continuous in p}. The modulus of continuity of $ \ou_\infty $ is still the same $ \omega_1 $ of $ \ou(y,p) $.
	\end{proof}
\end{corollary}

Finally, we summarize all properties of $ \ou $ in the following proposition.
\begin{prop}\label{prop u}
	Let $ g_\nu(n) $ be as in Theorem \ref{existence} and $ \Omega $ as in the assumption. For every $ p\in\bnd $ there exists a unique non-negative bounded solution $ \ou(y,p) $ to \eqref{bvpgreyboundary1}. For every $ p\in \bnd $ the function $ \ou(\cdot,p) $ is uniformly continuous in $[0,\infty)$ and has a non-negative limit $ \ou_\infty(p)=\lim\limits_{y\to\infty}\ou(y,p) $, which is strictly positive if $ \left\{n\in\Ss:n\cdot N_p<0\text{ and }\int_0^\infty d\nu g_\nu(n)\not\equiv 0\right\} $ is not a zero measure set. Moreover, $ \ou(y,p) $ is uniformly bounded in both variables and it is Lipschitz continuous with respect to $ p\in\bnd $ uniformly on $ y\in\RR_+ $. Finally, $ \ou_\infty $ is Lipschitz continuous and there exists a constant $ C>0 $ independent of $ p\in\bnd $ such that $ |\ou(y,p)-\ou_\infty(p)|\leq Ce^{-\frac{|y|}{2}} $.
\end{prop}
\section{Rigorous proof of the diffusion equilibrium approximation for constant absorption coefficient}
This section of the paper deals with the rigorous proof of the diffusion equilibrium approximation for the constant absorption coefficient case. We will show that the Stefan-Boltzmann law $ u^\eps(x)=4\pi\sigma T_\eps^4(x) $ for the temperature $ T_\eps $ associated to the boundary value problem \eqref{bvpnoscattering} converges pointwise as $ \eps\to0 $ to $ v $, the solution to the 
Dirichlet problem
\begin{equation}\label{defv}
\begin{cases}
-\Delta v =0& \text{ in } \Omega,\\v=\ou_\infty&\text{ on }\bnd,
\end{cases}
\end{equation}
where $ \ou_\infty $ is defined as in Proposition \ref{prop u}.
\subsection{Derivation of the equation for $ u^\eps $}
Let us call $ I^{\eps}_\nu $ the solution to the initial boundary value problem \eqref{bvpnoscattering}. We start with the derivation of the integral equation satisfied by $ u^\eps=4\pi\sigma T_\eps^4 $. To this end we solve by characteristics the equation 
\begin{equation*}
n\cdot \nabla_x I_\nu \left(x,n\right)= \frac{1}{\eps}\left(B_\nu\left(T\left(x\right)\right)-I_\nu \left(x,n\right)\right)
\end{equation*}
Let $ x\in\Omega $ and $ n\in\Ss^2 $. The convexity of $ \Omega $ implies the existence of a unique $ x_\Omega(x,n)\in\bnd $ connecting $ x $ in direction $ -n $ with the boundary $ \bnd $. Hence, $ \frac{x-x_\Omega(x,n)}{\left|x-x_\Omega(x,n)\right|}=n $ and we define $ s(x,n)=\left|x-x_\Omega(x,n)\right| $. Then $ x=x_\Omega(x,n)+s(x,n)n $. Integrating along the characteristics equation \eqref{bvpnoscattering} we get
\begin{equation*}\label{I eps charact}
\begin{split}
I^\eps_\nu(x,n)=g_\nu(n)e^{-\frac{\left|x-x_\Omega(x,n)\right|}{\eps}}+\frac{1}{\eps}\int_0^{s(x,n)}e^{-\frac{t}{\eps}}\Bnu{x-tn}\;dt.
\end{split}
\end{equation*}
Using the heat equation, i.e. $ \nabla_x \cdot\F=0 $ (see \eqref{bvpnoscattering}), we calculate
\begin{equation*}\label{I eps 2}
0=\intnu \intS \;n\cdot \nabla_xI^\eps_\nu(x,n)=\frac{1}{\eps}\intnu\intS\; B_\nu(T_\eps)(x)-I^\eps_\nu(x,n).
\end{equation*}
We define $ u^\eps(x)=4\pi\sigma T_\eps^4(x)=\intnu\intS B_\nu\left(T_\eps(x)\right) $ according to \eqref{sigma}. Hence also $ u^\eps(x)=\intnu\intS\;I_\nu^\eps(x,n) $. We integrate now the expression we got for the intensity and we conclude with the equation satisfied by $ u^\eps $ as follows
\begin{equation*}\label{eq u eps}
\begin{split}
u^\eps(x)&=\intnu\int_{\Ss^2}dn\;g_\nu(n)e^{-\frac{\left|x-x_\Omega(x,n)\right|}{\eps}}+\frac{1}{4\pi\eps}\intS\int_0^{s(x,n)}e^{-\frac{t}{\eps}}u^{\eps}(x-tn)\;dt\\
&=\intnu\int_{\Ss^2}dn\;g_\nu(n)e^{-\frac{\left|x-x_\Omega(x,n)\right|}{\eps}}+\frac{1}{4\pi\eps}\int_\Omega \frac{e^{-\frac{\left|x-\eta\right|}{\eps}}}{\left|x-\eta\right|^2}u^{\eps}(\eta)\;d\eta,
\end{split}
\end{equation*}
where the last equality is due to the change of variables $ \Ss^2\times (0,\infty)\to\Omega $ with $ (n,t)\mapsto x-tn=\eta $. Hence the sequence $ u^\eps $ of exact solutions solves
\begin{equation}\label{u eps}
u^\eps(x)-\int_\Omega \frac{e^{-\frac{\left|x-\eta\right|}{\eps}}}{4\pi\eps\left|x-\eta\right|^2}u^{\eps}(\eta)\;d\eta=\intnu\int_{\Ss^2}dn\;g_\nu(n)e^{-\frac{\left|x-x_\Omega(x,n)\right|}{\eps}}.
\end{equation}
We define the kernel $ K_\eps(x):=\Kern{x} $ and we notice that its integral in $ \RR^3 $ is $ 1 $.
\begin{remark}
	There exists a unique solution $ u^\eps $ continuous and bounded. We adapt the proof in \cite{jang}. The existence and uniqueness of a solution $ u^\eps \in L^\infty\left(\Omega\right)$ can be shown with the Banach fixed-point Theorem. We define for every given $ g $ and $ \eps>0 $ the self map $  A_g^\eps: L^\infty\left(\Omega\right)\to L^\infty\left(\Omega\right) $ by 
	\begin{equation*}
	 A_g^\eps(u)(x)=\int_\Omega K_\eps\left(\eta-x\right)u(\eta)\;d\eta+\intnu\int_{\Ss^2}dn\;g_\nu(n)e^{-\frac{\left|x-x_\Omega(x,n)\right|}{\eps}}.
	\end{equation*} Then since $ \int_\Omega K_\eps(\eta-x)\;d\eta<\int_{\RR^3} K_\eps(\eta-x)\;d\eta=1 $ we conclude that $ A_g^\eps $ is a contraction, hence there is a unique fixed-point, which is the desired unique solution. Moreover,  $ G^\eps_{x_\Omega}(x):=\intnu\int_{\Ss^2}dn\;g_\nu(n)e^{-\frac{\left|x-x_\Omega(x,n)\right|}{\eps}} $ is continuous and since $ u^\eps\in L^\infty\left(\Omega\right) $ and $ K_\eps(x-\cdot)\in\Lprr{1} $ we conclude that the convolution $ \int_\Omega K_\eps(\eta-x)u^{\eps}(\eta)\;d\eta $ is continuous and bounded. Hence, $ u^\eps $ is continuous and bounded. We can also extend continuously $ u^\eps $ to the boundary $ \bnd $ defining $ |x-x_\Omega(x,n)|=0 $ for $ x\in\bnd $ and $ n\cdot N_x\leq 0 $. Then using the generalized dominated convergence theorem we see that both integral terms in \eqref{u eps} are continuous up to the boundary. Hence, $ u^\eps \in C\left(\overline{\Omega}\right) $. Moreover, $ u^\eps $ is non-negative. This is because of the maximum principle as stated in the following theorem. 
\end{remark}

\begin{theorem}[Maximum Principle]\label{max principle omega}
Let $ v $ be bounded and continuous, $ v\in C\left(\overline{\Omega}\right) $. Let $ \LL_\Omega^\eps (v)(x)=v(x)-\int_\Omega K_\eps(\eta~-~x)v(\eta)\;d\eta$. Assume $ v $ satisfies one of the following properties:
\begin{enumerate}
	\item[(i)] $
	\LL_\Omega^\eps (v)(x)\geq 0$ if $x\in\Omega$;
	\item[(ii)] $ \LL_\Omega^\eps (v)(x)\geq 0$ if $ x\in O\subset\Omega $ open and $ v(x)\geq 0 $ if $x\in\Omega\setminus O$.
\end{enumerate}

Then, $ v\geq 0 $.
\begin{proof}
Let $ y\in\overline{\Omega} $ such that $ v(y)=\min_{x\in\overline{\Omega}}v(x) $. Assume $ v(y)<0 $.

Assume that property $ (i) $ holds. By continuity of the the operator we have that $ \LL^\eps_\Omega(v)(x)\geq 0 $ for all $ x\in\overline{\Omega} $. Then
\begin{equation}\label{max principle omega 1}
\begin{split}
0\leq& \LL^\eps_\Omega(v)(y)=v(y)-\int_\Omega K_\eps(\eta-y)v(\eta)\;d\eta\\=&\int_\Omega K_\eps(\eta-y)\left(v(y)-v(\eta)\right)\;d\eta +v(y)\int_{\Omega^c} K_\eps(\eta-y)\;d\eta<0,
\end{split}
\end{equation}
where we used the normalization of the kernel $ K_\eps $.
Hence, this contradiction yields $ v\geq 0 $.

Assume now that $ (ii) $ holds. Then in this case $ y\in \overline{O} $. Then again by the continuity of the operator we obtain exactly as in \eqref{max principle omega 1} a contradiction. Thus the Theorem is proved.
\end{proof}
\end{theorem}

\subsection{Uniform boundedness of $ u^\eps $}
In this section we will show that the sequence $ u^\eps $ is uniformly bounded in $ \eps $. We will use the maximum principle again. Indeed, we will construct functions $ \Phi^\eps $ uniformly bounded such that $ \LL^\eps_\Omega(\Phi^\eps)(x)\geq \Arrowvert g\Arrowvert_1e^{-\frac{\dist(x,\partial \Omega)}{\eps}} $. We will use this to prove $ \LL^\eps_\Omega\left(\Phi^\eps-u^\eps\right)(x)\geq 0 $ which implies using the maximum principle $ 0\leq u^\eps\leq \Phi^\eps $. The main result of this subsection is the following.

\begin{theorem}\label{interior supsol}
	There exists suitable constants $ 0<\mu<1 $, $ 0<\gamma(\mu)<\frac{1}{3} $, $ C_1,\;C_2,\;C_3>0 $ and there exists some $ \eps_0>0 $ such that the function 
	\begin{equation*}\label{supsol1}
	\Phi^\eps(x)=C_3 \Arrowvert g\Arrowvert_1 \left(C_1-\left|x\right|^2\right)+C_2\Arrowvert g\Arrowvert_1\left[\left(1-\frac{\gamma}{1+\left(\frac{d(x)}{\eps}\right)^2}\right)\wedge\left(1-\frac{\gamma}{1+\left(\frac{\mu R}{\eps}\right)^2}\right)\right],
	\end{equation*}
	for  $ a\wedge b\:=\min\left(a,b\right) $, $ R>0 $ the minimal radius of curvature $ R=\min_{x\in\bnd}R(x) $ and $ d(x):=\dist\left(x,\partial\Omega\right) $, 
	satisfies $ \LL_\Omega^\eps\left(\Phi^\eps\right)(x)\geq \Arrowvert g\Arrowvert_1 e^{-\frac{d(x)}{\eps}} $ in $ \Omega $ uniformly for all $ \eps<\eps_0 $. Moreover, the solutions $ u^\eps $ of \eqref{u eps} are uniformly bounded in $ \eps $.
\end{theorem}

We split the proof of this theorem in two lemmas.

\begin{lemma}\label{first supsol}
	Let $ C_1:= 2\max_{x\in\overline\Omega}\left|x\right|^2+2\;\diam\left(\Omega\right)^2+4\;\diam\left(\Omega\right)+4 $, let $ 0<\eps<1 $. Then 
	\begin{equation*}\label{first supsol lem 1}
	\LL^\eps_\Omega \left(C_1-\left|x\right|^2\right)\geq 2\eps^2. 
	\end{equation*}
	\begin{proof}
		We start computing the action of $ \LL^\eps_{\RR^3} $ on $ \left|x\right|^2 $. 
		\begin{equation*}\label{first supsol lem 2}
		\begin{split}
		\LL^\eps_{\RR^3}\left[\left|\cdot\right|^2\right](x)=& \left|x\right|^2-\int_{\RR^3} K_\eps\left(\eta-x\right)\left|\eta\right|^2\;d\eta\\
		=&-\int_{\RR^3} K_\eps\left(\eta-x\right)\left|\eta-x\right|^2\;d\eta=-2\eps^2,
		\end{split}
		\end{equation*}
		where we expanded $ |\eta|^2=|x+(\eta-x)|^2 $ and we used that $ \int_{\RR^3} K_\eps=1 $ and the symmetry of the kernel $ K_\eps $.
		
		Le $ D:=\diam(\Omega) $ and let $ B=2\max_{x\in\overline\Omega}\left|x\right|^2 $ and $ \beta=2D^2+4D+4 $. Thus, $ C_1=B+\beta $. Then
		\begin{equation*}\label{first supsol lem 3}
		\begin{split}
		\LL^\eps_\Omega &\left(K+\beta-\left|\cdot\right|^2\right)(x)=(B+\beta)\int_{\Omega^c} K_\eps\left(\eta-x\right)\;d\eta\;-\LL^\eps_{\RR^3}\left[\left|\cdot\right|^2\right](x)-\int_{\Omega^c} K_\eps\left(\eta-x\right)\left|\eta\right|^2\;d\eta\\
		\geq &\left(B+\beta\right)\int_{\Omega^c} K_\eps\left(\eta-x\right)\;d\eta+2\eps^2-2\left|x\right|^2\int_{\Omega^c} K_\eps\left(\eta-x\right)\;d\eta-2\int_{\Omega^c} K_\eps\left(\eta-x\right)\left|\eta-x\right|^2\;d\eta,\\
		\end{split}
		\end{equation*}
		where we used $\left|\eta\right|^2\leq 2\left|x\right|^2+2\left|\eta-x\right|^2  $. Moreover using that $ B-2|x|^2\geq 0 $ and splitting for $ x\in\Omega $ the complement of the domain as $ \Omega^c= \left(\Omega^c\cap B_D(x)\right)\cup B_D^c(x) $ we obtain
		\begin{equation*}\label{first supsol lem 4}
		\begin{split}
		\LL^\eps_\Omega \left(K+\beta-\left|\cdot\right|^2\right)(x)\geq& 2\eps^2 +\int_{B^c_{D/\eps}(0)}K_\eps\left(\eta\right)\left(\beta-2\eps^2\left|\eta\right|^2\right)\;d\eta\\
		=&2\eps^2+\beta e^{-\frac{D}{\eps}}-e^{-\frac{D}{\eps}}\left(2D^2+4D\eps+4\eps^2\right)\geq 2\eps^2,
		\end{split}
		\end{equation*}
		where in the first inequality we used that $ 2\left|\eta-x\right|^2\leq 2D^2\leq \beta $ for $ \eta,x\in B_D(x) $ and for the integral in $ B^c_D(x) $ we changed variables $ \frac{\eta-x}{\eps}\mapsto \eta $ and we computed the resulting integral using also that $ \eps<1 $.
	\end{proof}
\end{lemma}
In order to proceed further with the construction of the supersolution, we will use repeatedly the distance function and its relation to the curvature of the domain's boundary. All the properties of this function can be found in the Appendix ``Boundary curvatures and distance functions" in \cite{Boundarycurvature}. It is well-know that if the boundary $ \bnd $ is $ C^3 $, then in a neighborhood of the boundary the distance function can be expanded by Taylor as 
\begin{equation}\label{taylor distance}
d(\eta)=d(x)+\nabla d(x)\cdot (\eta-x)+\frac{1}{2}\left(\eta-x\right)^\top\nabla^2 d(x)\left(\eta-x\right)+\mathcal{O}\left(\left|\eta-x\right|^3\right)
\end{equation}
Moreover, the following proposition holds.
\begin{prop}\label{prop dist}
	For $ x\in\Omega $ in a neighborhood of the boundary the gradient of the distance function is the inner normal, so that $ \left|\nabla d(x)\right|=1 $. Moreover, denoting $ R=\min_{x\in\bnd} R(x)>0$ the minimal radius of curvature and letting $ \mu\in(0,1) $ we have
	\begin{equation}\label{dist hessian}
	\xi^\top \nabla^2 d(x)\xi\leq \frac{1}{(1-\mu)R}
	\end{equation}
	for  every $ x\in\left\{y\in\Omega:\;d(y)<R\mu\right\} $ and $ \Arrowvert\xi\Arrowvert=1 $.
	\begin{proof}
		See 14.6, Appendix ``Boundary curvatures and distance functions" (\cite{Boundarycurvature}).
	\end{proof}
\end{prop}
Using these properties of the distance function we can prove the next lemma.
\begin{lemma}\label{second supsol}
	Let $ \psi(x):=\left(1-\frac{\gamma}{1+\left(\frac{d(x)}{\eps}\right)^2}\right)\wedge\left(1-\frac{\gamma}{1+\left(\frac{\mu R}{\eps}\right)^2}\right) $. Then there exists some $ 0<\mu<1 $ small enough, $ 0<\gamma(\mu)<\frac{1}{3} $, $ 0<\eps_1<1 $ small enough and constants $ C_0:=C_0(R, \Omega, \mu, \gamma)>0 $ and $ c:=c(R, \mu,\gamma)>0 $ such that for all $ 0<\eps\leq\eps_1 $
	\begin{equation}\label{second supsol 1}
	\LL_\Omega^\eps\left(\psi\right)(x)\geq\begin{cases}
	C_0 e^{-\frac{d(x)}{\eps}} & 0<d(x)\leq \frac{R\mu}{2}\\
	-c\eps^2 & \frac{R\mu}{2}<d(x)< R\mu\\
	0& d(x)\geq R\mu\\
	\end{cases}
	\end{equation}
	\begin{proof}
		We start with some preliminary consideration on the distance function. We define $ \frac{d(\eta)}{\eps}:=\deps{\eta} $. For every $ x,\eta\in\left\{y\in\Omega:\;d(y)<R\mu\right\} $ we have using \eqref{taylor distance}
		\begin{equation}\label{taylor square distance}
		\begin{split}
		\deps{\eta}^2
		=&\deps{x}^2+\frac{2d(x)\nabla d(x)\cdot \left(\eta-x\right)}{\eps^2}+\frac{d(x)\left(\eta-x\right)^\top\nabla^2 d(x)\left(\eta-x\right)}{\eps^2}\\&+\frac{\left(\nabla d(x)\cdot (\eta-x)\right)^2}{\eps^2}+\mathcal{O}\left(\frac{d(x)}{\eps^2}\left|\eta-x\right|^3\right).
		\end{split}
		\end{equation}
		Then Taylor's expansion shows
		\begin{equation}\label{taylor fraction distance}
		\begin{split}
		\frac{1}{1+\deps{\eta}^2}=&\frac{1}{\left(1+\deps{x}^2\right)\left(1+\left[\deps{\eta}^2-\deps{x}^2\right]\frac{1}{1+\deps{x}^2}\right)}\\=&Q_\eps^{(1)}(x,\eta)+Q_\eps^{(2)}(x,\eta)+Q_\eps^{(3)}(x,\eta),\\
		\end{split}
		\end{equation}
		where we the terms $ Q_\eps^{(i)} $ are defined as follows.
		\begin{equation*}\label{Q1}
		Q_\eps^{(1)}(x,\eta)= \frac{1}{1+\deps{x}^2}-\frac{2d(x)\nabla d(x)\cdot \left(\eta-x\right)}{\eps^2\left(1+\deps{x}^2\right)^2},
		\end{equation*}
		\begin{equation*}\label{Q_2}
		Q_\eps^{(2)}(x,\eta)= -\frac{d(x)\left(\eta-x\right)^\top\nabla^2 d(x)\left(\eta-x\right)}{\eps^2\left(1+\deps{x}^2\right)^2}-\frac{\left(\nabla d(x)\cdot (\eta-x)\right)^2}{\eps^2\left(1+\deps{x}^2\right)^2}+\frac{4d^2(x)\left(\nabla d(x)\cdot (\eta-x)\right)^2}{\eps^4\left(1+\deps{x}^2\right)^3},
		\end{equation*}
		\begin{equation*}\label{Q3}
		Q_\eps^{(3)}(x,\eta)=\mathcal{O}\left(\frac{d(x)}{\eps^2}\frac{\left|\eta-x\right|^3}{\left(1+\deps{x}^2\right)^2}\right)+\mathcal{O}\left(\frac{d(x)}{\eps^4}\frac{\left|\eta-x\right|^3}{\left(1+\deps{x}^2\right)^3}\right).
		\end{equation*}
		We consider now the function $ \psi(x)$ defined in the statement of Lemma \ref{second supsol}. We take $ M=\frac{1}{\mu^2} $ for $ 0<\mu<1 $ small enough and $ 0<\eps<1 $ also small enough such that $ 0< M\eps<\frac{R\mu}{2} $, i.e. $ 0<\eps<\frac{R\mu^3}{2} $, and we decompose $ \Omega $ in four disjoint sets $$ \Omega=\left\{ d(x)\geq R\mu\right\}\cup \left\{d(x)< M\eps\right\}\cup\left\{ M\eps\leq d(x)\leq\frac{R\mu}{2}\right\}\cup\left\{\frac{R\mu}{2}<d(x)< R\mu\right\}. $$ We proceed estimating $ \LL_\Omega^\eps(\psi)(x) $ for $ x $ in each of these regions of $ \Omega $.\\
		\begin{figure}[H]\centering
			\includegraphics[height= 5 cm]{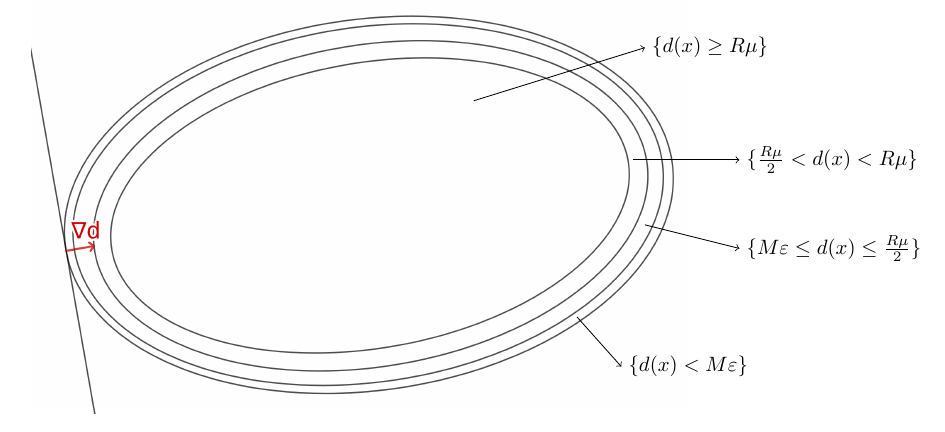}\caption{Decomposition of $ \Omega $.}
		\end{figure}
		For further reference we write
		\begin{equation}\label{L psi}
		\begin{split}
		\LL^\eps_\Omega\left(\psi\right)(x)=&\psi(x)-\int_{\Omega\cap\{d(\eta)<R\mu\}}d\eta K_\eps(\eta-x)\left(1-\frac{\gamma}{1+\deps{\eta}^2}\right)\\&-\int_{\Omega\cap\{d(\eta)\geq R\mu\}}d\eta K_\eps(\eta-x)\left(1-\frac{\gamma}{1+\left(\frac{\mu R}{\eps}\right)^2}\right).
		\end{split}
		\end{equation}
		In order to estimate $ \LL_\Omega^\eps(\psi)(x) $ in the region $ \{d(x)\geq R\mu\} $ we will use the fact that the minimum of supersolutions is again a supersolution. In the region where $ d(x)<M\eps $ we will use the explicit form of the kernel to see that the main contribution has the right sign. Finally, in the region $ \{M\eps\leq d(x)<R\mu\} $ the idea behind the arguments we will present is that $ \LL_\Omega^\eps(\psi)(x) $ can be approximated by $ -\eps^2\Delta\psi $ using Taylor.\\
		
		\noindent\textbf{Step 1: $ \{d(x)\geq R\mu\} $}
		
		First of all we notice that if $ d(x)\geq R\mu $ then $ \LL^\eps_\Omega(\psi)(x)\geq 0 $. Indeed, $\psi(\eta)	\leq\psi(x)=1- \frac{\gamma}{1+\left(\frac{\mu R}{\eps}\right)^2} $ in the first integral of \eqref{L psi} since $ d(\eta)<R\mu $ there. Hence
		\begin{equation}\label{distance big}
		\LL^\eps_\Omega\left(\psi\right)(x)\geq \LL^\eps_\Omega\left(1-\frac{\gamma}{1+\left(\frac{\mu R}{\eps}\right)^2}\right)\geq 0.
		\end{equation} 
		\noindent\textbf{Step 2: $\{ d(x)< M\eps\} $}
		
		We consider now the region $ \{d(x)< M\eps\} $. After a suitable rigid motion we can assume $ 0\in\bnd $ and $ x=(d(x), 0, 0) $. Hence, $ \Omega\subset \RR_+\times \RR^2 $ and
		\begin{equation*}\label{distance small 1}
		\begin{split}
		\int_{\Omega^c} \Kern{\eta-x}\;d\eta&\geq \int_{-\infty}^{-d(x)/\eps}K\left(\eta\right)\;d\eta\geq\int_{-\infty}^{-M}K\left(\eta\right)\;d\eta:=\nu_M>0.
		\end{split}
		\end{equation*}
		$ K $ is as usual the normalized exponential integral.		
		On the other hand, using that $ \frac{1}{1+\deps{x}^2}\leq 1 $ and choosing $ \gamma<\frac{\nu_M}{2} $ we can conclude
		\begin{equation}\label{distance small 3}
		\begin{split}
		\LL^\eps_\Omega\left(\psi\right)(x)=&-\frac{\gamma}{1+\deps{x}^2}+\int_{\Omega^c}d\eta K_\eps(\eta-x)+\gamma\int_{\Omega}d\eta K_\eps(\eta-x)\left(\frac{1}{1+\deps{\eta}^2}\vee\frac{1}{1+\left(\frac{\mu R}{\eps}\right)^2}\right)\\
		\geq&\frac{\nu_M}{2}\geq \frac{\nu_M}{2}e^{-\deps{x}},
		\end{split}
		\end{equation}
		where $ a\vee b=\max(a,b) $.\\
		
		\noindent\textbf{Step 3: $ \left\{M\eps\leq d(x)\leq \frac{R\mu}{2}\right\} $}
		
		We consider now the set $ \left\{M\eps\leq d(x)\leq \frac{R\mu}{2}\right\} $. As first step we plug \eqref{taylor fraction distance} into the right hand side of \eqref{L psi}. To this end we define three integral terms $ J_1 ,\;J_2,\; J_3$ as
		\begin{equation}\label{J1}
		\begin{split}
		J_1=&1-\frac{\gamma}{1+\deps{x}^2}-\int_{\Omega\cap\{d(\eta)<R\mu\}}d\eta\;K_\eps(\eta-x)\left(1-\gamma Q_\eps^{(1)}(x,\eta)\right)\\&-\int_{\Omega\cap\{d(\eta)\geq R\mu\}}d\eta\;K_\eps(\eta-x)\left(1-\frac{\gamma}{1+\frac{R^2\mu^2}{\eps^2}}\right),
		\end{split}
		\end{equation}
		\begin{equation}\label{J_2}
		J_2=\int_{\Omega\cap\{d(\eta)<R\mu\}}d\eta\;K_\eps(\eta-x)\left(\gamma Q_\eps^{(2)}(x,\eta)\right),
		\end{equation}
		\begin{equation}\label{J_3}
		J_3=\int_{\Omega\cap\{d(\eta)<R\mu\}}d\eta\;K_\eps(\eta-x)\left(\gamma Q_\eps^{(3)}(x,\eta)\right).
		\end{equation}
		Hence, we have
		\begin{equation}\label{Js}
		\begin{split}
		\LL_\Omega^\eps(\psi)(x)=J_1+J_2+J_3.
		\end{split}
		\end{equation}
		The main contribution to these terms is due to $ J_2 $. Therefore we start with this term and we show that for $ 0<\mu<1 $ small enough there exists a constant $ \tilde{C}(\mu)>0 $ independent of $ \eps $ such that
		\begin{equation}\label{distance middle 5}
		\begin{split}
		J_2	\geq\frac{\tilde{C}(\mu)\gamma}{\left(1+\deps{x}^2\right)^2}.
		\end{split}
		\end{equation}
		In order to prove this estimate we first notice that
		\begin{equation}\label{distance middle 1}
		\begin{split}
	\frac{4 \deps{x}^2}{\left(1+\deps{x}^2\right)}-1=&3-\frac{4 }{\left(1+\deps{x}^2\right)}\geq 3-\frac{4 }{\left(1+M^2\right)}\geq 0.
		\end{split}
		\end{equation}
		Hence, multiplying this inequality by $ K_\eps(\eta-x)
		\frac{\gamma\left(\nabla d(x)\cdot \left(\eta-x\right)\right)^2}{\eps^2\left(1+\deps{x}^2\right)^2} $ and integrating on $\{d(\eta)<R\mu\} $ we obtain
		\begin{equation}\label{distance middle 2}
		\begin{split}
		\int_{\Omega\cap\{d(\eta)<R\mu\}}&d\eta\;K_\eps(\eta-x)\left(-
		\frac{\gamma\left(\nabla d(x)\cdot \left(\eta-x\right)\right)^2}{\eps^2\left(1+\deps{x}^2\right)^2}+\frac{4\gamma d^2(x)\left(\nabla d(x)\cdot \left(\eta-x\right)\right)^2}{\eps^4\left(1+\deps{x}^2\right)^3}\right)\\
		\geq& \frac{\gamma\left(3-\frac{4}{1+M^2}\right) }{\left(1+\deps{x}^2\right)^2}\int_{B_{M\eps}(x)}d\eta\;K_\eps(\eta-x)\frac{\left(\nabla d(x)\cdot \left(\eta-x\right)\right)^2}{\eps^2}\\
		=& \frac{\gamma\left(3-\frac{4}{1+M^2}\right) }{\left(1+\deps{x}^2\right)^2}\frac{1}{4\pi}\int_0^{2\pi}d\varphi\int_0^\pi d\theta \sin(\theta)\cos^2(\theta)\int_0^M dr\: e^{-r}r^2= \frac{\gamma C(M)\left(3-\frac{4}{1+M^2}\right) }{\left(1+\deps{x}^2\right)^2},
		\end{split}
		\end{equation}
		where used that $ B_{M\eps}(x)\subset\{d(\eta)<R\mu\} $ and we define the constant $ C(M)=\frac{1}{3}\int_0^M dr\;e^{-r}r^2=\frac{1}{3}(2-2e^{-M}-2Me^{-M}-M^2e^{-M}) $ which depends on $ M=\frac{1}{\mu^2} $. Notice that $ C(M)\to \frac{2}{3} $ as $ M\to \infty $ and hence for $ M $ sufficiently large we have also $ C(M)\geq \frac{1}{2} $.\\
		
		In order to conclude the estimate for $ J_2 $ we use the result \eqref{dist hessian} to estimate the Hessian of the distance function, thus
		\begin{equation}\label{distance middle 3}
		\frac{\gamma d(x)\left(\eta-x\right)^\top\nabla^2d(x)\left(\eta-x\right)}{\eps^2\left(1+\deps{x}^2\right)^2}\leq \frac{\gamma\mu\left|\eta-x\right|^2}{\eps^2(1-\mu)\left(1+\deps{x}^2\right)^2}
		\end{equation}
		and we conclude 
		\begin{equation}\label{distance middle 4}
		-\int_{\Omega\cap\{d(\eta)<R\mu\}}d\eta\;K_\eps(\eta-x)\frac{\gamma d(x)\left(\eta-x\right)^\top\nabla^2d(x)\left(\eta-x\right)}{\eps^2\left(1+\deps{x}^2\right)^2}\geq -C \frac{\gamma\mu}{(1-\mu)\left(1+\deps{x}^2\right)^2},
		\end{equation}
		for some constant $ C>0 $.
		
		Combining \eqref{distance middle 2} and \eqref{distance middle 4} we obtain \eqref{distance middle 5}.\\
		
		We proceed now with the term $ J_1 $ in \eqref{J1}. Using the symmetry of the scalar product in $ \RR^3 $ we write
		\begin{equation}\label{distance middle 7}
		\begin{split}
		J_1=&\int_{\Omega^c}d\eta\;K_\eps(\eta-x)\left(1-\gamma Q_\eps^{(1)}(x,\eta)\right)+\int_{\Omega\cap\{d(\eta)\geq R\mu\}}d\eta\;K_\eps(\eta-x)\left(\frac{\gamma}{1+\frac{R^2\mu^2}{\eps^2}}-\gamma Q_\eps^{(1)}(x,\eta)\right)\\
		=&J_{1,1}+J_{1,2}.
		\end{split}
		\end{equation}
		
		We proceed with the estimate for $ J_{1,1} $ in \eqref{distance middle 7}. By means of a suitable coordinate system we can assume again $ 0\in\bnd $ and $ x=(d(x), 0, 0) $.  We notice that if $ \eta\in\left(-\infty, d(x)\right)\times\RR^2 $ then $ \nabla d(x)\cdot \left(\eta-x\right)=\eta_1-d(x)\leq 0 $, while if $ \eta\in\left(d(x),\infty\right)\times\RR^2 $ then $ \nabla d(x)\cdot \left(\eta-x\right)\geq 0 $. Hence, we obtain
		\begin{equation}\label{distance middle 8}
		\begin{split}
		J_{1,1}
		\geq&\int_{\Omega^c\cap\left(-\infty,d(x)\right)\times \RR^2}d\eta\;K_\eps(\eta-x)\left(1-\gamma Q_\eps^{(1)}(x,\eta)\right).\\
		\end{split}
		\end{equation}
		
		We now decompose the set $ \Omega^c\cap\left(\left(-\infty,d(x)\right)\times \RR^2\right)=\left(\left(-\infty,0\right)\times \RR^2\right)\cup\left(\Omega^c\cap\left(\left(0,d(x)\right)\times \RR^2\right)\right)  $. Using that \begin{equation}\label{distance middle identity}
		 \frac{d(x)}{\eps^2\left(1+\deps{x}^2\right)^2}=\frac{1}{d(x)\left(1+\deps{x}^2\right)}-\frac{1}{d(x)\left(1+\deps{x}^2\right)^2}
		\end{equation} and since $ \gamma<\frac{1}{3} $ we have $ 1-\frac{\gamma}{1+\deps{x}^2}>0 $ and therefore we obtain
		\begin{equation}\label{distance middle 9}
		\begin{split}
		\int_{\left(-\infty,0\right)\times \RR^2}&d\eta\;K_\eps(\eta-x)\left(1-\gamma Q_\eps^{(1)}(x,\eta)\right)\geq\int_{\left(-\infty,0\right)\times \RR^2}d\eta\;K_\eps(\eta-x)\frac{2\gamma \nabla d(x)\cdot \left(\eta-x\right)}{\deps{x}\eps \left(1+\deps{x}^2\right)}\\=&-\frac{2\gamma}{\deps{x} \left(1+\deps{x}^2\right)}\int_{\deps{x}}^\infty dz\; K\left(z\right)z
		\geq-\frac{\gamma}{2\deps{x}}\frac{1+\deps{x}}{1+\deps{x}^2}e^{-\deps{x}}\\\geq& -\frac{\gamma C}{M}\frac{1}{\left(1+\deps{x}^2\right)^2},		
		\end{split}
		\end{equation}
		where we also changed variable $ (\deps{x}-z)\mapsto z $, we used the identity \eqref{intK3} for the normalized exponential integral in Proposition \ref{propK1}, we estimated $ \deps{x}\geq M $ and finally we denote by $ C $ the constant such that $ \frac{(1+x^2)^2}{2}e^{-|x|}\leq C $.
		
		Concerning the integral in the set $ \Omega^c\cap\left(\left(0,d(x)\right)\times \RR^2\right) $ we proceed similarly using again \eqref{distance middle identity} and also the fact that if $ z>0 $ then $ z-d(x)>-d(x) $. Hence, we have
		\begin{equation}\label{distance middle 10}
		\begin{split}
		&\int_{\Omega^c\cap\left(\left(0,d(x)\right)\times \RR^2\right)}d\eta\;K_\eps(\eta-x)\left(1-\gamma Q_\eps^{(1)}(x,\eta)\right)\\
		\geq&\int_{\Omega^c\cap\left(\left(0,d(x)\right)\times \RR^2\right)}d\eta\;K_\eps(\eta-x)\left(1-\frac{\gamma}{1+\deps{x}^2}+\frac{2\gamma \nabla d(x)\cdot \left(\eta-x\right)}{\deps{x}\eps \left(1+\deps{x}^2\right)}\right)\\
		=&\int_{\Omega^c\cap\left((0,d(x)\times\RR^2\right)}dz\;K_\eps(\eta-d(x)e_1)\left(1-\frac{\gamma}{1+\deps{x}^2}+\frac{2\gamma(\eta_1-d(x))}{d(x)\left(1+\deps{x}^2\right)}\right)\geq 0
		\end{split}
		\end{equation}
		Hence, for $  M\eps\leq d(x)< R\mu $ and $ \gamma<\frac{1}{3} $ we can summarize
		\begin{equation}\label{distance middle 11}
		J_{1,1}\geq -\frac{\gamma}{\left(1+\deps{x}^2\right)^2}\frac{C}{M}.
		\end{equation}
		\begin{remark}
		Notice that the estimates \eqref{distance middle 8}-\eqref{distance middle 11} are valid in the whole region $ \{M\eps\leq d(x)<R\mu\} $.
		\end{remark}
	
		We still have to consider the integral $ J_{1,2} $ in \eqref{distance middle 7}. We notice that for all $ \eta\in\Omega $ with $ d(\eta)\geq R\mu $ we have on the one hand $ \left|\eta-x\right|\geq \frac{R\mu}{2} $ and on the other hand $ \nabla d(x)\cdot (\eta-x)\geq 0 $ since $ d(\eta)>d(x) $. We recall that $ D:=\diam\left(\Omega\right) $ and that $ \Omega\cap\{d(\eta)\geq R\mu\}\subset B_D(x) $. Therefore, we estimate
		\begin{equation}\label{distance middle 12}
		\begin{split}
		J_{1,2}\geq&-\int_{\Omega\cap\{d(\eta)\geq R\mu\}}d\eta\;K_\eps(\eta-x)\frac{\gamma}{1+\deps{x}^2}
		\geq -\frac{\gamma e^{-\frac{R\mu}{2\eps}}}{1+\deps{x}^2}\int_{B_D(0)}dz\;\frac{1}{4\pi\eps|z|^2}\\\geq&-\gamma\frac{ e^{-\frac{\deps{x}}{2}}}{1+\deps{x}^2}\frac{4D}{R\mu}\geq -\gamma C \frac{D}{R}\frac{\mu}{\left(1+\deps{x}^2\right)^2}
		\end{split}
		\end{equation}
		where we used the well-known estimate $ xe^{-x} \leq e^{-1}$ combined with $ e^{-\frac{R\mu}{4\eps}}\leq e^{-\frac{d(x)}{2\eps}} $ and we denoted by $ C $ the constant such that $ 4x(1+x^2)e^{-\frac{x}{2}}\leq C $ and finally the relation $ M=\frac{1}{\mu^2} $.\\
		
		Finally we estimate the term $ J_3 $ in \eqref{J_3}. Here we have to estimate the integral term containing the error terms $ Q_\eps^{(3)}(x,\eta) $ of the Taylor expansion \eqref{taylor square distance}.
		If $  M\eps\leq d(x)\leq \frac{R\mu}{2} $ and if $ \eps<1 $ we use $ \frac{x}{1+x^2}=\frac{1}{x}-\frac{1}{x\left(1+x^2\right)} $ and we calculate
		\begin{equation}\label{error 3}
		\begin{split}
		\gamma\int_{\Omega\cap\{d(\eta)<R\mu\}}d\eta\;& K_\eps(\eta-x)\left(\frac{d(x)}{\eps^2}\frac{\left|\eta-x\right|^3}{\left(1+\deps{x}^2\right)^2}+\frac{d(x)}{\eps^4}\frac{\left|\eta-x\right|^3}{\left(1+\deps{x}^2\right)^3}\right)\\
		\leq&\int_{\RR^3}d\eta\; \frac{\gamma e^{-|\eta|}}{4\pi}\frac{\left|\eta\right|}{\left(1+\deps{x}^2\right)^2}\left(d(x)\eps+\frac{1}{\frac{d(x)}{\eps}}-\frac{1}{\frac{d(x)}{\eps}\left(1+\deps{x}^2\right)}\right)\\
		\leq& \frac{C \gamma}{\left(1+\deps{x}^2\right)^2}\left(\frac{R\mu}{2}+\mu^2\right).
		\end{split}
		\end{equation}
		Hence, also $ J_3\geq-\frac{C \gamma}{\left(1+\deps{x}^2\right)^2}\left(\frac{R\mu}{2}+\mu^2\right) $.\\
		
		We conclude putting together estimates \eqref{distance middle 5} \eqref{distance middle 7}, \eqref{distance middle 11}, \eqref{distance middle 12} and \eqref{error 3} the existence of a constant $ C(\Omega)>0 $ independent of $ \mu, \gamma, \eps $ such that
		\begin{equation}\label{distance middle 14}
		\LL_\Omega^\eps\left(\psi\right)(x)\geq \frac{\gamma}{\left(1+\deps{x}^2\right)^2}\left[C(M)\left(3-\frac{4}{1+M^2}\right)-C(\Omega)\frac{\mu}{1-\mu}\right].
		\end{equation}
		Choosing $ 0<\mu<1 $ small enough, depending only on $ \Omega $, such that $ C(M)>\frac{1}{3} $ and $ C(\Omega)\frac{\mu}{1-\mu}<\frac{1}{6} $ we obtain
		\begin{equation}\label{distance middle 13}
		\LL_\Omega^\eps\left(\psi\right)(x)\geq\frac{\gamma}{6\left(1+\deps{x}^2\right)^2}\geq Ce^{-\frac{d(x)}{\eps}}
		\end{equation}
		for $ M\eps\leq d(x)\leq\frac{R\mu}{2} $ and some constant $ C $ depending on $ \Omega $, $ R$, $ \gamma $, $ \mu $ but independent of $ \eps $.\\
		
		\noindent\textbf{Step 4: $\left\{ \frac{R\mu}{2}<d(x)< R\mu\right\} $}
		
				
		It remains to calculate the behavior of $ \LL_\Omega^\eps(\psi) $ when $ \frac{R\mu}{2}<d(x)< R\mu $. Here, we show that there exists a constant $ c(R,\mu,\gamma) $ such that $ \LL_\Omega^\eps\left(\psi\right)(x)\geq -c\eps^2 $. We can use several results we obtained in Step 3. We decompose again the operator $ \LL_{\Omega}^\eps(\psi)(x)=J_1+J_2+J_3 $ according to \eqref{Js} using the integral terms defined in \eqref{J1}-\eqref{J_3}.\\
		
		First of all \eqref{distance middle 1} implies
		\begin{equation*}\label{distance middle middle 1}
		\int_{\Omega\cap\{d(\eta)<R\mu\}}d\eta\;K_\eps(\eta-x)\left(-
		\frac{\gamma\left(\nabla d(x)\cdot \left(\eta-x\right)\right)^2}{\eps^2\left(1+\deps{x}^2\right)^2}+\frac{4\gamma d^2(x)\left(\nabla d(x)\cdot \left(\eta-x\right)\right)^2}{\eps^4\left(1+\deps{x}^2\right)^3}\right)\geq 0
		\end{equation*}
		and hence we estimate $ J_2 $ using \eqref{distance middle 3} and \eqref{distance middle 4}
		\begin{equation}\label{distance middle middle 4}
		\begin{split}
		J_2\geq&-\int_{\Omega\cap\{d(\eta)<R\mu\}}d\eta\;K_\eps(\eta-x)\frac{\gamma d(x)\left(\eta-x\right)^\top\nabla^2d(x)\left(\eta-x\right)}{\eps^2\left(1+\deps{x}^2\right)^2}\\
		\geq& -C \frac{\gamma\mu}{(1-\mu)\left(1+\deps{x}^2\right)^2}
		\geq -\frac{8\gamma C}{(1-\mu)R^3}\eps^3,
		\end{split}
		\end{equation}
		where we used $ 1+\deps{x}^2\geq \deps{x}^2\geq \left(\frac{R\mu}{2\eps}\right)^2 $ and $ 0<\eps<\frac{R\mu^3}{2} $.\\
		
		We now proceed to estimate $ J_1 $. To this end we use again the decomposition \eqref{distance middle 7}. The estimate \eqref{distance middle 11} for $ J_{1,1} $ is also valid in the region $ \{\frac{R\mu}{2}<d(x)< R\mu\} $, as we indicated in the remark after \eqref{distance middle 11}.  Hence we have for $ \eps<\frac{R\mu^3}{2} $
		\begin{equation*}\label{distance middle middle 3}
		\begin{split}
		J_{1,1}\geq -\frac{\gamma\mu^2 C}{\left(1+\deps{x}^2\right)^2}
		\geq -\frac{8\gamma C}{R^2}\eps^3.
		\end{split}
		\end{equation*}
		Concerning the term $ J_{1,2} $ we have to argue slightly different than in Step 3. Using now the first inequality in \eqref{distance middle 12} and $ \int_{\RR^3}d\eta K_\eps(\eta-x)=1 $ we compute
		\begin{equation}\label{distance middle middle 2}
		\begin{split}
		J_{1,2}
		\geq &-\int_{\Omega\cap\{d(\eta)\geq R\mu\}}d\eta\;K_\eps(\eta-x)\frac{\gamma}{1+\deps{x}^2}\geq-\frac{\gamma}{1+\deps{x}^2}
		\geq-\frac{4\gamma}{\left(R\mu\right)^2}\eps^2.
		\end{split}
		\end{equation}
		Finally, we estimate $ J_3 $ as defined in \eqref{J_3}. Arguing as in \eqref{error 3} and using $ 1+x^2\geq x^2 $ and $ 0< \eps<\frac{R\mu^3}{2} $ we compute
		\begin{equation}\label{error 4}
		\begin{split}
		\int_{\Omega\cap\{d(\eta)<R\mu\}}d\eta\;& K_\eps(\eta-x)\left(\frac{d(x)}{\eps^2}\frac{\left|\eta-x\right|^3}{\left(1+\deps{x}^2\right)^2}+\frac{d(x)}{\eps^4}\frac{\left|\eta-x\right|^3}{\left(1+\deps{x}^2\right)^3}\right)\\
		\leq&\frac{\gamma \left(d(x)^2+1\right)}{4\pi\left(\deps{x}\right)^5}\int_{\RR^3}d\eta\;e^{-|\eta|} \left|\eta\right|
		\leq\frac{2 \gamma C \left({R^2}+2\right)}{R^3}\eps^2.
		\end{split}
		\end{equation}
		Thus, also $ J_3\geq -\frac{2 \gamma C \left({R^2}+2\right)}{R^3}\eps^2$.
	
		Hence, \eqref{distance middle middle 4},\eqref{distance middle middle 3},\eqref{distance middle middle 2} and \eqref{error 4}  imply the existence of a constant $ c(R,\mu,\gamma)>0 $ independent of $ \eps $ such that
		\begin{equation}\label{distance middle middle 6}
		 \LL_\Omega^\eps\left(\psi\right)(x)\geq -c\eps^2 
		\end{equation}for all $ \frac{R\mu}{2}<d(x)< R\mu $ .\\
		
	We know summarize the results. 
		Equations \eqref{distance big}, \eqref{distance small 3}, \eqref{distance middle 13}, \eqref{distance middle middle 6} imply the claim in \eqref{second supsol 1}. We remark that $ \mu $, $ \gamma $ and $ \eps_1 $ are chosen as follows. First of all $ \mu $ is chosen according to Step 3 as in \eqref{distance middle 14}, then $ \gamma $ is taken according to Step 2 such that $ 0<\gamma<\frac{\nu_M}{2} $ and finally $ \eps_1 $ satisfies $ 0<\eps_1< \frac{R\mu^3}{2} $. This concludes the Lemma \ref{second supsol}.
	\end{proof}
\end{lemma}
Using Lemma \ref{first supsol} and \ref{second supsol} we can now prove Theorem \ref{interior supsol}.
\begin{proof}[(Proof of Theorem \ref{interior supsol})]
	Let $ C_1 $ be the constant defined in	Lemma \ref{first supsol} and let $ \gamma,\;\mu,\;C_0,\;c $ be as in Lemma \ref{second supsol}. We define $ C_2:=\frac{1}{C_0} $ and $ C_3:=\frac{C_0+c}{2C_0}>\frac{1}{2} $. Notice that all these constants are independent of $ \eps $. Hence, Lemma \ref{first supsol} and \ref{second supsol} imply
	\begin{equation}\label{interior supsol 1}
	\LL_\Omega^\eps\left(\Phi^\eps\right)(x)\geq\Arrowvert g\Arrowvert_1 \begin{cases}
	e^{-\frac{d(x)}{\eps}}+2C_3 \eps^2 & 0<d(x)\leq \frac{R\mu}{2},\\
	\eps^2 & \frac{R\mu}{2}<d(x)< R\mu,\\
	2C_3\eps^2& d(x)\geq R\mu,\\
	\end{cases}\geq \Arrowvert g\Arrowvert_1\begin{cases}
	 e^{-\frac{d(x)}{\eps}} & 0<d(x)\leq \frac{R\mu}{2},\\
\eps^2 & \frac{R\mu}{2}<d(x)< R\mu,\\
	\eps^2& d(x)\geq R\mu.\\
	\end{cases}
	\end{equation}
	We define now $ \eps_0:=\min\left\{1,\;a, \eps_1\right\} $ with $ a $ such that $ 2a\ln(\frac{1}{a})<\frac{R\mu}{2} $ and $ \eps_1>0 $ as in Lemma \ref{second supsol}. Then $ \eps^2\geq e^{-\frac{R\mu}{2\eps}}\geq e^{-\frac{d(x)}{\eps}} $ for all $ d(x)>\frac{R\mu}{2} $. 
	
	We now apply the maximum principle in Theorem \ref{max principle omega} to the function $ \Phi^\eps-u^\eps $. This function satisfies the continuity and boundedness assumption. Indeed, for any $ \eps>0 $ the function $ u^\eps $ is continuous and bounded as we have seen at the beginning of Section 4.1. Moreover, by construction $ \Phi^\eps $ is continuous and it is easy to see that it is even uniformly bounded since
	\begin{equation*}\label{interior supsol 2}
	0\leq \Phi^\eps(x)\leq \Arrowvert g\Arrowvert_1\left(2C_3C_1+C_2\right).
	\end{equation*}
We also have
	\begin{equation*}\label{interior supsol 3}
	\LL_\Omega^\eps\left(\Phi^\eps-u^\eps\right)(x)\geq \Arrowvert g\Arrowvert_1 e^{-\frac{d(x)}{\eps}}-\intup d\nu \int_{n\cdot N_{x_\Omega}<0}dn\; g_\nu(n)e^{-\frac{\left|x-x_\Omega(x,n)\right|}{\eps}}\geq 0,
	\end{equation*}
	since $ \left|x-x_\Omega(x,n)\right|\geq d(x) $. Hence, Theorem \ref{max principle omega} implies that $ \Phi^\eps-u^\eps\geq 0 $ and thus
	\begin{equation*}\label{interior supsol 4}
	0\leq u^\eps\leq \Phi^\eps\leq \tilde{C}<\infty
	\end{equation*}
	uniformly in $ \eps $ and $ x\in\Omega $.
\end{proof}
\subsection{Estimates of $ u^\eps-\ou $ near the boundary $ \bnd $}

 In this subsection we will prove that for each point $ p\in\bnd $ the function $ \ou $ defined in \eqref{bvpgreyboundary1} is a good approximation of $ u^\eps $ in a neighborhood of size close to $ \eps^{\frac{1}{2}} $. Notice that this neighborhood is much greater than the region of size $ \eps $. We will do it by means of the maximum principle in Theorem \ref{max principle omega}. Now we start estimating the action of the operator $ \LL_\Omega^\eps $ on $ \ou-u^\eps $.
\begin{lemma}\label{estimate op ou-u eps}
Let $ p\in\bnd $ and let $ \Rot_p $ be the isometry defined in \eqref{rotp}. Then the following holds for $ x\in\Omega $, $ \delta>0 $ sufficiently small and independent of $ \eps $ and a suitable $ 0<A<1 $ and constant $ C>0 $
\begin{equation}\label{op ou-u eps 1}
\left|\LL_\Omega^\eps\left(\ou\left(\frac{\Rot_p(\cdot)\cdot e_1}{\eps},p\right)-u^\eps\right)(x)\right|\leq C e^{-\frac{Ad(x)}{\eps}}\begin{cases}\eps^\delta & \text{ if }|x-p|<\eptw,\\ 1& \text{ if }|x-p|\geq \eptw.
\end{cases}
\end{equation}
\begin{proof}
Let us denote by $ \Pi_p $ the half space $ \Pi_p:=\Rot_p^{-1}\left(\RR_+\times \RR^2\right) $. Then the function $ \oU_\eps(x,p):=\ou\left(\frac{\Rot_p(x)\cdot e_1}{\eps},p\right) $ is a continuous bounded function which maps $ \Pi_p\times \bnd $ to $ \RR_+ $. Notice that $ \oU_\eps(x,p) $ is the solution to the planar equation \eqref{bvpgreyboundary1} before rescaling and rotating. Our plan is to approximate $ \LL_{\Omega}^\eps\left(\oU_\eps\right) $ by $ \LL_{\Pi_p}^\eps\left(\oU_\eps\right) $. Let $ x\in\Pi_p $ and $ p\in\Omega $. Using the definition of $ \ou $ in \eqref{bvpgreyboundary1} we can compute
\begin{equation*}\label{op ou-u eps 2}
\begin{split}
\int_0^\infty d\eta\;& \PlanarKern{\eta-\frac{\Rot_p(x)\cdot e_1}{\eps}} \ou(\eta,p)= \int_{\RR_+\times\RR^2} d\eta\; \frac{e^{-\left|\eta-\frac{\Rot_p(x)}{\eps}\right|}}{4\pi \left|\eta-\frac{\Rot_p(x)}{\eps}\right|^2}\ou\left(\eta_1,p\right)\\&= \int_{\RR_+\times\RR^2} d\eta\; \frac{e^{-\frac{\left|\eta-\Rot_p(x)\right|}{\eps}}}{4\pi\eps\left|\eta-\Rot_p(x)\right|^2}\ou\left(\frac{\eta_1}{\eps},p\right)=\int_{\Pi_p}d\eta\; K_\eps(\eta-x) \oU_\eps(\eta,p),
\end{split}
\end{equation*}
where we used in the first equality the translation invariance of the integral with respect to the second and third variable, the definition of the planar kernel and the definition of $ y $. For the second equality we used the change of variables $ \tilde{\eta}=\eps\eta $ and in the last identity the change of variables $ \tilde{\eta}=\Rot_p^{-1}(\eta) $ gives the result. In order to write the value of $ \LL_{\Pi_p}^\eps\left(\oU_\eps\right) $ we use once again equation \eqref{bvpgreyboundary1} and we define $ x_{\Pi_p}(x,n) $ as the point on the boundary of $ \Pi_p $ with $ \frac{x-x_{\Pi_p}(x,n)}{\left|x-x_{\Pi_p}(x,n)\right|}=n $, i.e. $ x=x_{\Pi_p}(x,n)+\left|x-x_{\Pi_p}(x,n)\right|n $ if $ n\cdot N_p<0 $. By construction we see that $ \frac{\Rot_p(x)\cdot e_1}{|n\cdot N_p|}=\left|x-x_{\Pi_p}(x,n)\right| $. Hence, 
\begin{equation*}\label{op ou-u eps 3}
\LL_{\Pi_p}^\eps\left(\oU_\eps(\cdot,p)\right)(x) =\intnu\int_{n\cdot N_{p}<0} dn\; g_\nu(n)e^{-\frac{\left|x-x_{\Pi_p}(x,n)\right|}{\eps}}.
\end{equation*}
We will hence estimate the two integrals terms on the right hand side of the following equation
\begin{equation}\label{op ou-u eps 4}
\begin{split}
\left|\LL_\Omega^\eps\left(\oU_\eps(\cdot,p)-u^\eps\right)(x)\right|\leq & \int_{\Pi_p\setminus\Omega}d\eta  K_\eps(\eta-x)\oU_\eps(\cdot,p)\\&+\intnu\intS\; g_\nu(n)\left|e^{-\frac{\left|x-x_{\Pi_p}(x,n)\right|}{\eps}}-e^{-\frac{\left|x-x_{\Omega}(x,n)\right|}{\eps}}\right|=S_1+S_2,
\end{split}
\end{equation}
where we put $ \left|x-x_{\Pi_p}(x,n)\right|=\infty $ if $ n\cdot N_p\geq0 $. \\

\noindent\textbf{Step 1: estimate of $ S_1 $.}

It always possible to estimate $ S_1 $ by $ e^{-\frac{d(x)}{\eps}} $, indeed using $ B_{d(x)}(x)\subset \Omega $ we compute
\begin{equation}\label{op ou-u eps 5}
 \int_{\Pi_p\setminus\Omega}d\eta  K_\eps(\eta-x)\oU_\eps(\eta,p)\leq  \int_{B^c_{d(x)}(x)}d\eta  K_\eps(\eta-x)\oU_\eps(\eta,p)\leq C e^{-\frac{d(x)}{\eps}},
\end{equation}
where $ C>0 $ is the uniform bound on $ \ou $ that we have obtained in Lemma \ref{unifbounded}. 

Our goal is to obtain a better estimate for the region $ |x-p|<\eptw $ (cf. \eqref{op ou-u eps 1}). Therefore we will now assume $ |x-p|<\eptw $ and since $ d(x)<|x-p| $ we can also assume $ d(x)<\eptw $.
\begin{figure}[H]\centering
\includegraphics[height=6 cm]{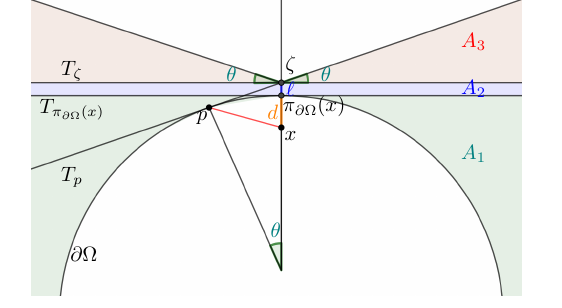}\caption{Decomposition of $ \Pi_p\setminus \Omega $.}
\end{figure}
Let $ p\in\bnd $ and $ \Pi_p $ the half space defined at the beginning of the proof. Let $ x\in\Omega $ with $ |x-p|<\eptw $ for $ \eps>0 $ as small as needed. Let $ \pi_{\bnd}(x)\in\bnd $ be the projection of $ x $ on the boundary as in \eqref{proj}. Then $ x=\pi_{\bnd}(x)-d(x)N_{\pi_{\bnd}(x)} $. We denote further by $ \theta(x) $ the angle between the normal vectors $ N_p $ and $ N_{\pi_{\bnd}(x)} $. Let $ T_p $ and $ T_{\pi_{\bnd}(x)} $ the tangent planes to $ \bnd $ containing $ p $ respectively $ \pi_{\bnd}(x) $. We define also $ \zeta=T_p\cap\left\{x+tN_{\pi_{\bnd}(x)}:t\geq 0\right\} $ and $ T_\zeta $ the plane orthogonal to $ N_{\pi_{\bnd}(x)} $ containing $ \zeta $. We denote by $ \ell(x)=\left|\zeta-N_{\pi_{\bnd}(x)}\right| $ the distance between $ T_\zeta $ and $ T_{\pi_{\bnd}(x)} $.

We decompose now $ \Pi_p\setminus\Omega $ in three larger regions, i.e. $ \Pi_p\setminus\Omega\subset A_1\cup A_2\cup A_3 $. We define $ A_1:= \Pi_{\pi_{\bnd}(x)}\setminus \Omega $, $ A_2 $ is the region containing all points between the planes $ T_{\pi_{\bnd}(x)} $ and $ T_\zeta $ and finally $ A_3:= \left\{\zeta\in\RR^3:\;0\leq N_{\pi_{\bnd}(x)}\cdot (\eta-\zeta)\leq \sin(\theta(x))|\eta-\zeta|\right\} $. The choice of these regions has been made in order to obtain integrals that are symmetric and easier to compute.

By standard differential geometry arguments we know that $ \theta(x)\leq \frac{1}{R}\eptw $ and that for some constant $ C(\Omega)>0 $ also $ \theta(x)\leq C(\Omega)|x-p| $. Moreover, denoting by $ \rho $ the radius of curvature of the curve given by the intersection of $ \bnd $ with the plane uniquely defined by $ N_p $, $ N_{\pi_{\bnd}(x)} $ and containing $ p $ we obtain 
\begin{equation*}\label{ell}
\ell(x)\leq C \rho \theta(x)^2\leq C(\Omega)|x-p|^2.
\end{equation*}
In case $ \theta=0 $ and hence $ N_{\pi_{\bnd}(x)}=N_p $ we only consider $ A_1 $.\\

\noindent\textbf{Region $ A_1 $}
 
We begin with the integral on the set $ A_1 $. Elementary differential geometry implies that $ A_1\subseteq \Pi_{\pi_{\bnd}(x)}\cap B_{\frac{R}{2}}^c\left(\pi_{\bnd}(x)-\frac{R}{2}N_{\pi_{\bnd}(x)}\right) $. Let us also denote by $ B^{c,+}_r(0):=B^c_r(0)\cap\left(\RR_+\times\RR^2\right) $. Moreover, using that $ \left\{(\eta_1,\tilde{\eta})\in[0,\frac{R}{2}]\times \RR^2:\;\left(\eta_1-\frac{R}{2}\right)^2+|\tilde{\eta}|^2\geq \frac{R^2}{4}\right\}\subseteq \left\{(\eta_1,\tilde{\eta})\in\RR_+\times \RR^2:\;|\tilde{\eta}|^2\geq \frac{R}{2}\eta_1\right\}$. \linebreak Hence, with a suitable change of variables we compute
\begin{equation}\label{A1}
\begin{split}
 \int_{A_1}&d\eta K_\eps(\eta-x)\leq \int_{\RR^2}d\tilde{\eta}
\int_{0}^{2\frac{|\tilde{\eta}|^2}{R}}d\eta_1\;\frac{e^{-\frac{\sqrt{|d(x)-\eta_1|^2+|\tilde{\eta}|^2}}{\eps}}}{4\pi \eps \left(|d(x)-\eta_1|^2+|\tilde{\eta}|^2\right)}+ \int_{B^{c,+}_{\frac{R}{2}}(0)}d\eta\; K_\eps\left(\eta-\left(d(x)-\frac{R}{2}\right)e_1\right)\\
\leq& \int_{B_{d(x)}(0)}d\tilde{\eta}
\int_{0}^{\frac{2|\tilde{\eta}|^2}{R}}d\eta_1\;\frac{e^{-\frac{\sqrt{|d(x)-\eta_1|^2+|\tilde{\eta}|^2}}{\eps}}}{4\pi \eps \left(|d(x)-\eta_1|^2+|\tilde{\eta}|^2\right)}+\int_{B_{d(x)}^c(0)}d\tilde{\eta}
\int_{0}^{2\frac{|\tilde{\eta}|^2}{R}}d\eta_1\;\frac{e^{-\frac{\sqrt{|d(x)-\eta_1|^2+|\tilde{\eta}|^2}}{\eps}}}{4\pi \eps \left(|d(x)-\eta_1|^2+|\tilde{\eta}|^2\right)}\\&+\int_{B^{c,+}_{\frac{R}{2}}(0)}d\eta\;\frac{e^{-\frac{|\eta|}{\eps}}}{\pi\eps R^2}\\
\leq& \int_{B_{d(x)}(0)}d\tilde{\eta} \frac{4|\tilde{\eta}|^2e^{-\frac{d(x)}{2\eps}}}{\pi \eps R d^2(x)}+\int_{B_{d(x)}^c(0)}d\tilde{\eta}  \frac{e^{-\frac{|\tilde{\eta}|}{\eps}}}{2\pi R \eps}+\frac{4\eps^2C}{R^2}\int_{\frac{R}{4\eps}}^{\infty} dr\;e^{-r}\\
\leq&\frac{2d^2(x)}{R\eps}e^{-\frac{d(x)}{2\eps}}+ C \eps e^{-\frac{d(x)}{2\eps}}+\frac{4\eps^2C}{R^2}e^{-\frac{R}{4\eps}}
\leq C(\Omega)\eps e^{-\frac{d(x)}{4\eps}}.\end{split}
\end{equation}
We also used that if $ (\eta_1,\tilde{\eta})\in \left[0,\frac{2|\tilde{\eta}|^2}{R} \right]\times B_{d(x)}(0) $ we can estimate
	\begin{equation}\label{op ou-u eps 8}
	|d(x)-\eta_1|=d(x)-\eta\geq d(x)\left(1-\frac{2\eptw}{R}\right)\geq \frac{d(x)}{2},
	\end{equation}
	since $ d(x)<\eptw $ and we combined \eqref{op ou-u eps 8} with $ |d(x)-\eta_1|^2+|\tilde{\eta}|^2\geq |d(x)-\eta_1|^2 $.
	If  $ (\eta_1,\tilde{\eta})\in \left[0,\frac{2|\tilde{\eta}|^2}{R} \right]\times B^c_{d(x)}(0) $ then we can estimate $ |d(x)-\eta_1|^2+|\tilde{\eta}|^2\geq|\tilde{\eta}|^2 $. Moreover, if $ \eta\in B^{c,+}_{\frac{R}{2}}(0)  $ then $ \eta_1+\frac{R}{2}-d(x)\geq \eta_1 $ and $ \left|\eta-\left(d(x)-\frac{R}{2}\right)e_1\right|\geq \frac{R}{2} $. In the third inequality we computed the first two integrals on the 2 dimensional balls using also the fact that there exists a constant $ C>0 $ such that $ e^{-x}x\leq Ce^{-\frac{x}{2}} $ if $ x\geq0 $ and the last integral holds by the existence of a constant $ C>0 $ such that $ x^2 e^{-\frac{x}{2}}\leq C$ for $ x\geq 0 $. For the last estimate we notice first of all that $ R\geq d(x) $ and we consider two different cases. If $ d(x)\leq \eps $ the result follows from the fact that $ \frac{d^2(x)}{2R\eps}\leq \frac{\eps}{2R} $. If $ d(x)\geq \eps $ we use the well-known estimate $ e^{-x}x^2\leq Ce^{-\frac{x}{2}}  $ for $ x\geq 0 $.\\

\noindent\textbf{Region $ A_2 $}
	
	We proceed with the integral on $ A_2 $. We compute using a change of variables
	\begin{equation}\label{op ou-u eps 9}
	\begin{split}
	&\int_{A_2}d\eta\; K_\eps(\eta-x)=\int_{\RR^2} d\tilde{\eta}\int_{d(x)}^{d(x)+\ell(x)} d\eta\;\frac{e^{-\frac{\sqrt{\eta^2+|\tilde{\eta}|^2}}{\eps}}}{4\pi \eps \left(\eta^2+|\tilde{\eta}|^2\right)}
	=\int_{\RR^2} d\tilde{\eta}\int_{\frac{d(x)}{\eps}}^{\frac{d(x)+\ell(x)}{\eps}} d\eta\;\frac{e^{-\sqrt{\eta^2+|\tilde{\eta}|^2}}}{4\pi \left(\eta^2+|\tilde{\eta}|^2\right)}\\
	&= \int_{\frac{d(x)}{\eps}}^{\frac{d(x)+\ell(x)}{\eps}} d\eta\;\PlanarKern{\eta},
	\end{split}
	\end{equation}
	where we rescaled by $ \eps $ and we used the definition of the normalized exponential integral $K$ as in \eqref{kernel}. The estimate of the last integral depends on the values for $ d(x) $ and $ \ell(x) $. We recall that $ d(x)<\eptw $ and that $ \ell(x)\leq C(\Omega)\eps^{1+4\delta} $. Proposition \ref{propK} implies also the following estimate for the normalized exponential integral
	\begin{equation}\label{estimateK}
	K(\eta)\leq C\begin{cases}
	1+|\ln(\eta)|&\text{ if }0\leq \eta\leq 2,\\e^{-\eta}&\text{ if }\eta\geq 1.	\end{cases}
	\end{equation}
	for some constant $ C>0 $. Let us assume first $ d(x)\geq \eps $. Then \eqref{op ou-u eps 9} and \eqref{estimateK} imply
	\begin{equation}\label{A2 1}
	\int_{A_2}d\eta\; K_\eps(\eta-x)\leq C \int_{\frac{d(x)}{\eps}}^{\frac{d(x)+\ell(x)}{\eps}} e^{-\eta}\;d\eta\leq C(\Omega)\eps^{4\delta} e^{-\frac{d(x)}{\eps}}.	\end{equation}
	Let us assume now $ d(x)<\eps $. If $ \ell(x)<d(x) $ we can use the monotonicity of the logarithm together with estimate \eqref{estimateK}. Thus,
	\begin{equation}\label{A2 2}
\begin{split}
	\int_{A_2}d\eta\; K_\eps(\eta-x)\leq& C \int_{\frac{d(x)}{\eps}}^{\frac{d(x)+\ell(x)}{\eps}} (1+|\ln(\eta)|)\;d\eta\leq C\left(\frac{\ell(x)}{\eps}+\frac{\ell(x)}{\eps}\left|\ln\left(\frac{d(x)}{\eps}\right)\right|\right)\\\leq& C(\eps^4+\eps^{\delta})\leq C\eps^\delta e^{-\frac{d(x)}{\eps}},
\end{split}
	\end{equation}
	where we used the estimates $ \sqrt{x}\left|\ln(x)\right|\leq \frac{2}{e}\leq 1 $ for all $ x\in[0,1] $ and $ e^{-1}\leq e^{-\frac{d(x)}{\eps}}$.
	
	If $ \ell(x)\geq d(x) $ we argue similarly as in \eqref{A2 2} using also $ (d(x), d(x)+\ell(x))\subset(0,2\ell(x)) $ and we conclude
	\begin{equation}\label{A2 3}
	\begin{split}
	\int_{A_2}d\eta\; K_\eps(\eta-x)\leq& C \int_{0}^{\frac{2\ell(x)}{\eps}} (1+|\ln(\eta)|)\;d\eta\leq C\left(\frac{\ell(x)}{\eps}+\frac{2\ell(x)}{\eps}\left|\ln\left(\frac{2\ell(x)}{\eps}\right)\right|\right)\\\leq& C(\eps^4+\eps^{2\delta})\leq C\eps^{2\delta} e^{-\frac{d(x)}{\eps}}.
	\end{split}
	\end{equation}
\noindent\textbf{Region $ A_3 $}

We are now ready for the estimate of the integral on the set $ A_3 $. We recall that for some constant $ C(\Omega) $ we can estimate $ \theta(x)\leq C(\Omega)|x-p|<C(\Omega)\eptw $. Arguing similarly as in \eqref{A1} we compute using $ \tan(\theta(x))\leq 2\theta(x) $ and a suitable change of variables
\begin{equation}\label{A3}
\begin{split}
\int_{A_3}&d\eta\; K_\eps(\eta-x)= \int_{\RR^2}d\tilde{\eta}
\int_{d(x)+\ell(x)}^{d(x)+\ell(x)+2\theta |\tilde{\eta}|}d\eta_1\;\frac{e^{-\frac{\sqrt{\eta_1^2+|\tilde{\eta}|^2}}{\eps}}}{4\pi \eps \left(\eta_1^2+|\tilde{\eta}|^2\right)}\\
=& \int_{B_{d(x)+\ell(x)}(0)}d\tilde{\eta}
\int_{d(x)+\ell(x)}^{d(x)+\ell(x)+2\theta |\tilde{\eta}|}d\eta_1\;\frac{e^{-\frac{\sqrt{\eta_1^2+|\tilde{\eta}|^2}}{\eps}}}{4\pi \eps \left(\eta^2+|\tilde{\eta}|^2\right)}+\int_{B_{d(x)+\ell(x)}^c(0)}d\tilde{\eta}
\int_{d(x)+\ell(x)}^{d(x)+\ell(x)+2\theta |\tilde{\eta}|}d\eta_1\;\frac{e^{-\frac{\sqrt{\eta^2+|\tilde{\eta}|^2}}{\eps}}}{4\pi \eps \left(\eta^2+|\tilde{\eta}|^2\right)}\\
\leq &\int_{B_{d(x)+\ell(x)}(0)}d\tilde{\eta} \frac{\theta|\tilde{\eta}|e^{-\frac{d(x)+\ell(x)}{\eps}}}{2\pi \eps \left(d(x)+\ell(x)\right)^2 }+\int_{B_{d(x)+\ell(x)}^c(0)}d\tilde{\eta}  \frac{\theta e^{-\frac{|\tilde{\eta}|}{\eps}}}{2\pi |\tilde{\eta}| \eps}\\
\leq& \theta\frac{d(x)+\ell(x)}{3\eps}e^{-\frac{d(x)+\ell(x)}{\eps}}+ 2\theta e^{-\frac{d(x)+\ell(x)}{\eps}}
\leq C\theta e^{-\frac{d(x)}{2\eps}}\leq C(\Omega) \eptw e^{-\frac{d(x)}{2\eps}} ,\end{split}
\end{equation}
where we used in the first inequality that $ \eta_1^2+|\tilde{\eta}|^2\geq \eta^2\geq \left(d(x)+\ell(x)\right)^2 $ and also that $ \eta_1^2+|\tilde{\eta}|^2\geq |\tilde{\eta}|^2 $ and the well-know estimate $ |x|e^{-\frac{|x|}{2}}\leq 1 $.\\

\noindent\textbf{Summarizing: estimate of $ S_1 $}

Since Lemma \ref{unifbounded} implies $ \oU_\eps\leq C(\Omega, g_\nu) $, then estimates \eqref{op ou-u eps 5}, \eqref{A1}, \eqref{A2 1}, \eqref{A2 2}, \eqref{A2 3} and \eqref{A3} yield
the existence of a constant $ C>0 $ independent of $ \eps,x,p,\delta $ such that 
\begin{equation}\label{Pip-omega}
\int_{\Pi_p\setminus\Omega}d\eta K_\eps(\eta-x)\oU_\eps(\eta,p)\leq C\begin{cases}

\eps^{\delta}e^{-\frac{d(x)}{4\eps}} &|x-p|<\eptw\\e^{-\frac{d(x)}{4\eps}}&|x-p|\geq\eptw
\end{cases}.
\end{equation}
\noindent\textbf{Step 2: estimate of $ S_2 $.}

In order to end the proof for this lemma we now estimate the integral term $ S_2 $ of \eqref{op ou-u eps 4}. If $  |x-p|\geq\eptw $ since $ \left|x-x_{\Pi_p}(x,n)\right|\geq \left|x-x_{\Omega}(x,n)\right|\geq d(x) $ we have the estimate $ S_2\leq 8\pi \Arrowvert g\Arrowvert_{\infty} e^{-\frac{d(x)}{\eps}} $. We now assume $ |x-p|< \eptw $. As before this implies $ d(x)<\eptw $. In order to estimate $ S_2 $ we will divide the integral on $ \Ss^2 $ in three integrals, which will be estimated using different approaches. \\

Figure 6 represents the decomposition we are going to consider.
We denote $ \theta_1 $ and $ \theta_2 $ the angles given by $ \tan(\theta_1)=\frac{\eptw}{\eps^{\frac{1}{2}+\delta}}=\eps^\delta $ and $ \tan(\theta_2)=2\eps^{\frac{1}{2}} $ and we denote by $ \theta(n) $ the angle between $ -n $ and $ N_p $, i.e. $ \theta(n)=\arg(\cos(-n\cdot N_p)) $. We decompose the sphere in three different regions $ \Ss^2=\mathcal{U}_1\cup \mathcal{U}_2\cup \mathcal{U}_3 $, were we define
\begin{equation*}\label{U1}
\mathcal{U}_1:=\left\{n\in\Ss^2:\;-n\cdot N_p> \sin(\theta_1) \right\},
\end{equation*}
\begin{equation*}\label{U2}
\mathcal{U}_2:=\left\{n\in\Ss^2:\;n\cdot N_p>\sin(\theta_2) \right\}
\end{equation*}
and \begin{equation*}\label{U3}
\mathcal{U}_3:=\left\{n\in\Ss^2:\; -\sin(\theta_1)\leq n\cdot N_p\leq \sin(\theta_2) \right\}.
\end{equation*}
\begin{figure}[H]
	\centering
	\includegraphics[height=5cm]{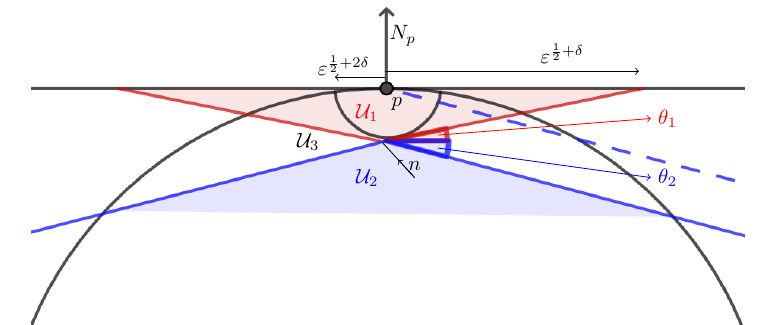}\vspace{-0.1cm}\caption{Decomposition of $ S_2 $.}
\end{figure}
\vspace{-0.1cm}\noindent\textbf{Region $ \mathcal{U}_1 $}

Let us consider $ n\in\mathcal{U}_1 $ and $ x=(y_1,y_2,y_3)^\top\in\Omega $ with $ |x-p|<\eptw $. Let us denote by $ \overline{x}_\Omega $ the intersection point of the axis-symmetric  paraboloid with curvature $ \kappa=\frac{1}{R} $ approximating from the inside of the domain $ \Omega $ the boundary $ \bnd $ in a neighborhood of the point p with the line connecting $ x $ and $ x_\Omega $. It is not difficult to see that $ |x_{\Pi_p}-x_\Omega|\leq |x_{\Pi_p}-\overline{x}_\Omega| $. Without loss of generality we can assume $ n=(n_1,n_2,0) $ with $ n_2<0 $ and as usual $ \Pi_p=\RR_+\times \RR^2 $. Hence, $ x_{\Pi_p}(x,n)= (0, \sigma,y_3) $  and $ \overline{x}_\Omega(x,n)= (x_1,\tilde{x}, y_3) $ and $ |x|<\eptw $. We see then that  $ \sigma\geq\tilde{x}  $ and $ |\sigma|\leq \frac{|x_1|}{ \tan(\theta_1)}\leq \eps^{\frac{1}{2}+\delta}  $. Using the curvature of the boundary we know that the point $ \overline{x}_\Omega $ satisfies the following system of equations 
\begin{equation*}\label{op ou-u eps 17}
\begin{cases}
x_1=\frac{2}{R}\left(\tilde{x}^2+y_3^2\right),\\x_1=\frac{\sigma-\tilde{x}}{\tan(\theta)}.
\end{cases}
\end{equation*}
Hence we calculate
\begin{equation*}\label{op ou-u eps 18}
\begin{split}
\left| x_{\Pi_p}-x_\Omega\right|\leq |x_{\Pi_p}-\overline{x}_\Omega|=\sqrt{|x_1|^2+\left(\sigma-\tilde{x}\right)^2}=\frac{\sigma-\tilde{x}}{\sin(\theta)}=\frac{2\left(\tilde{x}^2+y_3^2\right)}{R\cos(\theta)}\leq \frac{\eps^{1+2\delta}C}{R\eps^\delta }=C(\Omega)\eps^{1+\delta}.
\end{split}
\end{equation*}
Where we used that for $ 0<\eps<1 $ sufficiently small also $ \tan{\theta_1}\approx\sin{\theta_1} $. Thus, we estimate
\begin{equation}\label{op ou-u eps 19}
\begin{split}
&\intnu\int_{\mathcal{U}_1}dn\;g_\nu(n)\left|e^{-\frac{\left|x-x_{\Pi_p}(x,n)\right|}{\eps}}-e^{-\frac{\left|x-x_{\Omega}(x,n)\right|}{\eps}}\right|\\\leq& \Arrowvert g\Arrowvert_{\infty} \int_{\mathcal{U}_1}dn\; e^{-\frac{\left|x-x_{\Omega}(x,n)\right|}{\eps}}\left|1-e^{-\frac{\left|x-x_{\Pi_p}(x,n)\right|-\left|x-x_{\Omega}(x,n)\right|}{\eps}}\right|\\
\leq & 4\pi \Arrowvert g\Arrowvert_{\infty} e^{-\frac{d(x)}{\eps}}\frac{\left|x-x_{\Pi_p}(x,n)\right|-\left|x-x_{\Omega}(x,n)\right|}{\eps}=4\pi \Arrowvert g\Arrowvert_{\infty} e^{-\frac{d(x)}{\eps}}\frac{\left|x_\Omega-x_{\Pi_p}\right|}{\eps}\\\leq& 4\pi \Arrowvert g\Arrowvert_{\infty}C(\Omega)\eps^\delta e^{-\frac{d(x)}{\eps}},
\end{split}
\end{equation}
where we used that $ x,x_{\Pi_p}, x_\Omega $ lie all on the same line.\\

\noindent\textbf{Region $ \mathcal{U}_2 $}

Let us consider $ n\in\mathcal{U}_2 $ and $ x\in\Omega $ with $ |x-p|<\eptw $.
We see first of all that $ n\cdot N_p\geq \sin(\theta_2)\geq 0 $. Thus, by definition $ e^{-\frac{\left|x-x_{\Pi_p}(x,n)\right|}{\eps}}=0 $. In this case we have that $ |x-x_\Omega(x,n)|\geq |x-x_\Omega(x,\tilde{n})| $, where $ \tilde{n}\cdot N_p=\sin(\theta_2) $. We denote by $ Q\in\bnd $ the intersection of the line $ \{x+tN_p:t>0\} $ and the boundary $ \bnd $. As usual $ N_Q $ is the other normal at $ Q\in\bnd $. Since $ |x-p|<\eptw $ also $ |p-Q|<\eptw $ and hence there exists a constant $ C>0 $ such that $ \theta_{pQ}<C\eptw $, where $ \theta_{pQ} $ is the angle between $ N_p $ and $ N_Q $. Let us also denote by $ \tilde{\theta} $ the angle such that $ \tilde{n}\cdot N_Q=\sin(\tilde{\theta}) $. By a geometrical argument on the sphere it is not difficult to see that choosing $ \eps $ sufficiently small, i.e. $ 0<\eps<\min\left(\frac{3}{24}, (4C)^{-\frac{1}{2\delta}}\right) $, we have $ \tilde{\theta}\geq\theta_2-\theta_{pQ}\geq \frac{3}{2}\eps^{\frac{1}{2}} $. Choosing a suitable coordinate system we can assume $ Q=(0,0,0) $, $ N_Q=-e_1 $ and $ \tilde{n}=(-\sin(\tilde{\theta}),-\cos(\tilde{\theta}),0) $. Let us denote by $ \overline{x}_\Omega(x,\omega) $ the intersection point between the line $ \{x-t\omega:\; t>0\} $ and the axis-simmetric paraboloid with curvature $ \kappa=\frac{1}{R} $ inside $ \Omega $ tangent to $ \bnd $ at $ Q $. Then, since now $ x_\Omega(x,\tilde{n}) $ lies outside this paraboloid we obtain $ |x-x_\Omega(x,\tilde{n})|\geq|x-\overline{x}_\Omega(x,\tilde{n})| $. Moreover, notice that since the angle between the axis $ -N_Q $ and the vector $ x-Q $ is given by $ \theta_{pQ}<C\eptw $ we see that $ x $ lies inside the paraboloid choosing $ \eps $ sufficiently small.

For $ \eps>0 $ small enough we also see that $ \sin(\tilde{\theta})\geq \eps^{\frac{1}{2}} $. Hence, it is also true that for $ \tilde{N}=\left(-\eps^{\frac{1}{2}}, -\sqrt{1-\eps},0\right) $ we have $ |x-\overline{x}_\Omega(x,\tilde{n})|\geq |x-\overline{x}_\Omega(x,\tilde{N})| $. 
 Thus, let us denote by $ \overline{x}_\Omega(x,\tilde{N})=(y_1,y_2,y_3) $ and $ x=(x_1,x_2,x_3)\in\Omega $. To compute the position of $  \overline{x}_\Omega(x,\tilde{N}) $ we solve the following system.
\begin{equation*}\label{op ou-u eps 26}
\begin{cases}
y_1=\frac{1}{R}\left(y_2^2+y_3^2\right),\\
y_1=x_1+t\eps^{\frac{1}{2}},\\ y_2=x_2+t\sqrt{1-\eps},\\y_3=x_3.
\end{cases}
\end{equation*}
We want to estimate from below the value of $ t>0 $. Hence we consider the quadratic equation 
\begin{equation*}\label{op ou-u eps 27}
(1-\eps)t^2+t\left(2\sqrt{1-\eps}x_2-R\eps^{\frac{1}{2}}\right)-\Delta=0,
\end{equation*}
where $ \Delta=Rx_1-x_2^2-x_3^2\geq 0 $ since $ x\in\Omega $ inside the paraboloid. Moreover, since $ x_2<\eptw $ for $ \eps $ small enough (i.e. $ \eps<(R/4)^{1/2\delta} $) we have that $ R\eps^{\frac{1}{2}}-2\sqrt{1-\eps}x_2\geq \frac{R}{2} \eps^{\frac{1}{2}} $. Thus,
\begin{equation*}\label{op ou-u eps 28}
(1-\eps)t^2\geq \frac{R}{2} \eps^{\frac{1}{2}}t+\Delta\geq\frac{R}{2} \eps^{\frac{1}{2}}t
\end{equation*} 
and therefore for $ \eps<1 $ we have
\begin{equation*}\label{op ou-u eps 29}
 |x-x_\Omega(x,n)|\geq |x-x_\Omega(x,\tilde{n})|\geq |x-\overline{x}_\Omega(x,\tilde{N})|\geq t\geq \frac{R}{4} \eps^{\frac{1}{2}}=C(\Omega)\eps^{\frac{1}{2}}.
\end{equation*}\\ Hence, using the usual estimate $ xe^{-\frac{x}{2}}\leq 1 $ and that $ |x-x_\Omega(x,n)|\geq d(x) $ we estimate
\begin{equation}\label{op ou-u eps 23}
\begin{split}
\intnu\int_{n\cdot N_p\geq \sin(\theta_2)}dn\;g_\nu(n)e^{-\frac{\left|x-x_{\Omega}(x,n)\right|}{\eps}}&\leq \Arrowvert g\Arrowvert_{\infty} \int_{n\cdot N_p\geq \sin(\theta_2)}dn\; \frac{2\eps}{\left|x-x_{\Omega}(x,n)\right|}e^{-\frac{\left|x-x_{\Omega}(x,n)\right|}{2\eps}}\\\leq&2\pi C(\Omega)^{-1} \Arrowvert g\Arrowvert_{\infty} \eps^{\delta}e^{-\frac{d(x)}{2\eps}},
\end{split}
\end{equation}
since $ \eps^{\frac{1}{2}}\leq \eps^{\delta} $ if $ \delta\leq \frac{1}{2} $.\\
\begin{figure}[H]
	\centering
	\begin{minipage}{.5\textwidth}
		\includegraphics[width=7 cm]{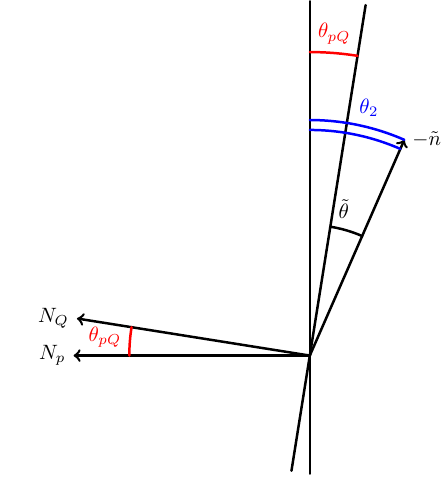}
		\captionof{figure}{Representation of the angles.}
	\end{minipage}%
	\begin{minipage}{.5\textwidth}
	
		\includegraphics[width= 9cm]{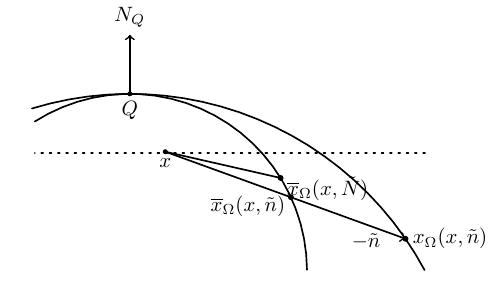}
		\captionof{figure}{Representation of the intersection points.}
	\end{minipage}%
\end{figure}
\noindent\textbf{Region $ \mathcal{U}_3 $}

Now, for the last estimate we notice that $ \left|\mathcal{U}_3\right|\leq 2\pi (\theta_1+\theta_2)\leq 4\pi \eps^\delta $ for $ \delta<\frac{1}{2}
$. Since it is always true that $ \left|e^{-\frac{\left|x-x_{\Pi_p}(x,n)\right|}{\eps}}-e^{-\frac{\left|x-x_{\Omega}(x,n)\right|}{\eps}}\right|\leq 2 e^{-\frac{d(x)}{\eps}} $, we estimate
\begin{equation}\label{op ou-u eps 24}
\intnu\int_{\mathcal{U}_3}dn\;g_\nu(n)\left|e^{-\frac{\left|x-x_{\Pi_p}(x,n)\right|}{\eps}}-e^{-\frac{\left|x-x_{\Omega}(x,n)\right|}{\eps}}\right|\leq C \Arrowvert g\Arrowvert_{\infty} \eps^{\delta}e^{-\frac{d(x)}{\eps}}.
\end{equation}
\noindent\textbf{Summarizing: estimate $ S_2 $}

We put together equations \eqref{op ou-u eps 19},\eqref{op ou-u eps 23}, \eqref{op ou-u eps 24} and we conclude
\begin{equation}\label{op ou-u eps 25}
S_2\leq C\begin{cases}
\eps^{\delta}e^{-\frac{d(x)}{2\eps}} &|x-p|<\eptw,\\e^{-\frac{d(x)}{2\eps}}&|x-p|\geq\eptw,
\end{cases}
\end{equation}
for a constant $ C>0 $ independent of $ x,p,\eps,\delta $.\\

Equations \eqref{op ou-u eps 25} and \eqref{Pip-omega} imply the lemma.
\end{proof}
\end{lemma}
We are ready to construct a super-solution that will allow us to estimate $ |u^\eps-\ou| $ near the boundary at a distance smaller than $ \eps^{\frac{1}{2}} $. We recall the rigid motion $ \Rot_p $ defined in \eqref{rotp}.
\begin{prop}\label{supsol ou-u eps}
	Let $ p\in\bnd $, $ 0<A<1 $ the constant of Lemma \ref{estimate op ou-u eps}. Let $ L>0 $ large enough and $ 0<\eps<1 $ sufficiently small. Let $ 0<\delta<\frac{1}{8} $. Then there exists a non negative continuous function $ W_{\eps, L}:\Omega\to \RR_+ $ such that 
	\begin{equation}\label{supsol case}
	\begin{cases}
	W_{\eps,L}\geq C>0 &\text{ for } \left|\Rot_p(x)\cdot e_i\right|\geq \epth ;\\\LL_\Omega^\eps\left(W_{\eps,L}\right)(x)\geq C \eps^\delta e^{-\frac{Ad(x)}{\eps}}&\text{ for }\left|\Rot_p(x)\cdot e_i\right|< \epth; \\0\leq W_{\eps,L}\leq C\left(\eps^\alpha+\frac{1}{\sqrt{L}}\right) &\text{ for } \left|\Rot_p(x)\cdot e_i\right|< \epfo, 
	\end{cases}
	\end{equation}
	for some constant $ C>0 $ and $ \alpha>0 $.
	\end{prop}
In order to construct this supersolution $ W_{\eps,L} $ we first need some definition for the geometrical setting. First of all we denote for simplicity $ x_i=\Rot_p(x)\cdot e_i $. Let us define for $ i=2,3 $ the radii $ \rho_i^\pm(x)=\sqrt{\left(x_1+\frac{L}{2}\eps\right)^2+\left(x_i\pm \epth\right)^2} $ and the angles $ \theta_i^\pm(x) $ given by $ \cos\left(\theta_i^\pm \right)=\frac{1}{\rho_i^\pm(x)}\left(x_1+\frac{L}{2}\eps\right) $. The function $ W_{\eps,L} $ is constructed with the following auxiliary functions 
\begin{equation}\label{Fpm}
F^{\pm}_i(x)=\frac{\pi}{2}\mp \arctan\left(\frac{x_i\pm\epth}{x_1+\frac{L}{2}\eps}\right),
\end{equation}
\begin{equation}\label{Gpm}
G^{\pm}_i(x)=a\left(\frac{\cos\left(\theta_i^\pm(x)\right)}{\rho_i^\pm(x)/\eps}\right)^{\frac{1}{2}},
\end{equation}
\begin{equation}\label{Hpm}
H^{\pm}_i(x)=-b\left(\frac{\cos\left(\theta_i^\pm(x)\right)}{\rho_i^\pm(x)/\eps}\right)^2
\end{equation}
for $ i=2,3 $ and $ a>b>0 $. Moreover, we define for $ i=2,3 $
\begin{equation}\label{Wpm}
W^{\pm}_i(x)=F^\pm_i(x)+G^\pm_i(x)+H^\pm_i(x).
\end{equation} We will prove that the desired supersolution of Proposition \ref{supsol ou-u eps} is given by
\begin{equation}\label{Wel}
W_{\eps,L}(x)=\sum_{i=2}^3 \left(W^+_i(x)+W^-_i(x)\right)+\frac{\tilde{C}}{\sqrt{L}}\phi_{\frac{1}{8},\eps}+
C\eps^{\delta}\phi_{A,\eps},
\end{equation} where $ \phi_{A,\eps}=\Phi^{\frac{\eps}{A}} $ the supersolution defined in Theorem \ref{interior supsol} and $ C,\tilde{C}>0 $ some suitable constants. We also define the following subsets of $ \Omega $ for $ i=2,3 $. 
\begin{equation}\label{C+2}
\mathcal{C}^+_{i,2\delta}:=\left\{x\in\Omega :\; x_i\leq -\epth \text{ or } |x_i|<\epth, x_1\geq \epth\right\};
\end{equation}
\begin{equation}\label{C-2}
\mathcal{C}^-_{i,2\delta}:=\left\{x\in\Omega :\; x_i\geq \epth \text{ or } |x_i|<\epth, x_1\geq \epth\right\};
\end{equation}
\begin{equation}\label{C3}
\mathcal{C}_{3\delta}:=\left\{x\in\Omega :\;x_1< \epth \text{ and } |x_i|<\epth \text{ for } i=2,3 \right\};
\end{equation}
\begin{equation}\label{C4}
\mathcal{C}_{i,4\delta}:=\left\{x\in\Omega :\;|x_i|<\epfo \text{ and } x_1< \epfo\right\}.
\end{equation}
In order to prove Proposition \ref{supsol ou-u eps} we need the following computational lemma.
\begin{lemma}\label{lemma supsol}
Assume $ p\in\Omega $, $ 0<\eps<1 $, $ L,\delta $ as indicated in Proposition \ref{supsol ou-u eps}. Let $ x_i=\Rot_p(x)\cdot e_i $ for $ i=1,2,3 $. Let $ W^\pm_i $ as in \eqref{Wpm}. Then there exist a constant $ \alpha>0 $ depending only on $ \delta $ and a constant $ C>0 $ depending on $ \Omega $ and $ g_\nu $ but independent of $ \eps $ and $ p $ and suitable $ b>0 $ and $ L>0 $ such that for $ i=2,3 $
\begin{numcases}{}
W^\pm_i(x)\geq 0 & in $ \Omega $\label{Wpmlemma1}\\
W^\pm_i(x)\geq \frac{\pi}{2}-\arctan(2) & in $\mathcal{C}^\pm_{i,2\delta}$\label{Wpmlemma2}\\
W^\pm_i(x)\leq C\eps^{\alpha} & in $\mathcal{C}_{i,4\delta}$\label{Wpmlemma4}\\
\LL_{\Omega}^\eps\left(W^{\pm}_i\right)(x)\geq-\frac{C}{\sqrt{L}}e^{-\frac{d(x)}{8\eps}} & in $\mathcal{C}_{3\delta}$\label{Wpmlemma5},
\end{numcases}
where the sets $ \mathcal{C}^\pm_{i,2\delta} $, $ \mathcal{C}_{i,4\delta} $ and $ \mathcal{C}_{3\delta} $ are defined in \eqref{C+2}, \eqref{C-2}, \eqref{C4} and \eqref{C3}.
\end{lemma}
\begin{proof}
Due to symmetry consideration it is enough to prove the lemma for $ W=W^-_2 $. For the sake of simplicity we write $ \rho(x)=\rho_2^-(x) $ and $ \theta(x)=\theta_2^-(x) $. Similarly we consider $ F=F^-_2 $, $ G=G^-_2 $ and $ H=H^-_2 $. We also denote by $ \mathcal{C}_{j\delta} $ the sets $ \mathcal{C}^-_{2,j\delta} $, $ \mathcal{C}_{2,j\delta} $ and $ \mathcal{C}_{j\delta} $ for $ j=2,3,4 $.

First of all we notice that $ W $ is smooth on $ x_1>-\frac{L}{2}\eps $. Moreover, since the arctangent is bounded from below by $ \frac{\pi}{2} $ we have that $ F\geq 0 $. Since $ x_1\geq0 $ for $ x\in\Omega $ we see that $ \rho\geq \frac{L}{2}\eps $ and hence for $ L $ big enough $ \frac{\rho}{\eps}>\frac{L}{2}>1 $. On the other hand $ 0\leq \cos(\theta)\leq 1 $ and hence for $ a>b $ we have that $$ G+H=a\left(\frac{\cos\left(\theta_2^-(x)\right)}{\rho_2^-(x)/\eps}\right)^{\frac{1}{2}}-b\left(\frac{\cos\left(\theta_2^-(x)\right)}{\rho_2^-(x)/\eps}\right)^2\geq (a-b)\left(\frac{\cos\left(\theta_2^-(x)\right)}{\rho_2^-(x)/\eps}\right)^{\frac{1}{2}}\geq 0 ,$$
which yields \eqref{Wpmlemma1}.

Assume now $ x_2\geq \epth $. This implies $ \frac{x_2-\epth}{x_1+\frac{L}{2}\eps}\geq0 $ and thus $ F(x)\geq\frac{\pi}{2} $. The non-negativity of $ G+H$ yields $ W(x)\geq \frac{\pi}{2} $.

Let us assume $ x_1\geq \epth $ and $ |x_2|<\epth $. A similar computation as above shows $ \frac{x_2-\epth}{x_1+\frac{L}{2}\eps}\geq -2\frac{\epth}{\epth\left(1+\frac{L}{2}\eps^{\frac{1}{2}-3\delta}\right)}\geq -2 $. Hence, $ W(x)\geq F(x)\geq \frac{\pi}{2}-\arctan(2)>0 $ for $ x\in\mathcal{C}_{2\delta} $ as in \eqref{Wpmlemma2}.

We move now to the proof of \eqref{Wpmlemma4}. Let therefore $ x\in\mathcal{C}_{4\delta} $. First of all $ W(x)\leq F+G $. Moreover, $ x_2-\epth<\epth(\eps^\delta-1)<0 $ and $ x_1+\frac{L}{2}\eps<\epfo \left(1+\frac{L}{2}\eps^{\frac{1}{2}-4\delta}\right)<\frac{3}{2}\epfo $ if $ \delta<\frac{1}{8} $, $ L>0 $ large enough and $ 0<\eps<1 $ sufficiently small such that $ L<\eps^{-\beta} $ for $ \beta=\frac{1-8\delta}{2} $. This computation implies
\begin{equation*}\label{supsol smallness1}
\frac{x_2-\epth}{x_1+\frac{L}{2}\eps}<-\frac{2}{3}\frac{1}{\eps^\delta}(1-\eps^\delta)<-\frac{1}{3}\frac{1}{\eps^\delta}
\end{equation*}
for $ \eps>0 $ small enough. With an application of Taylor expansion for $ y\to-\infty $ we conclude $ F(x)\leq 3\eps^\delta $, since $ \frac{\pi}{2}+\arctan(y)\approx |y|^{-1}-\frac{1}{3|y|^3} $. Moreover, since $ \frac{\rho}{\eps}>1 $ and also $ \rho\geq|x_2-\epth|\geq \frac{\epth}{2} $ if $ \eps $ small enough, we have that $ \cos(\theta)=\frac{x_1+\frac{L}{2}\eps}{\rho}<\frac 3\eps^\delta $. Hence, $ 0\leq\frac{\cos(\theta)}{\rho}<3\eps^\delta $ implies $ G(x)<\sqrt{3}\eps^{\delta/2} $. Taking then $ \alpha=\frac{\delta}{2} $ we conclude $ W\leq C\eps^\alpha $.

It remains now to show, that $ W $ satisfies the estimate \eqref{Wpmlemma5}. The main idea for this proof is to approximate the operator $ \LL_{\Omega}^\eps $ by a Laplacian expanding the function $ W $ by Taylor. Let us assume from now on that $ x\in\mathcal{C}_{3\delta} $, i.e. $ |x_3|,|x_2|,x_1<\epth $. We first notice that for these $ x $ the function $ F $ is harmonic and the functions $ G $ and $ H $ are super-harmonic. In order to prove this we change the coordinates in cylinder coordinates: $ (x_1,x_2,x_3)\mapsto(\rho(x_1,x_2),\theta(x_1,x_2),x_3) $. With this notation, since $ F,G,H $ are actually functions only of $ x_1 $ and $ x_2 $, we can write $ F $ as $ F(\rho,\theta)=\frac{\pi}{2}-\theta $. Thus, since in cylinder coordinates the Laplacian can be written as $ -\Delta= -\frac{1}{\rho}\partial_{\rho}\left(\rho\partial_\rho\right)-\frac{1}{\rho^2}\partial_\theta^2-\partial_{x_3}^2 $, we compute $ -\Delta F=0 $. On the other hand we can also compute for $ H $
\begin{equation}\label{Laplacian H}
 \frac{\eps^2}{\rho}\partial_{\rho}\left(\rho\partial_\rho\frac{\cos^2(\theta)}{\rho^2}\right)+\frac{\eps^2}{\rho^2}\partial_\theta^2\frac{\cos^2(\theta)}{\rho^2}=\frac{4\eps^2\cos^2(\theta)}{\rho^4}+\frac{2\eps^2\sin^2(\theta)}{\rho^4}-\frac{2\eps^2\cos^2(\theta)}{\rho^4}=\frac{2\eps^2}{\rho^4}.
\end{equation}
And similarly we have also for $ G $
\begin{equation}\label{Laplacian G}
 -\frac{\eps^{\frac{1}{2}}}{\rho}\partial_{\rho}\left(\rho\partial_\rho\frac{\cos^{\frac{1}{2}}(\theta)}{\rho^{\frac{1}{2}}}\right)-\frac{\eps^{\frac{1}{2}}}{\rho^2}\partial_\theta^2\frac{\cos^{\frac{1}{2}}(\theta)}{\rho^{\frac{1}{2}}}=-\frac{\eps^{\frac{1}{2}}\cos^{\frac{1}{2}}(\theta)}{4\rho^{\frac{5}{2}}}-\frac{\eps^{\frac{1}{2}}}{2\rho^{\frac{5}{2}}}\left(-\cos^{\frac{1}{2}}(\theta)-\frac{\sin^{2}(\theta)}{2\cos^{\frac{3}{2}}(\theta)}\right)=\frac{\eps^{\frac{1}{2}}}{4\cos^{\frac{3}{2}(\theta)}\rho^{\frac{5}{2}}}.
\end{equation}
We will use the (super-)harmonicity of these functions while applying the Taylor expansion on suitable domains.\\

Before moving to the exact estimate of the operator acting on $ W $ we estimate the derivatives of these functions. We start with analyzing $ F $. As we have seen before $ F=\frac{\pi}{2}-\theta\left(x_1,x_2\right) $ and hence we have $ \partial_1 F(x_1,x_2)=\frac{\sin(\theta)}{\rho} $ and $ \partial_2 F(x_1,x_2)=-\frac{\cos(\theta)}{\rho} $. Since the numerator contains only power laws of cosine and sine with exponent greater or equal 1 and the denominator also only power laws of $ \rho $ with exponent greater or equal 1, using the definition of derivatives in polar coordinates we see that there exists a constant $ C_{F,n}>0 $ for $ n\geq 1 $ such that 
\begin{equation*}\label{derivative F}
 \left|\nabla^n_x F(x)\right|\leq \frac{C_{F,n}}{\rho^n},
\end{equation*} where we also estimated the cosine and the sine by $ 1 $.

Let us move to the function $ H=-b\eps^2\frac{\cos^2(\theta)}{\rho^2} $. We use a similar argument. We compute using polar coordinates $ \partial_1H(x)=\left(\cos(\theta)\partial_\rho-\frac{\sin(\theta)}{\rho}\partial_\theta)\right)H(x)=-2b\eps^2\frac{\cos^3(\theta)}{\rho^3}+2b\eps^2\frac{\sin^2(\theta)\cos(\theta)}{\rho^3} $ and similarly $ \partial_2H(x)=\left(\sin(\theta)\partial_\rho+\frac{\cos(\theta)}{\rho}\partial_\theta)\right)H(x)=-2b\eps^2\frac{\sin(\theta)\cos^2(\theta)}{\rho^3}-2b\eps^2\frac{\cos^2(\theta)\sin(\theta)}{\rho^3} $. Again, the numerator only contains power of cosine and sin of degrees greater or equal $ 1 $, while the denominator only power of $ \rho $ of degree $ 3 $. Hence, applying again the definition of derivative in polar coordinates and estimating cosine and sine by $ 1 $ we conclude again the existence of a constant $ C_{H,n}>0 $ such that 
\begin{equation*}\label{derivative H}
\left|\nabla^n_x H(x)\right|\leq b\frac{C_{H,n}\eps^2}{\rho^{n+2}}.
\end{equation*} 
While the function $ F $ and $ H $ produce non singular derivatives for $ x\in\Omega $, $ G $ produces sigular terms in the derivatives. This is because the denominator of this function contains a square root of the cosine, hence when differentiating by $ \theta $, it appears in the denominator, indeed
\begin{equation*}\label{xderivative G}
\partial_1G(x)=\left(\cos(\theta)\partial_\rho-\frac{\sin(\theta)}{\rho}\partial_\theta)\right)G(x)=-a\eps^{\frac{1}{2}}\frac{\cos^{\frac{3}{2}}(\theta)}{2\rho^{\frac{3}{2}}}+a\eps^{\frac{1}{2}}\frac{\sin^2(\theta)}{2\rho^{\frac{3}{2}}\cos^{\frac{1}{2}}(\theta)};
\end{equation*}
\begin{equation*}\label{yderivative G}
\partial_2G(x)=\left(\sin(\theta)\partial_\rho+\frac{\cos(\theta)}{\rho}\partial_\theta)\right)G(x)=-a\eps^{\frac{1}{2}}\frac{\cos^{\frac{1}{2}}(\theta)\sin(\theta)}{\rho^{\frac{3}{2}}}.
\end{equation*}
Hence, the singular terms appear when differentiating with respect to $ x_1 $. Using that $ \cos(\theta),\sin(\theta)<1 $ and that $ \cos^{-\alpha}(\theta)>\cos^{-\beta}(\theta) $ for $ \alpha>\beta\geq 0 $ we conclude the existence of a constant $ C_{G,n}>0 $ such that
\begin{equation*}\label{derivative G}
\left|\nabla^n_x G(x)\right|\leq a \frac{C_{G,n}\eps^{\frac{1}{2}}}{\cos^{n-{\frac{1}{2}}}(\theta)\rho^{{\frac{1}{2}}+n}}.
\end{equation*} 
We remark that we used always that by construction $ \cos(\theta)\geq0 $.

As we anticipated we will estimate $ \LL_{\Omega}^\eps(W)(x) $ applying the Taylor expansion on $ F$, $ G $ and $ H $ on suitable subsets of $ \RR^3 $ where these functions are smooth. The functions $ F $ and $ H $ will be expanded until the third derivative and we will write
\begin{equation}\label{taylor+error F}
F(\eta)=F(x)+\nabla_xF(x)\cdot(\eta-x)+\frac{1}{2}(\eta-x)^\top\nabla_x^2F(x)(\eta-x)+\sum_{|\alpha|=3} \frac{D^\alpha F(x)}{\alpha!}(\eta-x)^\alpha+E^4_F(\eta,x)
\end{equation}
\begin{equation}\label{taylor+error H}
H(\eta)=H(x)+\nabla_xH(x)\cdot(\eta-x)+\frac{1}{2}(\eta-x)^\top\nabla_x^2H(x)(\eta-x)+\sum_{|\alpha|=3} \frac{D^\alpha H(x)}{\alpha!}(\eta-x)^\alpha+E^4_H(\eta,x).
\end{equation}
The function $ G $ will be expanded only until the second derivative, hence
\begin{equation}\label{taylor+error G}
G(\eta)=G(x)+\nabla_xG(x)\cdot(\eta-x)+\frac{1}{2}(\eta-x)^\top\nabla_x^2G(x)(\eta-x)+E^3_G(\eta,x).
\end{equation}
We recall also that for any smooth function $ \varphi(x_1,x_2) $ the following is true 
\begin{equation}\label{laplacian taylor}
\begin{split}
\int_{B^3_r(x)}d\eta\;\Kern{\eta-x}\frac{1}{2}(\eta-x)^\top& \nabla_x^2\varphi(x)(\eta-x)=\frac{1}{6}\Delta \varphi(x)\int_{B^3_r(x)}d\eta\;\Kern{\eta-x}|\eta-x|^2\\= \frac{1}{6}\Delta \varphi(x)\eps^2\int_0^{\frac{r}{\eps}} t^2e^{-t}dt&= \frac{\eps^2}{3}\Delta \varphi(x)-(r^2+2\eps r+2\eps^2)e^{-\frac{r}{\eps}}\frac{1}{6}\Delta \varphi(x).
\end{split}
\end{equation}
We can now move to the estimate of $ \LL_\Omega^\eps\left(W\right)(x) $ for $ |x_2|,|x_3|<\epth $ and $ 0<x_1<\epth $. We will consider three different cases: $ \rho(x)<L\eps $, $ \rho(x)>L\eps $ with $ d(x)<\eps $ and finally $ \rho(x)>L\eps $ with $ d(x)>\eps $.\\

\noindent\textbf{Case 1: $ \rho(x)<L\eps $}

Let us assume $ \rho(x)<L\eps $. Then, we remark first of all that if $ \eta\in B_{\frac{\rho(x)}{4}}(x) $ then $ \eta_1\geq x_1-\frac{\rho(x)}{4}>- \frac{L}{4}\eps>-\frac{L}{2}\eps$, which implies the smoothness of $ W $ on the whole ball $ B_{\frac{\rho(x)}{4}}(x) $. Moreover, it is also true that $ \cos(\theta)=\frac{x_1+\frac{L}{2}\eps}{\rho}\geq \frac{L\eps}{2 \rho}>\frac{1}{2} $. Hence, in this case the derivative of $ G $ is not singular and we can estimate $ \left|\nabla_x^n G(x)\right|\leq \frac{C_{G,n}2^{n-{\frac{1}{2}}}}{\rho^{{\frac{1}{2}}+n}} $. Moreover, if $ \eta\in B_{\frac{\rho(x)}{4}}(x) $, then from one hand we have $ \frac{3}{4}\rho(x)<\rho(\eta)<\frac{5}{4}\rho(x) $ and on the other hand $ \cos(\theta(\eta))\rho(\eta)=\eta_1+\frac{L}{2}\eps> \frac{L}{4}\eps $, thus $ \cos\left(\theta(\eta)\right)>\frac{1}{5} $
\begin{equation*}\label{taylor G rho small}
\sup_{\eta\in B_{\frac{\rho(x)}{4}}(x)}\left[\frac{1}{\cos^{\frac{5}{2}}\left(\theta(\eta)\right)\rho^{\frac{7}{2}}(x)}\right]\leq 5^{\frac{5}{2}}\left(\frac{4}{3}\right)^{\frac{7}{2}}\left(\frac{1}{\rho(x)}\right)^{\frac{7}{2}}\leq \frac{2C}{\eps L}\left(\frac{1}{\rho(x)}\right)^{\frac{5}{2}},
\end{equation*}
where at the end we used that $ \rho(x)>\frac{L\eps}{2} $ for all $ x\in\Omega $. For the computation of the operator $ \LL_\Omega^\eps $ for the function $ W $ we will use the Taylor expansion of this function on the ball $ B_{\frac{\rho(x)}{4}}(x) $. The error terms as defined in \eqref{taylor+error F}, \eqref{taylor+error H} and \eqref{taylor+error G} satisfy then
\begin{equation}\label{error F rho small}
\begin{split}
\left|E_F^4(\eta,x)\right|\leq \left(\frac{4}{3}\right)^4&\frac{C_{F,4}}{\rho(x)^4}|x-\eta|^4,\;\;\;\left|E_H^4(\eta,x)\right|\leq b \left(\frac{4}{3}\right)^6\frac{C_{H,4}}{\rho(x)^6}|x-\eta|^4\eps^2\;\text{ and }\\&\left|E_G^3(\eta,x)\right|\leq a \frac{C_{G,3}}{L\rho(x)^{\frac{5}{2}}}\frac{|x-\eta|^3}{\eps}\eps^{\frac{1}{2}}.
\end{split}
\end{equation}

We can now proceed with the estimate for the operator. Applying the Taylor expansion we obtain
\begin{equation}\label{estimate rho small}
\begin{split}
&\LL_\Omega^\eps(W)(x)=W(x)-\int_{B_{\frac{\rho(x)}{4}}(x)\cap\Omega} \Epskern{\eta-x}W
(\eta)\;d\eta-\int_{B^c_{\frac{\rho(x)}{4}}(x)\cap\Omega} \Epskern{\eta-x}W
(\eta)\;d\eta\\\geq & W(x)-\int_{B_{\frac{\rho(x)}{4}}(x)}\Epskern{\eta-x}W
(\eta)\;d\eta-\int_{B^c_{\frac{\rho(x)}{4}}(x)\cap\Omega} \Epskern{\eta-x}W
(\eta)\;d\eta\\
=&W(x)\left(1-\int_{B_{\frac{\rho(x)}{4}}(x)}\Epskern{\eta-x}\;d\eta\right)- \left(\nabla_xW(x)\right)\cdot\int_{B_{\frac{\rho(x)}{4}}(x)}\Epskern{\eta-x}(\eta-x)\;d\eta\\&-\frac{1}{2}\int_{B_{\frac{\rho(x)}{4}}(x)}\Epskern{\eta-x}(\eta-x)^\top\nabla^2_xW(x)(\eta-x)\;d\eta\\&-\sum_{|\alpha|=3}\frac{D^\alpha (F+H)(x)}{\alpha!}\int_{B_{\frac{\rho(x)}{4}}(x)}\Epskern{\eta-x}(\eta-x)^\alpha\;d\eta\\&-\int_{B_{\frac{\rho(x)}{4}}(x)}\Epskern{\eta-x}(E_F^4+E_H^4+E_G^3)(\eta,x)\;d\eta-\int_{B^c_{\frac{\rho(x)}{4}}(x)\cap\Omega} \Epskern{\eta-x}W
(\eta)\;d\eta
\end{split}\end{equation}
For the terms of the Taylor expansion containing the first and third derivatives we use now the symmetry of the integral in $ \RR^3 $, while for the second degree derivative terms we use the Laplacian identity as in equation \eqref{laplacian taylor}, together with the fact that $ F $ is harmonic while $ H $ and $ G $ are super-harmonic as in equations \eqref{Laplacian H} and \eqref{Laplacian G}. Denoting by $ C_F,C_G,C_H $ constants depending on $ F,G $ resp. $ H $ only and changing the coordinates $ y\mapsto (\eta-x) $ we estimate using the estimates for the error terms as in \eqref{error F rho small}
\begin{equation}\begin{split}\label{estimate rho small1}
&\LL_\Omega^\eps(W)(x)\geq - C_F\left(\frac{\eps}{\rho}\int_{B^c_{\frac{\rho(x)}{4}}(0)} \frac{e^{-\frac{|y|}{\eps}}}{4\pi \eps^2|y|}\;dy+\frac{\eps^3}{\rho^3}\int_{B^c_{\frac{\rho(x)}{4}}(0)} \frac{e^{-\frac{|y|}{\eps}}|y|}{4\pi \eps^4}\;dy\right) \\&
 - C_Hb\left(\frac{\eps^3}{\rho^3}\int_{B^c_{\frac{\rho(x)}{4}}(0)} \frac{e^{-\frac{|y|}{\eps}}}{4\pi \eps^2|y|}\;dy+\frac{\eps^4}{\rho^4}\int_{B^c_{\frac{\rho(x)}{4}}(0)} \frac{e^{-\frac{|y|}{\eps}}}{4\pi \eps^3}\;dy+\frac{\eps^5}{\rho^5}\int_{B^c_{\frac{\rho(x)}{4}}(0)} \frac{e^{-\frac{|y|}{\eps}}|y|}{4\pi \eps^4}\;dy\right)\\&
 - C_Ga\left(\frac{\eps}{L^{\frac{1}{2}}\rho}\int_{B^c_{\frac{\rho(x)}{4}}(0)} \frac{e^{-\frac{|y|}{\eps}}}{4\pi \eps^2|y|}\;dy+\frac{\eps^2}{L^\frac{1}{2}\rho^2}\int_{B^c_{\frac{\rho(x)}{4}}(0)} \frac{e^{-\frac{|y|}{\eps}}}{4\pi \eps^3}\;dy\right)\\&
 +\frac{2b}{3}\frac{\eps^4}{\rho^4}-C_F \frac{\eps^4}{\rho^4}-C_Hb \frac{\eps^6}{\rho^6}+ \frac{1}{12}\frac{a\eps^\frac{5}{2}}{\cos^{\frac{3}{2}}(\theta)\rho^{\frac{5}{2}}}-a \frac{C_G\eps^{\frac{5}{2}}}{L\rho(x)^{\frac{5}{2}}}-(\pi+a)e^{-\frac{\rho}{4\eps}}\\\geq&\frac{2}{3}\frac{\eps^4}{\rho^4}\left(b-\frac{3}{2}(C_F+\pi\tilde{C})-\frac{3C_Hb}{2L^2}-\frac{3C_F\tilde{C}}{2L}-\frac{3bC_H\tilde{C}}{2L^3}\right)+\frac{a\eps^\frac{5}{2}}{\rho^{\frac{5}{2}}}\left(\frac{1}{12}-\frac{C_G}{L}-\frac{\tilde{C}}{L^{\frac{3}{2}}}-\frac{C_G\tilde{C}}{L^3}\right).\end{split}
\end{equation}
Moreover, we used that $ \frac{\rho}{\eps}>\frac{L}{2}>1 $ and the well-known estimates $ \int_\frac{r}{\eps}^\infty e^{-x}x^n\leq C_ne^{-\frac{r}{2\eps}}\leq \tilde{C}\frac{\eps^4}{r^4} $ for $ n=0,1,2,3 $, as well as $ \frac{\eps^n}{\rho^n}<\frac{\eps}{\rho}<\frac{2}{L} $ and the fact that $ 0<\cos(\theta)<1 $.  Choosing 
\begin{equation}\label{estimate rho small 2}
b> 3 C_F(1+2\tilde{C}\pi) \;\;\;\;\;\text{ and }\;\;\;\;\;L>\max\{\tilde{C},\;\sqrt{6C_H},\;(24C_G+12)^2\}.
\end{equation}
We can conclude \begin{equation}\label{estimate rho small 3}
 \LL_{\Omega}^\eps(W)(x)\geq 0 .
\end{equation} 
\noindent\textbf{Case 2: $  \rho(x)>L\eps $ and $ d(x)<\eps $}

We consider now the case when $ \rho(x)>L\eps $ with $ d(x)<\eps $ for $ |x_2|,|x_3|<\epth $ and $ 0<x_1<\epth $. First of all, we see that if $ d(x)<\eps $ then also $ x_1<2\eps $. This is true since for all these $ x $ the distance can be estimated $\eps> d(x) = |x-z| $ for a unique $ \pi_{\bnd}(x)=z\in \bnd $. Hence, $ x_1<\eps+z_1 $ since also by the convexity and for $ \eps $ sufficiently small we have $ z_1\leq x_1 $. Thus, if $ z_1\geq\eps $ approximating the curvature by a sphere of radius $ R $ from the interior tangent to $ \{0\}\times \RR^2 $ we see that $ z_i\geq C(R)\eps^{\frac{1}{2}} $ for an $ i\in\{2,3\}$ and hence $ d(x)=|x-z|\geq \eps^{\frac{1}{2}}(C(R)-\eps^{3\delta})>\eps $ for $ \eps>0 $ sufficiently small. This implies a contradiction, and thus $ x_1<2\eps $. Let us consider now $ \eta\in B_{\frac{\rho(x)}{4}}(x)\cap\Omega $, then $ \frac{3}{4}\rho(x)<\rho(\eta)<\frac{5}{4}\rho(x) $. Hence, using the definition of cosine
\begin{equation*}\label{d small}
\begin{split}
\cos\left(\theta(x)\right)-\cos&\left(\theta(\eta)\right)=\frac{x_1+\frac{L}{2}\eps}{\rho(x)}-\frac{\eta_1+\frac{L}{2}\eps}{\rho(\eta)}\leq\frac{x_1+\frac{L}{2}\eps}{\rho(x)}-\frac{\frac{2L}{5}\eps}{\rho(x)}\\=&\frac{1}{\rho(x)}\left(x_1+\frac{1}{10}L\eps\right)=\frac{1}{\rho(x)}\left(\frac{x_1}{2}+\frac{L}{4}\eps\right)+\frac{1}{\rho(x)}\left(\frac{x_1}{2}-\frac{3L}{20}\eps\right)\leq  \frac{\cos\left(\theta(x)\right)}{2}
\end{split}
\end{equation*}
if $ L>\frac{20}{3} $. This implies that $ \cos\left(\theta(\eta)\right)\geq\frac{1}{2}\cos\left(\theta(x)\right) $ and therefore using $ \cos(\theta)\rho=x_1+\frac{L}{2}\eps\geq \frac{L}{2}\eps $ we obtain
\begin{equation}\label{taylor d small}
\sup_{\eta\in B_{\frac{\rho(x)}{4}}(x)\cap\Omega}\left[\cos^{-\frac{5}{2}}\left(\theta(\eta)\right)\rho^{-\frac{7}{2}}(x)\right]\leq 5^{\frac{5}{2}}\left(\frac{4}{3}\right)^{\frac{7}{2}}\cos\left(\theta(x)\right)^{-\frac{5}{2}}\rho(x)^{-\frac{7}{2}}\leq  \frac{2C}{\eps L}\cos\left(\theta(x)\right)^{-\frac{3}{2}}\rho(x)^{-\frac{5}{2}}.
\end{equation}

We can now proceed similarly as we did in equations \eqref{estimate rho small} and \eqref{estimate rho small1}. We apply the Taylor expansion on the set $ B_{\frac{\rho(x)}{4}}(x)\cap\Omega $ where $ W $ is smooth. 
	\begin{equation}\label{estimate d small}\begin{split}
	&\LL_\Omega^\eps(W)(x)=W(x)-\int_{B_{\frac{\rho(x)}{4}}(x)\cap\Omega} \Epskern{\eta-x}W
	(\eta)\;d\eta-\int_{B^c_{\frac{\rho(x)}{4}}(x)\cap\Omega} \Epskern{\eta-x}W
	(\eta)\;d\eta\\\geq &W(x)-\int_{B_{\frac{\rho(x)}{4}}(x)\cap\Omega} \Epskern{\eta-x}W
	(\eta)\;d\eta-\int_{B^c_{\frac{\rho(x)}{4}}(x)} \Epskern{\eta-x}(F+G)
	(\eta)\;d\eta\\
	=&W(x)\left(1-\int_{B_{\frac{\rho(x)}{4}}(x)\cap\Omega}\Epskern{\eta-x}\;d\eta\right)- \left(\nabla_xW(x)\right)\cdot\int_{B_{\frac{\rho(x)}{4}}(x)\cap\Omega}\Epskern{\eta-x}(\eta-x)\;d\eta\\&-\frac{1}{2}\int_{B_{\frac{\rho(x)}{4}}(x)\cap\Omega}\Epskern{\eta-x}(\eta-x)^\top\nabla^2_xW(x)(\eta-x)\;d\eta\\&-\sum_{|\alpha|=3}\frac{D^\alpha (F+H)(x)}{\alpha!}\int_{B_{\frac{\rho(x)}{4}}(x)\cap\Omega}\Epskern{\eta-x}(\eta-x)^\alpha\;d\eta\\&-\int_{B_{\frac{\rho(x)}{4}}(x)\cap\Omega}\Kern{\eta-x}(E_F^4+E_H^4+E_G^3)(\eta,x)\;d\eta-\int_{B^c_{\frac{\rho(x)}{4}}(x)} \Epskern{\eta-x}(F+G)
	(\eta)\;d\eta\\\end{split}
	\end{equation}Once again we use the symmetry for the first and third order term on the set $ B_{\frac{d(x)}{4}}(x)\subset B_{\frac{\rho(x)}{4}}(x)\cap\Omega $ estimating the integral on $ B_{\frac{\rho(x)}{4}}(x)\cap\left(\Omega\setminus B_{\frac{d(x)}{4}}(x)\right) $ by the integral on the larger set $ B^c_{\frac{d(x)}{4}}(x)$. Once more we need the identity for the Laplacian on the set $ B_{\frac{d(x)}{4}}(x)$ too. We estimate the error terms of $ F $ and $ H $ by \eqref{error F rho small} and the error term of $ G $ by the last equation \eqref{taylor d small}. Hence, we estimate\begin{equation}\label{estimate d small1}\begin{split}
	\geq& - C_F\left(\frac{\eps}{\rho}\int_{B^c_{\frac{d(x)}{4}}(0)} \frac{e^{-\frac{|y|}{\eps}}}{4\pi \eps^2|y|}\;dy+\frac{\eps^3}{\rho^3}\int_{B^c_{\frac{d(x)}{4}}(0)} \frac{e^{-\frac{|y|}{\eps}}|y|}{4\pi \eps^4}\;dy\right) \\&
- C_Hb\left(\frac{\eps^3}{\rho^3}\int_{B^c_{\frac{d(x)}{4}}(0)} \frac{e^{-\frac{|y|}{\eps}}}{4\pi \eps^2|y|}\;dy+\frac{\eps^4}{\rho^4}\int_{B^c_{\frac{d(x)}{4}}(0)} \frac{e^{-\frac{|y|}{\eps}}}{4\pi \eps^3}\;dy+\frac{\eps^5}{\rho^5}\int_{B^c_{\frac{d(x)}{4}}(0)} \frac{e^{-\frac{|y|}{\eps}}|y|}{4\pi \eps^4}\;dy\right)\\&
- C_Ga\left(\frac{\eps}{L^{\frac{1}{2}}\rho}\int_{B^c_{\frac{d(x)}{4}}(0)} \frac{e^{-\frac{|y|}{\eps}}}{4\pi \eps^2|y|}\;dy+\frac{\eps}{L^\frac{3}{2}\rho}\int_{B^c_{\frac{d(x)}{4}}(0)} \frac{e^{-\frac{|y|}{\eps}}}{4\pi \eps^3}\;dy\right)\\&
+\frac{2b}{3}\frac{\eps^4}{\rho^4}-C_F \frac{\eps^4}{\rho^4}-C_Hb \frac{\eps^6}{\rho^6}+ \frac{1}{12}\frac{a\eps^\frac{5}{2}}{\cos^{\frac{3}{2}}(\theta)\rho^{\frac{5}{2}}}-a \frac{C_G\eps^{\frac{5}{2}}}{L\cos^{\frac{3}{2}}(\theta)\rho(x)^{\frac{5}{2}}}-(\pi+a)e^{-\frac{\rho}{4\eps}}\\
\geq&-Ce^{-\frac{d}{8\eps}}\frac{\eps}{\rho}\left(a+C_F+\frac{aC_G}{L^\frac{1}{2}}+\frac{bC_H}{L^2}\right)+\frac{2}{3}\frac{\eps^4}{\rho^4}\left(b-\frac{3}{2}C_F-\frac{3C_Hb}{2L^2}\right)
+\frac{1}{12}\frac{a\eps^\frac{5}{2}}{\cos^{\frac{3}{2}}(\theta)\rho^{\frac{5}{2}}}\left(1-12\frac{C_G}{L}\right).
\end{split}
\end{equation}
In all the estimate above we used that  $ e^{-\frac{\rho}{4\eps}}\leq \frac{C\eps}{\rho}e^{-\frac{d}{8\eps}} $, since $ \rho>L\eps>d(x) $. Hence, choosing
\begin{equation}\label{estimate d small 2}
 b>3C_F\;\;\;\;\;\text{ and }\;\;\;\;\;L>\max\left\{\frac{20}{3},\;12 C_G,\;\sqrt{3C_H}\right\}
\end{equation}
we conclude using $ L\geq \sqrt{L}$
\begin{equation}\label{estimate d small 3}
\LL_\Omega^\eps(W)(x)\geq -\frac{C(a,\pi,b,C_F)}{\sqrt{L}}e^{-\frac{d(x)}{8\eps}}.
\end{equation}
\noindent\textbf{Case 3: $ \rho(x)>L\eps $ and $ d(x)>\eps$}

We can finish the proof of this Lemma by estimating the operator acting on $ W$ when $ \rho(x)>L\eps $ and $ d(x)>\eps$ for $ x\in\mathcal{C}_{3\delta} $. We notice that by definition $ d(x)\leq \rho(x) $. In this case, we estimate first the operator acting only on $ F+H $ proceeding as for the derivation of \eqref{estimate d small} and \eqref{estimate d small1} with the only difference that this time
\begin{equation*}\label{F+H d big 1}
\begin{split}
(F+H)(x)\left(1-\int_{B_{\frac{\rho(x)}{4}}(x)\cap\Omega}\Epskern{\eta-x}\;d\eta\right)\geq&-\frac{b\eps^2\cos(\theta)^2}{\rho^2}\int_{B^c_{\frac{d}{4}}(x)}\Epskern{\eta-x}\;d\eta\\\geq &-\frac{b\eps^2}{\rho^2}e^{-\frac{d(x)}{4\eps}}\geq-\frac{b}{L^2}e^{-\frac{d(x)}{4\eps}}.
\end{split}
\end{equation*}
Hence we get for $ F+H $
\begin{equation*}\label{F+H big 2}
\begin{split}
\LL_\Omega^\eps(F+H)(x)\geq&-Ce^{-\frac{d(x)}{8\eps}}\frac{1}{L}\left(C_F+\pi+\frac{bC_H}{L^2}+\frac{b}{L}\right)+\frac{2}{3}\frac{\eps^4}{\rho^4}\left(b-\frac{3}{2}C_F-\frac{3C_Hb}{2L^2}\right)\\\geq& -Ce^{-\frac{d(x)}{8\eps}}\frac{1}{L}(\pi+2b),
\end{split}
\end{equation*}
for $ b $ and $ L $ as in \eqref{estimate d small 2}. We consider now the operator acting only on $ G $ and we compute
\begin{equation*}\label{G big 1}
\begin{split}
\LL_\Omega^\eps(G)(x)&=G(x)-\int_{B_{\frac{d(x)}{4}}(x)}\Epskern{\eta-x}G(\eta)\;d\eta-\int_{B_{\frac{\rho(x)}{4}}(x)\cap\Omega\setminus B_{\frac{d(x)}{4}}(x)}\Epskern{\eta-x}G(\eta)\;d\eta\\&-\int_{B^c_{\frac{\rho(x)}{4}}(x)\cap\Omega}\Epskern{\eta-x}G(\eta)\;d\eta\\\geq& G(x)-\int_{B_{\frac{d(x)}{4}}(x)}\Epskern{\eta-x}G(\eta)\;d\eta-\frac{2a}{\sqrt{3L}}e^{-\frac{d(x)}{4\eps}}-\frac{Ca}{L}e^{-\frac{d(x)}{8\eps}}.
\end{split}
\end{equation*}
We used $ \rho(\eta)\geq \frac{3}{4}\rho(x)\geq\frac{3}{4}L\eps $ for $ \eta\in B_{\frac{\rho(x)}{4}}(x)\cap\Omega\setminus B_{\frac{d(x)}{4}}(x) $, the integral on $ B_{\frac{\rho(x)}{4}}(x)\cap\Omega\setminus B_{\frac{d(x)}{4}}(x) $ can be estimated from above by the one on $ B^c_{\frac{d(x)}{4}}(x) $ and the last integral was estimated as in \eqref{estimate d small} and \eqref{estimate d small1} using $ e^{-\frac{\rho(x)}{4\eps}}\leq \frac{C}{L}e^{-\frac{d(x)}{8\eps}} $. In order to estimate the integral on $ B_{\frac{d(x)}{4}}(x) $ we will expand $ G $ by Taylor and therefore we have to control the singularity $ \left(\cos(\theta)\right)^{-\frac{5}{2}} $ of the error term $ E^3_G $ defined in \eqref{taylor+error G}. Let hence $ \eta\in B_{\frac{d(x)}{4}}(x) $ for $ d(x)>\eps $. Since $ d(x)\leq x_1 $, we know that $ x_1>\eps $ and also that $ \eta_1>x_1-\frac{d(x)}{4}>\frac{3}{4}x_1 $. That $ d(x)\leq x_1 $ can be proved in the following way. Let $ z=\left\{x-te_1:t\geq 0\right\}\cap \bnd $, hence $ z_1\geq 0 $. Then $ d(x)\leq |x-z|=x_1-z_1\leq x_1 $. We can thus estimate
\begin{equation*}\label{G big 2}
\begin{split}
\cos\left(\theta(\eta)\right)=&\frac{\eta_1+\frac{L}{2}\eps}{\rho(\eta)}>\frac{x_1-\frac{d(x)}{4}+\frac{L}{2}\eps}{\rho(\eta)}>\frac{\frac{3}{4}x_1+\frac{L}{2}\eps}{\rho(\eta)}\\=&\frac{3}{4}\frac{x_1+\frac{L}{2}\eps}{\rho(\eta)}+\frac{1}{8}\frac{L\eps}{\rho(\eta)}>\frac{3}{4}\frac{x_1+\frac{L}{2}\eps}{\rho(\eta)}>\frac{3}{5}\frac{x_1+\frac{L}{2}\eps}{\rho(x)}=\frac{3}{5}\cos\left(\theta(x)\right),
\end{split}
\end{equation*}
where we used that $ \rho(\eta)<\rho(x)+\frac{d(x)}{4}<\frac{5}{4}\rho(x) $ since $ \rho(x)\geq x_1+\frac{L}{2\eps}\geq d(x)+\frac{L}{2\eps}> d(x) $. Similarly $ \rho(\eta)> \frac{3}{4}\rho(x) $. Hence,
\begin{equation*}\label{taylor d big}
\sup_{\eta\in B_{\frac{d(x)}{4}}(x)}\left[\cos^{-\frac{5}{2}}\left(\theta(\eta)\right)\rho^{-\frac{7}{2}}(x)\right]\leq \left(\frac{5}{3}\right)^{\frac{5}{2}}\left(\frac{4}{3}\right)^{\frac{7}{2}}\cos\left(\theta(x)\right)^{-\frac{5}{2}}\rho(x)^{-\frac{7}{2}}\leq  \frac{2C}{\eps L}\cos\left(\theta(x)\right)^{-\frac{3}{2}}\rho(x)^{-\frac{5}{2}},
\end{equation*}
where we used in addition $ \cos(\theta)\rho=x_1+\frac{L}{2}\eps\geq \frac{L}{2}\eps $.
Now we are ready to conclude the estimate for $ G $. We proceed as we did in \eqref{estimate rho small} and in \eqref{estimate d small} using the Taylor expansion. We have then for $ L $ as in \eqref{estimate d small 2}
\begin{equation*}\label{G big 3}
\begin{split}
\LL_\Omega^\eps(G)(x)&\geq -\frac{2a}{\sqrt{3L}}e^{-\frac{d(x)}{4\eps}}-\frac{Ca}{L}e^{-\frac{d(x)}{8\eps}}-\frac{aC_G}{L^{\frac{3}{2}}}e^{-\frac{d(x)}{8\eps}} +\frac{1}{12}\frac{a\eps^\frac{5}{2}}{\cos^{\frac{3}{2}}(\theta)\rho^{\frac{5}{2}}}\left(1-12\frac{C_G}{L}\right)\\\geq&-\frac{C(a)}{\sqrt{L}}e^{-\frac{d(x)}{8\eps}},
\end{split}
\end{equation*}
Thus, we obtain once more
\begin{equation}\label{estimate d big}
\LL_\Omega^\eps(W)(x)\geq -\frac{C(a,\pi,C_F,b)}{\sqrt{L}}e^{-\frac{d(x)}{8\eps}}.
\end{equation}
Equations \eqref{estimate rho small 3}, \eqref{estimate d small 3} and \eqref{estimate d big} imply the last claim \eqref{Wpmlemma5}. Indeed, there exists a constant $ C>0 $ independent of $ \eps,L,\delta $ such that for $ L>L_0 $ as in \eqref{estimate rho small 2} and \eqref{estimate d small 2} and for $ 0<\eps<1 $ sufficiently small such that $ L<\eps^{-\frac{1-8\delta}{2}} $, $ x\in\mathcal{C}_{3\delta} $ it holds 
\begin{equation*}\label{estimate d big 2}
\LL_\Omega^\eps(W)(x)\geq -\frac{C}{\sqrt{L}}e^{-\frac{d(x)}{8\eps}}.
\end{equation*}
This conclude the proof of the Lemma.
\end{proof}
We can now prove Proposition \ref{supsol ou-u eps}.
\begin{proof}[Proof of Proposition \ref{supsol ou-u eps}]
Let $ W_{\eps,L} $ be as in \eqref{Wel}. Lemma \ref{lemma supsol} implies for $ \overline{W}:=\sum_{i=2}^3\left(W^+_i(x)+W^-_i(x)\right) $ that for any $ i=1,2,3 $
\begin{equation}\label{proofprop1}
	\begin{cases}
	\overline{W}(x)\geq 0 & \text{ if }x\in\Omega;\\\overline{W}(x)\geq\frac{\pi}{2}-\arctan(2)&\text{ if }|x_i|\geq \epth;\\
	\overline{W}(x)\leq C\eps^\alpha&\text{ if }|x_i|< \epfo;\\\LL_{\Omega}^\eps(\overline{W})(x)\geq -\frac{C}{\sqrt{L}}e^{-\frac{d(x)}{8\eps}}&\text{ if }|x_i|< \epth.
	\end{cases}
\end{equation}
Moreover, Theorem \ref{interior supsol} imply that there exists a constant $ \tilde{C}(\Omega,g_\nu) $ such that $ \frac{\tilde{C}}{\sqrt{L}}\phi_{\frac{1}{8},\eps} $ satisfies
\begin{equation*}
\LL_{\Omega}^\eps\left(\frac{\tilde{C}}{\sqrt{L}}\phi_{\frac{1}{8},\eps}\right)(x)\geq \frac{C}{\sqrt{L}}e^{-{\frac{d(x)}{8\eps}}}\;\;\;\;\;\forall x\in\Omega,
\end{equation*}
where $ C $ is the constant in \eqref{proofprop1}. Theorem \ref{interior supsol} implies also $ 0\leq \phi_{A,\eps}(x)\leq C(\Omega)<\infty $ for any $ 0<A<1 $ and any $ 0<A\eps\leq \eps_0 $. Hence, $ W_{\eps,L} $ is non negative and we can conclude equation \eqref{supsol case} for $ \alpha=\frac{\delta}{2} $, for some constant $ C(\Omega, g_\nu)>0 $ independent of $ \delta,p\in\bnd,\eps,L $ and for $ L>0 $ and $ 0<\eps<1 $ sufficiently small such that $ L>\eps^{-\frac{2}{1-8\delta}} $.
\end{proof}
The properties of $ W_{\eps,L} $ imply now the boundedness of $ |u^\eps-\ou| $ near the boundary $ \bnd $.
\begin{corollary}\label{good estimate}
Let $ 0<\delta<\frac{1}{16} $. There exists a constant $ C>0 $, a large $ L>0 $ and an $\alpha>0 $ independent of $ x,p,\eps $ such that 
\begin{equation*}\label{good estimate 1}
\left|\ou\left(\frac{R_p(\cdot)\cdot e_1}{\eps},p\right)-u^\eps\right|(x)\leq C \left(\eps^\alpha+\frac{1}{\sqrt{L}}\right)
\end{equation*}
for all $ |x-p|<\epfo $.
\begin{proof}
Let $ L>0 $ large as in Proposition \ref{supsol ou-u eps} and let $ 0<\eps<1 $ sufficiently small such that $ L<\eps^{-\frac{1}{\beta}} $ for $ \beta=\frac{1-8\delta}{2} $. Let $ \delta<\frac{1}{16} $ and $ \alpha=\frac{\delta}{2} $. If $ |x-p|<\epfo $ with $ x\in\Omega $ then $ 0<x\cdot (-N_p)<\epfo $ and $ |\Rot_p(x)\cdot e_i|<\epfo $ for $ i=2,3 $. Let us consider $ W_{\eps,L} $ as defined in \eqref{Wel}. As we know, $ \ou\left(\frac{R_p(\cdot)\cdot e_1}{\eps},p\right) $ and $ u^\eps $ are both uniformly bounded, independently of $ p $. Let us call $ K>0 $ this bound. Moreover, $ W_{\eps,L}\geq \frac{\pi}{2}-\arctan(2):= \tilde{C}>0 $. Hence, for all $ |\Rot_p(x)\cdot e_i|\geq \epth $ we have that $ \frac{K}{\tilde{C}}W_{\eps,L}(x)-\left|\ou\left(\frac{R_p(\cdot)\cdot e_1}{\eps},p\right)-u^\eps\right|(x)\geq 0 $. The function $ \frac{K}{\tilde{C}}W_{\eps,L} $ is also continuous and satisfies $ \frac{K}{\tilde{C}} \LL_\Omega^\eps\left(W_{\eps,L}\right)\geq C\eps^\delta \frac{K}{\tilde{C}}e^{-\frac{Ad(x)}{\eps}} $ for all $ |\Rot_p(x)\cdot e_i|< \epth $. Using the maximum principle of Theorem \ref{max principle omega} and the estimate for the operator acting on $ \ou\left(\frac{R_p(\cdot)\cdot e_1}{\eps},p\right)-u^\eps $ in Lemma $ \ref{estimate op ou-u eps} $ we obtain 
\begin{equation*}\label{good estimate 2}
\left|\ou\left(\frac{R_p(\cdot)\cdot e_1}{\eps},p\right)-u^\eps\right|(x)\leq C \left(\eps^\alpha+\frac{1}{\sqrt{L}}\right)
\end{equation*}
for all $ |x-p|<\epfo $. 
\end{proof} 
\end{corollary}
This is a key estimate for the proof of the convergence of the exact solutions $ u^\eps $ up to the boundary $ \bnd $.
\subsection{Convergence of $ u^\eps $ to the solution of the new boundary value problem}
This last section is devoted to the proof of the pointwise convergence of $ u^\eps $ to the solution of the Laplace equation in $ \Omega $ with boundary value $ \ou_\infty $ (cf. Proposition \ref{prop u}).
We proceed defining new regions. We define $ \hat{\Omega}_\eps:=\left\{x\in\Omega: d(x)>\epfo\right\} $, $ \Sigma_\eps:=\left\{x\in\Omega: \epsi<d(x)\leq\epfo\right\} $ and their union $ \Omega_\eps=\hat{\Omega}_\eps\cup\Sigma_\eps $. We also define for $ 0<\sigma\ll 1 $ independent of $ \eps $ the set $ \Omega^\sigma:=\Omega\cup\left\{x\in\Omega^c: d(x)<\sigma\right\} $. We recall the continuous projection $ \pi_{\bnd} $ as given in \eqref{proj}.
\begin{figure}[H]\centering
\includegraphics[height=6cm]{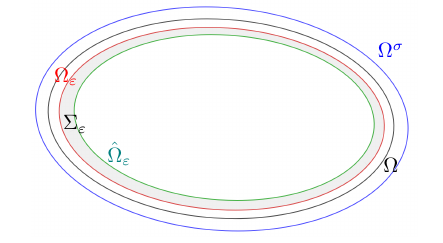}\caption{Decomposition of $ \Omega $}
\end{figure}
\begin{lemma}\label{estimate in sigma  eps}
Let $ 0<\eps<1 $ and $ 0<\delta<1\frac{1}{16} $. Let $ C,L,\alpha $ be as in the Corollary \ref{good estimate}. Then it holds
\begin{equation}\label{estimate sigma eps}
\sup_{x\in\Sigma_\eps}\left|\ou_\infty\left(\pi_{\bnd}(x)\right)-u^\eps(x)\right|\leq C\left(\eps^\alpha+\frac{1}{\sqrt{L}}\right)+\tilde{\omega}_1\left(\eps^{\frac{1}{2}-6\delta}\right),
\end{equation}
where $ \tilde{\omega}_1(r)=C e^{-\frac{r}{2}} $ for a suitable constant $ C>0 $.
\begin{proof}

By Corollary \ref{rate} there exists a constant $ C>0 $ independent of $ p\in\bnd $ such that\linebreak $ \left|\ou(y,p)-\ou_\infty(p)\right|\leq Ce^{-\frac{y}{2}}=\tilde{\omega}_1(y) $. Hence, let $ x\in\Sigma_\eps $. Then
\begin{equation*}\label{estimate sigma eps 1}
\begin{split}
&\left|\ou_\infty\left(\pi_{\bnd}(x)\right)-u^\eps(x)\right|\\\leq& \left|u^\eps(x)-\ou\left(\frac{R_{\pi_{\bnd}(x)}(x)\cdot e_1}{\eps},\pi_{\bnd}(x)\right)\right|-\left|\ou_\infty\left(\pi_{\bnd}(x)\right)-\ou\left(\frac{R_{\pi_{\bnd}(x)}(x)\cdot e_1}{\eps},\pi_{\bnd}(x)\right)\right|\\<& C \left(\eps^\alpha+\frac{1}{\sqrt{L}}\right)+\tilde{\omega}_1\left(\eps^{\frac{1}{2}-6\delta}\right),
\end{split}
\end{equation*}
where we used the result of Corollary \ref{good estimate} and the fact that since $ \epsi<d(x)<\epfo $ then we have that $ \eps^{-\frac{1}{2}+6\delta} \leq \left|\frac{R_{\pi_{\bnd}(x)}(x)\cdot e_1}{\eps}\right|\leq \eps^{-\frac{1}{2}+4\delta } $. Moreover, since $ \delta<\frac{1}{16} $, we have that $ \frac{1}{2}-6\delta>\frac{1}{8}>¨0 $.
\end{proof}
\end{lemma}
We recall first the definition of $ v $ as solution to the Laplacian (cf. \eqref{defv})
\begin{equation*}\label{definition v}
\begin{cases}
-\Delta v(x)=0 & x\in\Omega,\\v(p)=\ou_\infty(p) & p\in\bnd.
\end{cases}
\end{equation*}
Clearly by the uniformly continuity of $ \ou_\infty $ we have that $ v\in C^\infty\left(\Omega\right)\cap C\left(\overline{\Omega}\right) $. We call $ \omega_2 $ its modulus of continuity. Further, we need to consider another function, harmonic on the larger domain $ \Omega^\sigma $
\begin{equation*}\label{definition vsigma}
\begin{cases}
-\Delta v_\sigma(x)=0 & x\in\Omega^\sigma,\\v(x)=\ou_\infty\left(\pi_{\bnd}(x)\right) & x\in\bnd^\sigma.
\end{cases}
\end{equation*}
We recall that $ \pi_{\bnd}:\bnd^\sigma\to\bnd $ is a continuous bijection if $ \sigma>0 $ small enough and therefore $ v_\sigma\in C^\infty\left(\Omega^\sigma\right)\cap C\left(\overline{\Omega^\sigma}\right) $. Denoting $ \omega $ its modulus of continuity a simple application of the maximum principle for harmonic functions implies $ \sup_{x\in\Omega} \left|v_\sigma(x)-v(x)\right|\leq \omega(\sigma)$. Indeed, $ v_\sigma-v $ is harmonic on $ \Omega $ and thus the maximum must be attained on $ \bnd $. Hence, using that $ \pi_{\bnd} $ is a bijection we obtain
\begin{equation}\label{v-vsigma}
v_\sigma(x)-v(x)\leq \max\limits_{x\in\bnd}\left(v_\sigma(x)-v(x)\right)=\max\limits_{x\in\bnd^\sigma}\left(v_\sigma(\pi_{\bnd}(x))-v_\sigma(x)\right)\leq \omega(\sigma),
\end{equation}
since $ |x-\pi_{\bnd}(x)|=\sigma $ for $ x\in\bnd^\sigma $. The same can be estimated for $ v-v_\sigma $.
\begin{lemma}\label{estimate op vsigma-ueps}
Let $ x\in\hat{\Omega}_\eps $. Then
\begin{equation}\label{estimate op vsigma-ueps 1}
\left|\LL_\Omega^\eps\left(v_\sigma-u^\eps\right)(x)\right|\leq C(\Omega,g_\nu) e^{-\frac{Ad(x)}{\eps}}\left(\eps^\beta+\frac{\eps}{\sigma}\right)+\frac{C}{\sigma^3}\eps^3,
\end{equation}
for some constant $ C(\Omega,g_\nu)>0 $ and $ \eps>0 $ sufficiently small.
\begin{proof}
We already know that in general we always have the estimate
 $ |\LL_\Omega^\eps\left(u^\eps\right)(x)|\leq Ce^{-\frac{d}{\eps}} $ (cf. Theorem \ref{interior supsol}). Since for $ x\in\hat{\Omega}_\eps $ the distance to the boundary satisfies $ d(x)>\epfo $, we estimate for these points
 \begin{equation}\label{LL ueps}
 |\LL_\Omega^\eps\left(u^\eps\right)(x)|\leq Ce^{-\frac{d}{\eps}} \leq Ce^{-\frac{d}{2\eps}} 2\frac{\eps}{\epfo}\leq  Ce^{-\frac{d}{2\eps}}\eps^{\frac{1}{2}-4\delta}=Ce^{-\frac{d}{2\eps}}\eps^\beta,
 \end{equation}
 where $ \beta =\frac{1}{2}-4\delta>0$ for $ \delta<\frac{1}{16} $. \\
 
 Let us consider now the operator $ \LL_{\Omega}^\eps $ acting on $ v_\sigma $. We apply as usual the Taylor expansion to $ v_\sigma(\eta)=v_\sigma(x)+\nabla_x v_\sigma(x)\cdot (\eta-x)+\frac{1}{2}(\eta-x)^\top\nabla_x^2 v_\sigma(x)(\eta-x)+E^3(\eta,x) $. Since $ x\in B_\sigma(x)\subset\Omega^\sigma $ for all $ x\in\Omega $ by the harmonicity of  $ v_\sigma $ we obtain
 \begin{equation*}\label{derivative v sigma}
 \left|\partial^\alpha v_\sigma(x)\right|\leq C(|\alpha|)\frac{\Arrowvert v_\sigma\Arrowvert_{\infty} \sigma^3 }{\sigma^{3+|\alpha|}}=C(|\alpha|)\frac{\Arrowvert \ou_\infty \Arrowvert_{\infty} }{\sigma^{|\alpha|}}\leq \frac{C }{\sigma^{|\alpha|}},
 \end{equation*}
 where we used that $ v_\sigma $ attains its maximum on the boundary and that $ \ou_\infty $ is uniformly bounded. Hence, we calculate for $x\in\hat{\Omega}_\eps $ and $ \eps>0 $ sufficiently small
 \begin{equation}\label{LL vsigma}
 \begin{split}
 \left|\LL_\Omega^\eps\left(v_\sigma\right)(x)\right|=& \left|v_\sigma(x)-\int_\Omega \Epskern{\eta-x}\left[v_\sigma(x)+\nabla_x v_\sigma(x)\cdot (\eta-x)+\frac{1}{2}(\eta-x)^\top\nabla_x^2 v_\sigma(x)(\eta-x)\right]\;d\eta\right|\\
 &+\left|\int_\Omega \Epskern{\eta-x} E^3(\eta,x)\;d\eta\right|\\\leq & v_\sigma(x)\int_{B^c_{d(x)(x)}}\Epskern{\eta-x}\;d\eta+ \frac{C\eps}{\sigma}\int_{B^c_{d(x)(0)}}\frac{e^{-\frac{|y|}{\eps}}}{4\pi \eps^2|y|}\;dy+\frac{C\eps^2}{\sigma^2}\int_{B^c_{d(x)(0)}}\frac{e^{-\frac{|y|}{\eps}}}{4\pi \eps^3}\;dy\\&+ \frac{C\eps^3}{\sigma^3}\int_{\RR^3}\frac{e^{-\frac{|y|}{\eps}}}{4\pi \eps^3}\frac{|y|}{\eps}\;dy\\\leq& Ce^{-\frac{d(x)}{\eps}}\left(1+\frac{\eps}{\sigma}\right)+ \frac{C\eps^3}{\sigma^3},
 \end{split}
 \end{equation}
 where we used the symmetry of the kernel for the first order term, the fact that $ v_\sigma $ is harmonic for the second order term, the boundedness of $ \int_0^\infty e^{-r}r^3\;dr $ and the fact that $ \Omega^c\subset B^c_{d(x)}(x) $. Equations \eqref{LL ueps} and \eqref{LL vsigma} yield \eqref{estimate op vsigma-ueps 1}.
\end{proof}
\end{lemma}
\begin{lemma}\label{estimate sigmaeps}
Let $ x\in\Sigma_\eps $ and $ \eps>0 $ small enough. Then the following uniform bound holds
\begin{equation*}\label{estimate sigmaeps 1}
\left|v_\sigma(x)-u^\eps(x)\right|\leq \omega(\sigma)+\omega_2\left(\epfo\right)+ C\left(\eps^\alpha+\frac{1}{\sqrt{L}}\right)+\tilde{\omega}_1\left(\eps^{\frac{1}{2}-6\delta}\right).
\end{equation*}
\begin{proof}
Let $ x\in\Sigma_\eps $. Then, $ d(x)=\left|x-\pi_{\bnd}(x)\right|<\epfo $. Hence, $ \left|v(x)-\ou_\infty\left(\pi_{\bnd}(x)\right)\right|<\omega_2(\epfo) $. Thus using equation \eqref{v-vsigma} and Lemma \ref{estimate in sigma  eps} we conclude
\begin{equation}\label{estimate sigmaeps1}
\begin{split}
\left|v_\sigma(x)-u^\eps(x)\right|\leq& \left|v_\sigma(x)-v(x)\right|+\left|v(x)-\ou_\infty\left(\pi_{\bnd}(x)\right)\right|+\left|\ou_\infty\left(\pi_{\bnd}(x)\right)-u^\eps(x)\right|\\
\leq&  \omega(\sigma)+\omega_2\left(\epfo\right)+ C\left(\eps^\alpha+\frac{1}{\sqrt{L}}\right)+\tilde{\omega}_1\left(\eps^{\frac{1}{2}-6\delta}\right).
\end{split}
\end{equation}
\end{proof}	
\end{lemma}
With an application of the maximum principle on $\Omega_\eps:= \hat{\Omega}_\eps\cup\Sigma_\eps $ we can prove the convergence of $ u^\eps $ to the harmonic function $ v $ $ \eps\to 0 $.
\begin{theorem}\label{convergence up to the boundary}
$ u^\eps $ converges to $ v $ uniformly in every compact set.

\begin{proof}
As we mentioned above we will use the maximum principle for the operator $ \LL_{\Omega_\eps}^\eps $. We start estimating for $ x\in\hat{\Omega}_\eps $ this operator acting on $ v_\sigma-u^\eps $. Thus, for $ \eps>0 $ sufficiently small
\begin{equation}\label{convergence 1}
\begin{split}
\left|\LL_{\Omega_\eps}^\eps\left(v_\sigma-u^\eps\right)(x)\right|\leq&\left|\LL_\Omega^\eps\left(v_\sigma-u^\eps\right)(x)\right|+\int_{\Omega\setminus\Omega_\eps} \Epskern{\eta-x}\left(v_\sigma(\eta)-u^\eps(\eta)\right)\;d\eta\\ 
\leq Ce^{-\frac{d(x)}{\eps}}&\left(1+\frac{\eps}{\sigma}\right)+ \frac{C\eps^3}{\sigma^3}+ \frac{C}{\eps}\exp\left(-\frac{\eps^{\frac{8\delta-1}{2}}}{2}\right)\int_{\Omega\setminus\Omega_\eps} \frac{1}{|\eta-x|^2}\;d\eta\\
\leq Ce^{-\frac{d(x)}{2\eps}}&\left(\eps^{\frac{1-8\delta}{2}}+\frac{\eps}{\sigma}\right)+ \frac{C\eps^3}{\sigma^3}+C\left(\Omega\right)\eps^3,
\end{split}
\end{equation}
where we used Lemma \ref{estimate op vsigma-ueps}, the fact that if $\eta\in \Omega\setminus\Omega_\eps $ and $ 0<\eps<2^{-\frac{1}{2\delta}} $ we have $ |x-\eta|>\epfo(1-\eps^{2\delta})>\frac{\epfo}{2} $ and that $ d(x)\geq \epfo $. Moreover, we used that $ \Omega\setminus\Omega_\eps\subset \Omega $ and that for any $ n\in\mathbb{N} $ there exists a constant $ C_n $ such that $ |x|^ne^{-|x|}\leq C_n $. We chose here $ \delta<\frac{1}{72} $ and $ n=9 $.\\

We now set for $ C>0 $ the maximum between the constants appearing in estimates \eqref{estimate sigmaeps1} and \eqref{convergence 1} 
\begin{equation*}\label{Keps}
K_\eps=C\left(\eps^{\frac{1-8\delta}{2}}+\frac{\eps}{\sigma}+\frac{\eps}{\sigma^3}+\omega(\sigma)+\omega_2\left(\epfo\right)+ C\left(\eps^\alpha+\frac{1}{\sqrt{L}}\right)+\tilde{\omega}_1\left(\eps^{\frac{1}{2}-6\delta}\right)\right).
\end{equation*} Analogously to Theorem \ref{interior supsol} we define $$ \psi(x)= K_\eps2C_3 \left(\left(C_1-\left|x\right|^2\right)+C_2\left[\left(1-\frac{\gamma}{1+\left(\frac{d_\eps(x)}{2\eps}\right)^2}\right)\wedge\left(1-\frac{\gamma}{1+\left(\frac{\mu R}{4\eps}\right)^2}\right)\right]\right), $$ where $ C_1(\Omega):= 2\max_{x\in\overline\Omega}\left|x\right|^2+2\;\diam\left(\Omega\right)^2+4\;\diam\left(\Omega\right)+4 $ and $ C_2(\Omega),\:C_3(\Omega) $ are independent of $ \eps $. We denote $ d_\eps:=\dist(x,\bnd_\eps) $. We claim that $ \psi $ is a supersolution for $ \LL_{\Omega_\eps}^\eps $. This is true because the geometrical properties of $ \Omega_\eps$, in particular its regularity and the estimate for the radii of curvature, are identical to those for $ \Omega $. More precisely, if $ \eps<\eps_0 $ sufficiently small the minimal radius of curvature $ R_\eps $ for $ \bnd_\eps $ satisfies $ \frac{R}{2}<R_\eps<R $. Moreover, $ d_\eps(x)\leq d(x) $ and therefore  $ \Arrowvert \nabla_x^2 d_\eps\Arrowvert_{op}\leq \frac{1}{(1-\mu)\frac{R}{2}} $ for $ \mu\in(0,1) $ and $ d_\eps(x)<\frac{R}{2}\mu $. Hence, $ C_2,C_3,\gamma,\mu $ can be chosen as in Theorem \ref{interior supsol}. Thus, $ \psi $ has uniform upper and lower bound independently on $ \eps $ and all calculations we performed in Lemma \ref{first supsol} and Lemma \ref{second supsol} apply also for $ \LL_{\Omega_\eps}^\eps(\psi)(x) $. 
Hence, since $ d_\eps(x)\leq d(x) $ we estimate using equation \eqref{interior supsol 1}
\begin{equation}\label{convergence 2}
\LL_{\Omega_\eps}^\eps\left(\psi\right)(x)\geq K_\eps \left(e^{-\frac{d_\eps(x)}{2\eps}}+\eps^2\right)\geq K_\eps \left(e^{-\frac{d(x)}{2\eps}}+\eps^2\right).
\end{equation}
 
We remark that $ \psi\geq 8 K_\eps C_3 $ by the definition of $ C_1 $. Hence, multiplying $ \psi $ by $ K=\max\{1,\frac{1}{8C_3}\} $ we have $ K\psi\geq K_\eps $. We apply now use the maximum principle for $ \LL_{\Omega_\eps}^\eps $ acting on $ K\psi-(v_\sigma-u^\eps) $. Indeed, $ K\psi(x)\geq |v_\sigma(x)-u^\eps(x)| $ if $ x\in\Sigma_\eps $ by Lemma \ref{estimate sigmaeps} and the definition of $ K_\eps $. We notice also that estimates \eqref{convergence 1} and \eqref{convergence 2} imply $ \LL_{\Omega_\eps}^\eps\left(K\psi-\left(v_\sigma(x)-u^\eps(x)\right)\right)(x)\geq 0 $ as well as $ \LL_{\Omega_\eps}^\eps\left(K\psi-\left(u^\eps(x)-v_\sigma(x)\right)\right)(x)\geq 0 $ if $ x\in\hat{\Omega}_\eps $. Hence, using the maximum principle for the operator $ \LL_{\Omega_\eps}^\eps $ we conclude $ |u^\eps-v_\sigma|\leq K\psi\leq CK_\eps $ for $ x\in\Omega_\eps $ and for some constant $ C(\Omega, K, g_\nu) $ since $ \psi $ is bounded (cf. Theorem \ref{interior supsol}). Thus,
\begin{equation}\label{convergence3}
CK_\eps-v_\sigma(x)\leq u^\eps(x)\leq CK_\eps+v_\sigma(x)
\end{equation}
for all $ x\in\Omega_\eps $. Since $ \lim\limits_{\eps\to 0} K_\eps= C\left(\omega(\sigma)+\frac{1}{\sqrt{L}}\right) $, we obtain
\begin{equation*}\label{convergence4}
C\left(\omega(\sigma)+\frac{1}{\sqrt{L}}\right) -v_\sigma(x)=\liminf\limits_{\eps\to 0} u^\eps(x)\leq \limsup\limits_{\eps\to 0} u^\eps(x)=C\left(\omega(\sigma)+\frac{1}{\sqrt{L}}\right) +v_\sigma(x).
\end{equation*}
Letting first $ L\to\infty $ and then $ \sigma\to 0 $ we conclude
\begin{equation*}\label{convergence5}
v(x)=\liminf\limits_{\eps\to 0} u^\eps(x)\leq \limsup\limits_{\eps\to 0} u^\eps(x)=v(x)
\end{equation*}
and hence
\begin{equation*}\label{convergence6}
\lim\limits_{\eps\to 0}u^\eps(x)=v(x)
\end{equation*}
for all $ x\in\Omega $.
The convergence is not only pointwise but also uniform in every compact set. Let indeed $ A\subset \Omega $ be compact. Then there exists an $ \eps_0 $ such that $ A\subset \Omega_\eps $ for all $ \eps<\eps_0 $. Since $ CK_\eps $ is independent of $ x\in\Omega $ and $ |v-v_\sigma|\leq \omega(\sigma) $ uniformly in $ x\in\Omega $ equation \eqref{convergence3} implies
\begin{equation*}
\sup\limits_{x\in A}|u^\eps(x)-v(x)|\leq CK_\eps+\omega(\sigma)\overset{\eps\to 0,L\to\infty,\sigma\to 0}{\longrightarrow}0
\end{equation*}
 Hence, Theorem \ref{convergence up to the boundary} follows.
\end{proof}
\end{theorem}
\section{Diffusion approximation for space dependent absorption coefficient}
We could prove in the previous sections the convergence of the initial boundary value problem \eqref{bvpnoscattering} to the solution of the Laplacian for constant absorption coefficient. In this section we will show an analogous result for the case, when the absorption coefficient depends on $ x\in\Omega $, but it does not depend on the frequency.\\
\subsection{The limit problem and the boundary layer equation}
We assume, as we did throughout the paper, $ \Omega\subset\RR^3 $ to be a convex bounded domain with $ C^3 $-boundary and strictly positive curvature. From now on we also assume $ \alpha\in C^3(\overline{\Omega}) $ with $ 0<c_0\leq \alpha(x)\leq \Arrowvert \alpha\Arrowvert_{C^3}:=c_1<\infty $. We define $ \overline{\alpha}\in C_b^3(\RR) $ to be the extension of $ \alpha $ in the whole space with $  0<c_0\leq \overline{\alpha}(x)\leq c_1 $ and $ \overline{\alpha}\left.\right|_\Omega=\alpha $. For conivience we will denote $ \overline{\alpha} $ by $ \alpha $. We assume $ g_\nu $ to satisfy the same assumption as in the rest of the paper, namely $ g_\nu(n)\geq 0 $ with $ \int_0^\infty g_\nu(n)\;d\nu\in L^\infty\left(\Ss^2\right) $. We study the limit as $ \eps\to 0 $ of the following boundary value problem
	\begin{equation}\label{bvpalphax}
\begin{cases}
n\cdot \nabla_x I_\nu \left(x,n\right)= \frac{\alpha(x)}{\eps}\left(B_\nu\left(T\left(x\right)\right)-I_\nu \left(x,n\right)\right) & x\in\Omega,\\
\nabla_x\cdot \F=0 & x\in\Omega,\\
I_\nu \left(x,n\right)=g_\nu\left(n\right)& x\in\partial\Omega \text{ and }n\cdot N_x<0.
\end{cases}
\end{equation}
We proceed in the same way as in the case of constant absorption coefficient. Following the computation in Section 2.1 we obtain the limit problem in the interior $ \Omega $ for $ u(x)=4\pi\sigma T^4(x) $ as the elliptic equation
\begin{equation}\label{divproblem}
-\Div\left(\frac{1}{\alpha(x)}\nabla_xu(x)\right)=0. 
\end{equation}
For the boundary layer equation we argue similarly as in Section 2.2. Let $ x_0\in\bnd $ and let $ \Rot_{x_0} $ be the rigid motion in \eqref{rotp}. We rescale $ x=\frac{\eps}{\alpha(x_0)}\Rot_{x_0}^{-1}(y)+x_0 $ for $ y\in\frac{\alpha(x_0)}{\eps}\Rot_{x_0}\left(\Omega\right) $ and we define $ \overline{g}_\nu(n)=g_\nu\left(\Rot_{x_0}^{-1}(n)\right) $. Moreover, since $ \alpha $ is a $ C^3 $-function we also have for $ \eps $ sufficiently small that $ \alpha(x)=\alpha(x_0)+\frac{\eps}{\alpha(x_0)} \mathcal{O}(|y|)$, and hence taking $ \eps\to 0 $ we obtain once again for $ N=\Rot_{x_0}(N_{x_0}) $
\begin{equation*}\label{boundarylayerx}
\begin{cases}
n\cdot \grady\Inuy= \left(\Bnuy-\Inuy\right) & y\in\RR_+\times\RR^2,\\
\grady\cdot \F=0 & y\in\RR_+\times\RR^2,\\
\Inuy=\og_\nu\left(n\right)& y\in \{0\}\times\RR^2 \text{ and }n\cdot N<0.
\end{cases}
\end{equation*}
This implies, that the boundary layer equation for $ \ou(y_1,p)=4\pi\sigma T^4(y) $ is also in this case given by the integral equation \eqref{bvpgreyboundary1}
\begin{equation*}
\ou(y_1,p)-\int_0^\infty d\eta\;K(y_1-\eta)u(\eta,p)= \intnu\int_{n\cdot N_p<0}dn \;g_\nu(n)e^{-\frac{y_1}{\left|n\cdot N_p\right|}},
\end{equation*}
where $ K $ is the normalized exponential integral.
Hence, the whole theory developed in Section 3 is still valid and can be summarized by the Proposition \eqref{prop u}. 

Before moving to the rigorous proof of the convergence to the solution of the elliptic equation given in \eqref{divproblem} we remark that the function $ \oU_\eps(x,p):=\ou\left(\frac{\alpha(p)}{\eps}\Rot_p(x)\cdot e_1,p\right) $ solves the integral equation
\begin{equation}\label{ou x}
\oU_\eps(x,p)-\int_{\Pi_p}d\eta\; \frac{\alpha(p)e^{-\frac{\alpha(p)|x-\eta|}{\eps}}}{4\pi\eps|x-\eta|^2} \oU_\eps(\eta,p)=\intnu\int_{n\cdot N_{p}<0} dn\; g_\nu(n)e^{-\frac{\alpha(p)\left|x-x_{\Pi_p}(x,n)\right|}{\eps}}.
\end{equation}
\subsection{Rigorous proof of the convergence: equation for $ u^\eps $ and properties of the kernel}
We can now move to the proof of the convergence of the solution to the problem \eqref{bvpalphax} to the elliptic equation \eqref{divproblem}. We will follow all arguments given in Section 4 and change them were needed. Hence, first of all we find the integral equation that the sequence $ u^\eps(x)=4\pi\sigma T_\eps^4(x) $ associated to the solution $ I^\eps_\nu(x,n) $ of \eqref{bvpalphax} satisfies. We follow the computations of Section 4.1. Let $ x\in \Omega $, $ n\in\Ss^2 $, $ x_\Omega(x,n)\in\bnd $ the unique intersection point of the line $ \{x-tn : t>0\} $ with the boundary. Let $ s(x,n)=|x-x_\Omega(x,n)| $ and let us denote for $ x,z\in\overline{\Omega} $ by $ \int_{[x,z]}\alpha(\xi)\;ds(\xi) $ the integral along the line connecting $ x $ with $ z $, i.e. $ \int_{\left[x,x_\Omega(x,n)\right]}\alpha(\xi)\;ds(\xi)=\int_0^{s(x,n)}\alpha(x-tn)\;dt $. Solving the first equation in \eqref{bvpalphax} by characteristics we obtain
\begin{equation*}
I^\eps_\nu(x,n)=g_\nu(n)e^{-\int_{\left[x,x_\Omega(x,n)\right]}\frac{\alpha(\xi)}{\eps}\;ds(\xi)}+\int_0^{s(x,n)}e^{-\int_0^t\frac{\alpha(x-\tau n)}{\eps}d\tau}\frac{\alpha(x-tn)}{\eps}\Bnu{x-tn}\;dt.
\end{equation*}
Therefore, analogously as in Section 4.1 using that the flux is divergence free together with the first equation in \eqref{bvpalphax}, equation \eqref{sigma}, the characteristic solution of $ I^\eps $ and changing from spherical coordinates to space coordinates we obtain the following integral equation for $ u^\eps(x)=4\pi\sigma T_\eps^4(x)$
\begin{equation}\label{u x}
u^\eps(x)-\int_\Omega \frac{\alpha(\eta)e^{-\int_{[x,\eta]}\frac{\alpha(\xi)}{\eps}ds(\xi)}}{4\pi\eps|x-\eta|^2}u^\eps(\eta)\;d\eta=\int_0^\infty d\nu\int_{\Ss^2}dn\;g_\nu(n)e^{-\int_{\left[x,x_\Omega(x,n)\right]}\frac{\alpha(\xi)}{\eps}\;ds(\xi)}.
\end{equation}
For $ x,\eta\in\RR^3 $ we define the kernel $ K_\eps $ by\begin{equation}\label{Kepsx}
 K_\eps(x;\eta)=\frac{\alpha(\eta)e^{-\int_{[x,\eta]}\frac{\alpha(\xi)}{\eps}ds(\xi)}}{4\pi\eps|x-\eta|^2}. 
\end{equation} Notice that $ K_\eps $ has been defined in the whole $ \RR^3 $ extending $ \alpha $ by $ \overline{\alpha} $. We remark that $  K_\eps(x;\eta) $ is not symmetric, i.e. $ K_\eps(x;\eta)\ne K_\eps(\eta;x) $. 
In the following we summarize some properties of the kernel $ K_\eps $. 
\begin{prop}\label{prop K eps}
The kernel $ K_\eps $ defined in \eqref{Kepsx} has integral equal $ 1 $. Moreover, it can be decomposed in the following three different ways:
\begin{enumerate}
\item[(i)] Let $ x,\eta\in\RR^3 $, then $ K_\eps(x;\eta)=K^{\alpha(x)}_\eps(x-\eta)+R^x_\eps(x;\eta) $, where $ K^{\alpha(x)}_\eps(x-\eta)=\frac{\alpha(x)e^{-\frac{\alpha(x)|x-\eta|}{\eps}}}{4\pi\eps|x-\eta|^2} $. Moreover, the remainder satisfies
\begin{equation*}\label{remainder1}
\left|R^x_\eps(x;\eta)\right|\leq C(c_0,c_1)\frac{c_0e^{-\frac{c_0|x-\eta|}{\eps}}}{4\pi\eps|x-\eta|^2}
\begin{cases}
\left(|x-\eta|+\frac{|x-\eta|^2}{\eps}\right) & |x-\eta|<\sqrt{\eps}\\1&|x-\eta|\geq \sqrt{\eps}.
\end{cases}
\end{equation*}
\item[(ii)] Let $ x,\eta\in\RR^3 $ and $ p\in\bnd $ with $ |x-p|<\eps^{\frac{1}{2}+2\delta} $ for $ \delta>0 $ very small. Then $ K_\eps(x;\eta)=K^{\alpha(p)}_\eps(x-\eta)+R^p_\eps(x;\eta) $, where the remainder satisfies
\begin{equation*}\label{remainderp}\hspace{-.7cm}
\left|R^p_\eps(x;\eta)\right|\leq C(c_0,c_1)\frac{c_0e^{-\frac{c_0|x-\eta|}{\eps}}}{4\pi\eps|x-\eta|^2}
\begin{cases}
\left(|x-p|+|x-\eta|+\frac{|x-\eta|^2}{\eps}+\frac{|x-\eta|}{\eps}|x-p|\right) & |x-\eta|<\eps^{\frac{1}{2}+2\delta}\\1&|x-\eta|\geq \eps^{\frac{1}{2}+2\delta}.
\end{cases}
\end{equation*}
\item[(iii)] Let $ x,\eta\in\RR^3 $, then $ K_\eps(x;\eta)=K^{\alpha(x)}_\eps(x-\eta)+\mathcal{K}_\eps^1(x-\eta)+\mathcal{K}^2_\eps(x-\eta)+\tilde{R}^x_\eps(x;\eta) $, where 
\begin{equation*}\label{K1}
 \mathcal{K}^1_\eps(x-\eta)=-\frac{1}{2}\frac{\alpha(x)e^{-\frac{\alpha(x)|x-\eta|}{\eps}}}{4\pi\eps|x-\eta|^2}\nabla_x\alpha(x)\cdot(\eta-x)\frac{|x-\eta|}{\eps}
\end{equation*}and
\begin{equation*}\label{K2}
\mathcal{K}^2_\eps(x-\eta)=\frac{e^{-\frac{\alpha(x)|x-\eta|}{\eps}}}{4\pi\eps|x-\eta|^2}\nabla_x\alpha(x)\cdot(\eta-x).
\end{equation*}Moreover, the remainder satisfies
\begin{equation}\label{remainder2}
\left|\tilde{R}^x_\eps(x;\eta)\right|\leq C(c_0,c_1)
\begin{cases}
\frac{c_0e^{-\frac{c_0|x-\eta|}{\eps}}}{4\pi\eps|x-\eta|^2}\left(\frac{|x-\eta|^3}{\eps}+\frac{|x-\eta|^4}{\eps^2}\right) & |x-\eta|<\sqrt{\eps}\\\frac{c_0e^{-\frac{c_0|x-\eta|}{2\eps}}}{4\pi\eps|x-\eta|^2}(1+\eps+\eps^2)&|x-\eta|\geq \sqrt{\eps}.
\end{cases}
\end{equation}
Notice that (iii) is a refinement of (i).
\end{enumerate}
\begin{proof}
We start proving that the integral of $ K_\eps $ is $ 1 $. We compute changing to spherical coordinates as $ \eta=x-rn $
\begin{equation*}
\begin{split}
\int_{\RR^3}& \frac{ \alpha(\eta)e^{-\int_{[x,\eta]}\frac{ \alpha(\xi)}{\eps}ds(\xi)}}{4\pi\eps|x-\eta|^2} \;d\eta=\int_{\Ss^2}dn\int_0^\infty dr\; \frac{ \alpha(x-rn)}{4\pi\eps}e^{-\int_0^r\frac{ \alpha(x-tn)}{\eps}dt}\\
&=-\int_{\Ss^2}dn\int_0^\infty dr\;\frac{1}{4\pi} \frac{d}{dr}e^{-\int_0^r\frac{ \alpha(x-tn)}{\eps}dt}=\int_{\Ss^2}dn\;\frac{1}{4\pi}\left(1-e^{-\int_0^\infty\frac{ \alpha(x-tn)}{\eps}dt}\right)=1,
\end{split}
\end{equation*}
since $ \alpha\geq c_0>0 $.\\

We now proceed with the decompositions of the kernel. We start with claim (i). To this end we consider first of all $ |x-\eta|<\sqrt{\eps} $. We can expand by Taylor and get
\begin{equation}\label{taylor1}
\alpha(\eta)=\alpha(x)+\mathcal{O}\left(|x-\eta|\right)\;\;\;\;\;\text{ and }\;\;\;\;\;\int_0^1\frac{\alpha(x-\tau(x-\eta))d\tau|x-\eta|}{\eps}=\frac{\alpha(x)|x-\eta|}{\eps}+\mathcal{O}\left(\frac{|x-\eta|^2}{\eps}\right).
\end{equation}
Since $ |x-\eta|<\sqrt{\eps} $ we can conclude \begin{equation}\label{taylor2}
 e^{-\int_{[x,\eta]}\frac{\alpha(\xi)}{\eps}ds(\xi)}=e^{-\frac{\alpha(x)|x-\eta|}{\eps}}\left(1+\mathcal{O}\left(\frac{|x-\eta|^2}{\eps}\right)\right).
\end{equation} By assumptions $ c_0\leq \alpha(x)\leq c_1 $ and hence \eqref{taylor1} and \eqref{taylor2} imply claim (i) in the case $ |x-\eta|<\sqrt{\eps} $. In the case $ |x-\eta|\geq\sqrt{\eps} $ we use the rough estimate \begin{equation}\label{taylor4}
 \left|\alpha(\eta)e^{-\int_{[x,\eta]}\frac{\alpha(\xi)}{\eps}ds(\xi)}-\alpha(x)e^{-\frac{\alpha(x)|x-\eta|}{\eps}}\right|\leq 2c_1e^{-\frac{c_0|x-\eta|}{\eps}} .
\end{equation}

Concerning claim (ii) we argue similarly. Let $ p\in\bnd $ and $ |x-p|<\eptw $. Let us first of all consider $ |x-\eta|<\eptw $, then $ |\eta-p|<\eps^{\frac{1}{2}+\delta} $ for $ \eps>0 $ sufficiently  small. We expand $ \alpha $ again by Taylor around $ p $. Hence, similarly as before
\begin{equation}\label{taylor3}
\alpha(\eta)=\alpha(p)+\mathcal{O}\left(|\eta-p|\right)\;\text{ and }\; e^{-\int_{[x,\eta]}\frac{\alpha(\xi)}{\eps}ds(\xi)}=e^{-\frac{\alpha(p)|x-\eta|}{\eps}}\left(1+\mathcal{O}\left(\frac{|x-\eta|^2+|x-p||x-\eta|}{\eps}\right)\right).
\end{equation} 
Using on one hand $ |\eta-p|\leq|\eta-x|+|x-p| $ and on the other hand a rough estimate as in \eqref{taylor4} we conclude also the proof of claim (ii).\\

It remains to show claim (iii). We assume again first of all $ |x-\eta|<\sqrt{\eps} $. We expand $ \alpha $ using Taylor. Since $ \alpha\in C_b^3\left(\RR^3\right) $ all terms in the computation below are well-defined. Hence, we see
\begin{equation}\label{taylor5}
\alpha(\eta)=\alpha(x)+\nabla_x\alpha(x)\cdot (\eta-x)+\mathcal{O}\left(|x-\eta|^2\right)
\end{equation}
and
\begin{equation*}
\begin{split}
\frac{1}{\eps}\int_0^1\alpha(x-\tau(x-\eta))d\tau|x-\eta|=&\frac{1}{\eps}\int_0^1\alpha(x)+\tau \nabla_x\alpha(x)\cdot (\eta-x)+\mathcal{O}\left(|x-\eta|^2\right)d\tau|x-\eta|\\=&\frac{\alpha(x)|x-\eta|}{\eps}+\frac{1}{2}\nabla_x\alpha(x)\cdot(\eta-x)\frac{|x-\eta|}{\eps}+\mathcal{O}\left(\frac{|x-\eta|^3}{\eps}\right).
\end{split}
\end{equation*}
Thus, this implies for $ |x-\eta|<\sqrt{\eps} $ the estimate
\begin{equation}\label{taylor6}
e^{-\int_{[x,\eta]}\frac{\alpha(\xi)}{\eps}ds(\xi)}=e^{-\frac{\alpha(x)|x-\eta|}{\eps}}\left(1-\frac{1}{2}\nabla_x\alpha(x)\cdot(\eta-x)\frac{|x-\eta|}{\eps}+\mathcal{O}\left(\frac{|x-\eta|^3}{\eps}+\frac{|x-\eta|^4}{\eps^2}\right)\right).
\end{equation}
Equations \eqref{taylor5} and \eqref{taylor6} imply the decomposition $ K_\eps(x;\eta)=K^{\alpha(x)}_\eps(x-\eta)+\mathcal{K}_\eps^1(x-\eta)+\mathcal{K}^2_\eps(x-\eta)+\tilde{R}^x_\eps(x;\eta) $ and the estimate on $ \tilde{R}^x_\eps $ when $ |x-\eta|<\sqrt{\eps} $, while a rough estimate similar to \eqref{taylor4} and the well-known inequality $ e^{-|x|}|x|^n\leq C_ne^{-\frac{|x|}{2}} $ imply the claim (iii) for $ |x-\eta|\geq \sqrt{\eps} $.
\end{proof}
\end{prop}
This proposition is one of the key results which will allow us to generalize the results obtained in the previous section when the absorption coefficient depends smoothly enough on the space variable. There are two very important consequences. First of all, the Banach fixed point theorem guarantees a unique continuous bounded solution $ u^\eps $. Indeed, the continuity of the function $ \frac{1}{|x|^2} $ and its integrability in $ \Omega $ together with the uniform continuity in $ x $ and uniform boundedness of $ e^{-\int_{[x,\eta]}\frac{\alpha(\xi)}{\eps}ds(\xi)} $ imply that the operator $ A_g^\eps:C(\Omega)\to C(\Omega) $ is a selfmap, where we consider $$ A_g^\eps(u)(x)=\int_\Omega d\eta K_\eps(x;\eta)u(\eta)+\int_0^\infty d\nu\int_{\Ss^2}dn\; g_\nu(n)e^{-\int_{[x,x_\Omega(x,n)]}\frac{\alpha(\xi)}{\eps}ds(\xi)}.$$ Moreover, since the integral on $ \RR^3 $ of the kernel $ K_\eps $ is $ 1 $, arguing as in Section 4.1 we conclude that $ A_g^\eps $ is a contraction.\\

\noindent The second consequence is that we can prove for the integral operator $ \LL^\eps_\Omega(u)(x):=u(x)-\int_\Omega d\eta\;K_\eps(x;n)u(\eta) $ comparison properties identical to the one in Theorem \ref{max principle omega} (Maximum principle). Hence, $ u^\eps $ in \eqref{u x} is non-negative.  
\subsection{Rigorous proof of the convergence: uniform boundedness of $ u^\eps $}
Next we generalize Section 4.2 for the case $ \alpha(x)\in C^3(\overline{\Omega}) $. We want to show that the sequence $ u^\eps $ is uniform bounded in $ \eps $. As we have seen in Section 5.2 the maximum principle of Theorem \ref{max principle omega} holds. We want to apply it with a suitable supersolution. Indeed, using the notation $ \eps d_\eps~=~d(x) $ as in Lemma \ref{second supsol} we see $ \int_{[x,x_\Omega(x,n)]}\frac{\alpha(\xi)}{\eps}ds(\xi)\geq \frac{c_0|x-x_\Omega|}{\eps}\geq c_0d_\eps(x) $ and thus $ \left|\LL^\eps_\Omega\left(u^\eps\right)(x)\right|\leq \Arrowvert g\Arrowvert_1e^{-c_0d_\eps(x)}$.

The supersolution we will consider is similar to the one constructed in Theorem \ref{interior supsol}. However, it is now not true anymore that there exists a constant $ C>0 $ such that $ \LL^\eps_\Omega(C-|x|^2)\geq 2\eps^2 $ for every $ x\in\Omega $. This can be already been looking at the expected limit elliptic operator $ L=-\Div\left(\frac{1}{\alpha(x)}\nabla_x \right) $, for which the inequality $ L\left(C-|x|^2\right)\geq 0 $ holds for arbitrary functions $ \alpha $ depending on $ x $ only for small values of $ |x| $. We remark that the supersolution constructed in Theorem \ref{interior supsol} contains two parts, namely a multiple of $ C-|x|^2 $ and a function proportional to the function $ \psi $ constructed in Lemma \ref{second supsol}. The role of the function $ \psi $ is to control the value of the operator $ \LL^\eps_\Omega $ near the boundary $ \bnd $. The contribution of $ \psi $ is relevant only in the region where $ e^{-d_\eps(x)}\geq \eps^2 $. We will see that the function $ \psi $ still gives a supersolution for the new operator $ \LL^\eps_\Omega $ (up to an error of order $ \eps^2 $). Therefore, in order to prove the analogous of Theorem \ref{interior supsol} we need to replace the function $ C-|x|^2 $ by exponential functions.
\begin{theorem}\label{interior supsol x}
	There exist suitable constants $ 0<\mu<1 $, $ 0<\gamma(\mu)<\frac{1}{3} $, $ C_1,\;C_2,\;C_3>0 $, $ \lambda>0 $ large enough and there exists some $ \eps_0>0 $ such that the function 
\begin{equation*}\label{supsol1 x}
\Phi^\eps(x)=C_3 \Arrowvert g\Arrowvert_1\left[ \left(e^{\lambda D}+C_1-e^{\lambda x_1}\right)+C_2\left(\left(1-\frac{\gamma}{1+\left(\frac{c_0d(x)}{\eps}\right)^2}\right)\wedge\left(1-\frac{\gamma}{1+\left(\frac{c_0\mu R}{\eps}\right)^2}\right)\right)\right],
\end{equation*}
for  $ a\wedge b\:=\min\left(a,b\right) $, $ R>0 $ the minimal radius of curvature, $ D=\diam(\Omega) $ and $ d(x):=\dist\left(x,\partial\Omega\right) $, 
satisfies $ \LL_\Omega^\eps\left(\Phi^\eps\right)(x)\geq \Arrowvert g\Arrowvert_1 e^{-\frac{c_0d(x)}{\eps}} $ in $ \Omega $ uniformly for all $ \eps<\eps_0 $. Moreover, the solutions $ u^\eps $ of \eqref{u x} are uniformly bounded in $ \eps $.
\begin{proof}
We will use the notation of Theorem \ref{interior supsol} and Lemma \ref{first supsol} and \ref{second supsol}. Notice first of all that \begin{equation}\label{rough estimate K}
 \frac{c_0}{c_1}K^{c_1}_\eps(x-\eta)\leq K_\eps(x;\eta)\leq \frac{c_1}{c_0}K^{c_0}_\eps(x-\eta) ,
\end{equation} where we used the notation of Proposition \ref{prop K eps}. Moreover, using the decomposition of claim (i) in Proposition \ref{prop K eps} we can always estimate for any $ n\in\mathbb{N} $ the integral $ \int_{\RR^3}|R^x_\eps(x;\eta)||(\eta-x)|^n\;d\eta $ by
\begin{equation}\label{R x eps}
\begin{split}
 \int_{\RR^3}|R^x_\eps(x;\eta)||(\eta-x)|^n\;d\eta\leq& C(c_0,c_1)\int_{|x-\eta|<\sqrt{\eps}}K_\eps^{c_0}(x-\eta)\left(\frac{|\eta-x|^{2+n}}{\eps}+|x-\eta|^{1+n}\right)\;|d\eta\\&+C(c_0,c_1)\int_{|x-\eta|\geq\sqrt{\eps}}K_\eps^{c_0}(x-\eta)|(\eta-x)|^n\;d\eta\\
 \leq& C(c_0,c_1)\eps^{n+1}\int_0^\infty e^{-r}r^{n+1}(1+r)\;dr+C(c_0,c_1)\eps^n\int_{\frac{c_0}{\sqrt{\eps}}}^\infty e^{-r}r^n\;dr\\
 \leq&C(c_0,c_1,n)\eps^{n+1}+C(c_0,c_1,n)\eps^n e^{-\frac{c_0}{2\sqrt{\eps}}}\leq C(c_0,c_1,n)\eps^{n+1},
\end{split}
\end{equation}
where we used $ e^{-|x|}|x|^{m}\leq C_m $.
Let now $ x\in\Omega $. Let $ x_0\in\bnd $ be a point such that $ |x-x_0|=d(x) $. Then we estimate with the help of Proposition \ref{prop K eps} 
\begin{equation*}\label{L 1}
\begin{split}
\LL^\eps_\Omega(1)=& \int_{\Omega^c}d\eta\;K_\eps(x;\eta)\geq \int_{\Pi_{x_0}}d\eta\;K_\eps(x;\eta)\geq \frac{c_0}{c_1}\int_{\RR_-\times\RR^2}K^{c_1}_\eps(d(x)e_1-\eta)\\=&\frac{c_0}{c_1}\int_{-\infty}^{-c_1\deps x}K(y)\;dy\geq\begin{cases}
\frac{c_0}{c_1}\nu_{c_1M_0}& \deps{x}\leq M_0,\\0 & \forall x\in\Omega.
\end{cases}
\end{split}
\end{equation*}
This estimate will play a crucial role the proof. We proceed considering the function $ \varphi(x)=e^{\lambda D}-e^{\lambda x_1} $. It is not difficult to see that for $ x\in\Omega $ we get $ \varphi(x)\geq 0 $. Moreover, expanding by Taylor $ e^{\lambda \eta_1}=e^{\lambda x_1}+\lambda e^{\lambda x_1}(\eta_1-x_1)+\frac{\lambda^2}{2}e^{\lambda x_1}(\eta_1-x_1)^2+E_3(x;\eta) $ with $ |E_3(x;\eta)|\leq \frac{\lambda^3}{6}e^{\lambda D}|x_1-\eta_1|^3 $ we compute
\begin{equation}\label{L exp}
\begin{split}
\LL^\eps_\Omega&\left(e^{\lambda x_1}\right)=e^{\lambda x_1}-\int_\Omega K_\eps(x;\eta)e^{\lambda \eta_1}d\eta\\
=&e^{\lambda x_1}\int_{\Omega^c}K_\eps(x;\eta)\;d\eta-\lambda e^{\lambda x_1}\int_\Omega K_\eps(x;\eta)(\eta_1-x_1)\;d\eta\\&-\frac{\lambda^2}{2}e^{\lambda x_1}\int_\Omega K_\eps(x;\eta)(\eta_1-x_1)^2\;d\eta-\int_\Omega K_\eps(x;\eta)E_3(x;\eta)\;d\eta.
\end{split}
\end{equation}
The first term on the right hand side can be controlled by $ e^{\lambda D} $. For the second term we use the decomposition of claim (i) in Proposition \ref{prop K eps}, the symmetry of the operator $ K^{\alpha(x)}_\eps $ and the estimate \eqref{R x eps}. Moreover, we remind that $ \Omega^c\subset B^c_{d(x)}(x) $. Hence,
\begin{equation}\label{L exp 1}
\begin{split}
-\lambda e^{\lambda x_1}\int_\Omega K_\eps(x;\eta)(\eta_1-x_1)\;d\eta=&\lambda e^{\lambda x_1}\int_{\Omega^c} K^{\alpha(x)}_\eps(x-\eta)(\eta_1-x_1)\;d\eta-\lambda e^{\lambda x_1}\int_\Omega R_\eps(x;\eta)(\eta_1-x_1)\;d\eta\\
\leq&\lambda \frac{c_1}{c_0^2}\eps e^{-\frac{c_0\deps{x}}{2}}e^{\lambda x_1}+C(c_0,c_1)\lambda e^{\lambda x_1}\eps^2.
\end{split}
\end{equation}
For the third term in equation \eqref{L exp} we can proceed as follows using the decomposition of claim (i) of Proposition \ref{prop K eps} and \eqref{R x eps}
\begin{equation}\label{L exp 2}
\begin{split}
\frac{\lambda^2}{2}e^{\lambda x_1}\int_\Omega &K_\eps(x;\eta)(\eta_1-x_1)^2\;d\eta=\frac{\lambda^2}{2}e^{\lambda x_1}\int_{\RR^3}  K^{\alpha(x)}_\eps(x-\eta)(\eta_1-x_1)^2\;d\eta\\&-\frac{\lambda^2}{2}e^{\lambda x_1}\int_{\Omega^c} K^{\alpha(x)}_\eps(x-\eta)(\eta_1-x_1)^2\;d\eta+\frac{\lambda^2}{2}e^{\lambda x_1}\int_\Omega R_\eps(x;\eta)(\eta_1-x_1)^2\;d\eta\\
\geq& \frac{2\eps^2}{3c_1^2}\lambda^2e^{\lambda x_1}-\frac{\eps^2}{c_0^2}\lambda^2e^{\lambda x_1}e^{-\frac{c_0\deps{x}}{2}}-c(c_0,c_1)\lambda^2 e^{\lambda x_1}\eps^3.
\end{split}
\end{equation}
Finally, using the estimate \eqref{rough estimate K} for $ K_\eps $ we compute for the term containing the error $ E_3(x;\eta) $
\begin{equation}\label{L exp 3}
\int_\Omega K_\eps(x;\eta)E_3(x;\eta)\;d\eta\leq \lambda^3e^{\lambda D}C(c_0,c_1)\eps^3\int_{\RR^3} e^{-r}r^3\leq C(c_0,c_1)\lambda^3e^{\lambda D}C(c_0,c_1)\eps^3.
\end{equation}
Hence, \eqref{L exp 1}, \eqref{L exp 2} and \eqref{L exp 3} imply for $ \lambda>0 $ large enough and $ 0<\eps<1 $ sufficiently small
\begin{equation*}\label{estimate exp}
\begin{split}
\LL^\eps_\Omega\left(e^{\lambda D}-e^{\lambda x_1}\right)\geq& \frac{2\eps^2}{3c_1^2}\lambda^2e^{\lambda x_1}-C(c_0,c_1)\lambda e^{\lambda x_1}\left(\eps e^{-\frac{c_0\deps{x}}{2}}+\eps^2+\eps^2\lambda e^{-\frac{c_0\deps{x}}{2}}\right)-c(c_0,c_1,\lambda,D)\eps^3\\
\geq& \begin{cases}
-A_1\eps e^{-\frac{c_0\deps{x}}{2}}& \text{ if }d(x)<\frac{2\eps}{c_0}\ln\left(\frac{1}{\eps}\right),\\A_2 \eps^2&\text{ if } d(x)\geq\frac{2\eps}{c_0}\ln\left(\frac{1}{\eps}\right),
\end{cases}
\end{split}
\end{equation*}
where $ A_1(c_0,c_1,\lambda,D)>0 $ and $ A_2(c_0,c_1,\lambda,D)>0 $ are constants independent of $ \eps $.\\

\noindent We proceed with the estimate for the operator acting on $ \psi(x)=\left(1-\frac{\gamma}{1+\left(\frac{c_0d(x)}{\eps}\right)^2}\right)\wedge\left(1-\frac{\gamma}{1+\left(\frac{c_0\mu R}{\eps}\right)^2}\right) $. Arguing similarly as in Lemma \ref{second supsol} we can show that for $ \mu>0 $  and  $ 0<\gamma<\frac{1}{3} $ small enough and $ \eps<\eps_1<\frac{R\mu^3}{2} $ sufficiently small there exists constants $ A_3(R,\Omega, \mu,\gamma)>0 $ and $ A_4(R,\mu,\gamma)>0 $ such that
\begin{equation*}\label{estimate psi}
\LL^\eps_\Omega(\psi)(x)\geq \begin{cases}
A_3\frac{1}{\left(1+(c_0\deps{x})^2\right)^2}& \text{ if }0<d(x)\leq \frac{R\mu}{2},\\-A_4\eps^2 & \text{ if }\frac{R\mu}{2}<d(x)<R\mu, \\0 & \text{ if } d(x)\geq R\mu .
\end{cases}
\end{equation*}
The proof works following the same steps of Lemma \ref{second supsol}. Indeed, for the regions $ \{d(x)\geq R\mu\} $ and $ \{0<d(x)<M\eps\} $, for $ M=\frac{1}{\mu^2} $, the proof does not change. We use the estimate \eqref{rough estimate K} for $ K_\eps $ and the fact that its integral in $ \RR^3 $ is $ 1 $. The constant $ \gamma $ must be chosen sufficiently small so that $ \gamma<\frac{\nu_{c_1M}c_0}{2c_1} $. For the regions $ \{M\eps\leq d(x)\leq \frac{R\mu}{2}\} $ and $ \{\frac{R\mu}{2}<d(x)<R\mu\} $ we use again the Taylor expansion, the estimates on $ K_\eps $, the fact that $ \eps<\frac{R\mu^2}{2} $ and $ M=\frac{1}{\mu^2} $. Beside the fact that $ K_\eps $ has integral $ 1 $ in $ \RR^3 $ we use its decomposition $ K_\eps=K_\eps^{\alpha(x)}+R^x_\eps $ according to claim (i) in Proposition \ref{prop K eps} and the estimate \eqref{R x eps} for the remainder. In every computation where we used the symmetry of the kernel, e.g. for the first term of the Taylor expansion, we decompose $ K_\eps $ and apply the symmetry argument for $ K_\eps^{\alpha(x)} $ and estimate the remainder. We omit the details of the proof, since the computations are similar to those in Lemma \ref{second supsol}.\\

We now finish the proof of Theorem \ref{interior supsol x}. Let $ \eps_0\leq\eps_1 $ such that $ \frac{R\mu}{2}>\frac{2\eps}{c_0}\ln\left(\frac{1}{\eps}\right) $ for all $ 0<\eps<\eps_0 $. Let $ C_2=\frac{A_2}{2A_4} $. Moreover, since $ (1+x^2)^2e^{-\frac{x}{2}}\to 0 $ as $ x\to \infty $ there exists some $ M_0>0 $ such that if $ \deps{x}\geq M_0 $ then
$$ -A_1\eps e^{-\frac{c_0\deps{x}}{2}}+\frac{C_2A_3}{\left(1+(c_0\deps{x})^2\right)^2}\geq \frac{C_2A_3}{2\left(1+(c_0\deps{x})^2\right)^2}\geq \frac{C_2A_3}{12}e^{-c_0\deps{x}}. $$ Let $ C_1>0 $ satisfy $ C_1>\frac{A_1c_1}{c_0\nu_{c_1M_0}} $. Then for all $ \deps{x}<M_0 $ we obtain $$ \LL^\eps_\Omega(C_1)\geq C_1\frac{c_0}{c_1}\nu_{c_1M_0}> A_1 \eps e^{-\frac{c_0\deps{x}}{2}} .$$ Finally, taking $ C_3^{-1}=\min\{\frac{A_2}{2}, \frac{C_2A_3}{12}\} $ we get the desired lower estimate $ \LL^\eps_\Omega\left(\Phi^\eps\right)(x)\geq \Arrowvert g\Arrowvert_1 e^{-\frac{c_0d(x)}{\eps}}  $.\\
We conclude the proof of this Theorem applying the maximum principle of Theorem \ref{max principle omega} to the continuous function $ \Phi^\eps-u^\eps $.

\end{proof}
\end{theorem}
\begin{remark}
We notice for further reference that we have obtained a stronger estimate than $ \LL^\eps_\Omega\left(\Phi^\eps\right)(x)\geq \Arrowvert g\Arrowvert_1 e^{-\frac{c_0d(x)}{\eps}}  $, namely
\begin{equation}\label{Phi}
\LL^\eps_\Omega\left(\Phi^\eps\right)(x)\geq\Arrowvert g\Arrowvert_1\begin{cases}
\frac{6}{\left(1+(c_0\deps{\eps})^2\right)^2} & \text{ if }d(x)<\frac{2\eps}{c_0}\ln\left(\frac{1}{\eps}\right)\\ \eps^2+\frac{6}{\left(1+(c_0\deps{\eps})^2\right)^2} & \text{ if }\frac{2\eps}{c_0}\ln\left(\frac{1}{\eps}\right)\leq d(x)<\frac{R\mu}{2}\\\eps^2 & \text{ if } d(x)\geq\frac{R\mu}{2}
\end{cases}\geq \Arrowvert g\Arrowvert_1 \eps^2,
\end{equation}
since $ \frac{6}{\left(1+(c_0\deps{\eps})^2\right)^2}\geq e^{-\frac{c_0d(x)}{\eps}} \geq \eps^2 $ if $ d(x)<\frac{2\eps}{c_0}\ln\left(\frac{1}{\eps}\right) $. This estimate will be important at the end of the paper.
\end{remark}
\subsection{Rigorous proof of the convergence: estimates of $ u^\eps-\ou$ near the boundary $ \bnd $}
In this section we will extend the results of Section 4.3 also for the case of space dependent absorption coefficient. We will show a slightly different result than the one in Section 4.3. We will prove that $ \oU_\eps(x,p) $ as defined in \eqref{ou x} is a good approximation of $ u^\eps(x) $ in a neighborhood of $ p\in\bnd $ of size close to $ \eps^{\frac{2}{3}} $ istead of $ \eps^{\frac{1}{2}} $. We will use once more the maximum principle of Theorem \ref{max principle omega}. We start with the estimate of $ \LL^\eps_\Omega\left(\oU_\eps(\cdot,p)-u^\eps\right) $.
\begin{lemma}\label{L ou-u x}
Let $ p\in\bnd $ and let $ \Rot_p $ be the rigid motion defined in \eqref{rotp}. Then the following holds for $ x\in\Omega $, $ \delta>0 $ sufficiently small and independent of $ \eps $ and a suitable $ 0<A<\frac{1}{4} $ and constant $ C>0 $
\begin{equation}\label{L ou-u x 1}
\left|\LL_\Omega^\eps\left(\ou\left(\alpha(p)\frac{\Rot_p(\cdot)\cdot e_1}{\eps},p\right)-u^\eps\right)(x)\right|\leq C e^{-Ac_0\deps{x}}\begin{cases}1,& \forall x\in\Omega\\\eps^\delta & \text{ if }|x-p|<\eptw.
\end{cases}
\end{equation}
\begin{proof}
We start with the rough estimate for $ x\in\Omega $. The operator can be estimated by $$ \left|\LL_\Omega^\eps\left(\oU_\eps(x,p)-u^\eps(x)\right)\right|\leq \left|\LL_\Omega^\eps\left(\oU_\eps(x,p)\right)\right|+\left|\LL_\Omega^\eps\left(u^\eps(x)\right)\right| .$$ As we have seen in Section 5.3 it is always true that $ \left|\LL_\Omega^\eps\left(u^\eps(x)\right)\right| \leq \Arrowvert g\Arrowvert_1e^{-c_0\deps{x}} $. Hence, we have to consider $ \oU_\eps(x,p)=\ou_\infty(p)+\oV_\eps(x,p) $ with $ \oV_\eps(x,p)=\oU_\eps(x,p)-\ou_\infty(p) $ and thus by Lemma \ref{exp rate} $ \left|\oV_\eps(x,p)\right|\leq Ce^{-\frac{c_0|\Rot_p(x)\cdot e_1|}{2\eps}} $.

By the geometry of the problem we can estimate $ |\Rot_p(x)\cdot e_1|\geq d(x) $. Indeed, let $ x_p\in\bnd $ be the unique intersection point of the line $ \{x+tN_p:t>0\} $ with the boundary $ \bnd $, i.e. $ x_p=x+t_pN_p $. Then, since $ \Omega $ is convex, \begin{equation}\label{Rotpx distance}
 |\Rot_p(x)\cdot e_1|=(x-p)\cdot(-N_p)\geq (x-x_p)\cdot(-N_p) =t_p=|x-x_p|\geq d(x).
\end{equation}Hence, using also that $ |x-\eta|\geq |d(\eta)-d(x)| $ we compute 
\begin{equation*}\label{L ou-u x 2}
\begin{split}
\left|\LL^\eps_\Omega\right.&\left.\left(\oU_\eps(x,p)\right)\right|\leq \ou_\infty(p)\int_{\Omega^c}K_\eps(x;\eta)d\eta+Ce^{-\frac{c_0}{2}\deps{x}}+C\int_\Omega K^\eps(x;\eta)e^{-\frac{c_0}{2}\deps{\eta}}\;d\eta\\
\leq& C(c_0,c_1)\Arrowvert\ou_\infty\Arrowvert_{\infty}e^{-\frac{c_0}{2}\deps{x}}\int_0^\infty e^{-\frac{r}{2}}\;dr+ Ce^{-\frac{c_0}{2}\deps{x}}\\
&+C(c_0,c_1)\int_{d(\eta)<d(x)}\frac{e^{-\frac{c_0|x-\eta|}{2\eps}}}{4\pi\eps|x-\eta|^2}e^{-\frac{c_0}{2}\deps{x}}\;d\eta+Ce^{-\frac{c_0}{2}\deps{x}}\int_{d(\eta)\geq d(x)}K_\eps(x;\eta)\;d\eta\\
\leq& C(c_0,c_1,\Arrowvert \ou_\infty\Arrowvert_{\infty})e^{-\frac{c_0}{2}\deps{x}}.
\end{split}
\end{equation*}
We now prove the estimate when $ |x-p|<\eptw $. In this case we use the decomposition $ K_\eps=K^{\alpha(p)}_\eps+R_\eps^{p} $ in claim (ii) of Proposition \ref{prop K eps}. Similarly as in \eqref{R x eps} we can estimate
\begin{equation}\label{R p eps}
\int_{\RR^3}|R^p_\eps(x;\eta)||x-\eta|^n\;d\eta\leq C(c_0,c_1)\eps^n\eptw.
\end{equation}
Hence, using the decomposition above we compute for $ |x-p|<\eptw $
\begin{equation*}
\begin{split}
\left|\LL_\Omega^\eps\right.&\left.\left(\oU_\eps(x,p)-u^\eps(x)\right)\right|\\\leq&\left|\int_0^\infty d\nu\int_{\Ss^2}dn\;g_\nu(n)e^{-\int_{\left[x,x_\Omega(x,n)\right]}\frac{\alpha(\xi)}{\eps}\;ds(\xi)}-\intnu\int_{n\cdot N_{p}<0} dn\; g_\nu(n)e^{-\frac{\alpha(p)\left|x-x_{\Pi_p}(x,n)\right|}{\eps}}\right|\\
&+\int_{\Pi_p\setminus \Omega}d\eta\; K_\eps^{\alpha(p)}(x-p)\oU_\eps(\eta,p)+\int_\Omega d\eta\; |R^p_\eps(x;\eta)|\oU_\eps(\eta)\\
=&I_1+I_2+I_3.
\end{split}
\end{equation*}
We estimate now these three terms. Analogous to Lemma \ref{estimate op ou-u eps} we decompose $ \Ss^2=\mathcal{U}_1\cup\mathcal{U}_2\cup\mathcal{U}_3 $. For the term $ I_1 $ we only have to notice that if $ n\in\mathcal{U}_1 $ then $ |x_\Omega-x|\leq \eps^{\frac{1}{2}+\delta}(C(\Omega)\eps^{\frac{1}{2}}+1+\eps^\delta)\leq 2 \eps^{\frac{1}{2}+\delta}$ for $ \eps>0 $ sufficiently small. Hence, $ \frac{|\xi-p||x-x_\Omega|}{\eps}<4\eps^{2\delta}<1 $ for any $ \xi=x-t(x-x_\Omega) $, $ 0\leq t\leq1 $ and $ \eps>0 $ sufficiently small. Thus, we obtain using the Taylor expansion on $ \alpha(\xi) $ as we did in \eqref{taylor3}
\begin{equation*}
\begin{split}
I_1\left.\right|_{\mathcal{U}_1}\leq&\int_0^\infty d\nu\int_{\mathcal{U}_1}dn\;g_\nu(n)\left|e^{-\frac{\alpha(p)|x-x_\Omega(x,n)|}{\eps}}-e^{-\frac{\alpha(p)\left|x-x_{\Pi_p}(x,n)\right|}{\eps}}\right|\\
&+\int_0^\infty d\nu\int_{\mathcal{U}_1}dn\;g_\nu(n)\left|e^{-\int_{\left[x,x_\Omega(x,n)\right]}\frac{\alpha(\xi)}{\eps}\;ds(\xi)}-e^{-\frac{\alpha(p)|x-x_\Omega(x,n)|}{\eps}}\right|\\\leq & C(\Omega)\Arrowvert g\Arrowvert_{\infty}\eps^{\delta}e^{-c_0\deps{x}}+C \Arrowvert g\Arrowvert_{\infty}\int_{\mathcal{U}_1}dn\;e^{-\frac{\alpha(p)|x-x_\Omega(x,n)|}{\eps}}\eps^{2\delta}\leq C(\Omega)\Arrowvert g\Arrowvert_{\infty}\eps^{\delta}e^{-c_0\deps{x}},
\end{split}
\end{equation*}
where to estimate the first term in the first inequality we used Lemma \ref{estimate op ou-u eps}, specifically equation \eqref{op ou-u eps 19}, and to estimate the second term we expanded $ \alpha(\xi) $ at $ \xi=p $. To estimate the contributions in the regions $ \mathcal{U}_2 $ and $ \mathcal{U}_3 $, the estimate $ \int_{\left[x,x_\Omega(x,n)\right]}\frac{\alpha(\xi)}{\eps}\;ds(\xi)\geq -c_0\frac{|x-x_\Omega(x,n)|}{\eps} $ together with the result of Lemma \ref{estimate op ou-u eps} implies the bound on $ I_1 $. The term $ I_2 $ can be handled exactly as in Lemma \ref{estimate op ou-u eps}. Finally, for $ I_3 $ we use the uniform boundedness of $ \ou(y,p) $ and equation \eqref{R p eps}. This implies
\begin{equation*}
\left|\LL_\Omega^\eps\left(\oU_\eps(x,p)-u^\eps(x)\right)\right|\leq C(\Omega,c_0,c_1,g) \eps^\delta e^{-c_0\deps{x}}+C\eptw.
\end{equation*}  
We conclude by interpolation. Indeed, if $ d(x)\leq \frac{4\eps}{c_0}\ln\left(\eps^{-\frac{1}{2}-\delta}\right) $, then $ \eptw\leq \eps^\delta e^{-\frac{c_0\deps{x}}{4}} $, while if $ d(x)>\frac{4\eps}{c_0}\ln\left(\eps^{-\frac{1}{2}-\delta}\right)  $ we use the global estimate to get $ e^{-\frac{c_0\deps{x}}{2}}\leq \eps^{\frac{1}{2}+\delta}e^{-\frac{c_0\deps{x}}{4}} $.
\end{proof}
\end{lemma}
It only remains to adapt the supersolution $ W_{\eps,L} $ of Proposition \ref{supsol ou-u eps}. As we anticipated at the beginning of Section 5.4 we are going to prove that $ \ou $ is a good approximation of $ u^\eps $ in a neighborhood of $ p\in\bnd $ of size close to $ \eps^{\frac{2}{3}} $. First of all we notice that if $ \delta<\frac{1}{12} $ then $ \frac{2}{3}>\frac{1}{2}+2\delta $ and hence we have also that
\begin{equation*}\label{2/3}
\left|\LL_\Omega^\eps\left(\ou\left(\frac{\Rot_p(\cdot)\cdot e_1}{\eps},p\right)-u^\eps\right)(x)\right|\leq C e^{-Ac_0\deps{x}}\begin{cases}1& \text{ if }|x-p|\geq \eps^{\frac{2}{3}}\\\eps^\delta & \text{ if }|x-p|<\eps^{\frac{2}{3}}.\end{cases}
\end{equation*}
The result we prove is the following
\begin{prop}\label{supsol ou-u eps x}
	Let $ p\in\bnd $, $ 0<A<\frac{1}{4} $ the constant of Lemma \ref{L ou-u x}. Let $ L>0 $ large enough and $ 0<\eps<1 $ sufficiently small. Let $ 0<\delta<\frac{1}{12} $. Then there exists a non negative continuous function $ W_{\eps, L}:\Omega\to \RR_+ $ such that 
	\begin{equation*}\label{supsol case x}
	\begin{cases}
	W_{\eps,L}\geq C>0 &\text{ for } \left|\Rot_p(x)\cdot e_i\right|\geq \eps^{\frac{2}{3}+\delta} ;\\\LL_\Omega^\eps\left(W_{\eps,L}\right)(x)\geq C \eps^\delta e^{-\frac{Ad(x)}{\eps}}&\text{ for }\left|\Rot_p(x)\cdot e_i\right|< \eps^{\frac{2}{3}+\delta}; \\0\leq W_{\eps,L}\leq C\left(\eps^\alpha+\frac{1}{\sqrt{L}}\right) &\text{ for } \left|\Rot_p(x)\cdot e_i\right|< \eps^{\frac{2}{3}+2\delta}, 
	\end{cases}
	\end{equation*}
	for some constant $ C>0 $ and $ \alpha>0 $.
\end{prop}
This proposition implies arguing as in Section 4.3, the following corollary (see Corollary \ref{good estimate}).
\begin{corollary}\label{good estimatex}
	There exists a constant $ C>0 $, a large $ L>0 $ and an $\alpha>0 $ independent of $ x,p,\eps $ such that 
	\begin{equation*}\label{good estimate 1x}
	\left|\ou\left(\frac{R_p(\cdot)\cdot e_1}{\eps},p\right)-u^\eps\right|(x)\leq C \left(\eps^\alpha+\frac{1}{\sqrt{L}}\right)
	\end{equation*}
	for all $ |x-p|<\eps^{\frac{2}{3}+2\delta}$.
	\end{corollary}
In order to adapt the supersolution $ W_{\eps,L} $ of Proposition \ref{supsol ou-u eps} in this case we start considering a new slightly different geometrical setting. Once more we denote for simplicity $ x_i=\Rot_p(x)\cdot e_i $. We define now for $ i=2,3 $ the radii $ \rho_i^\pm(x)=\sqrt{\left(x_1+\frac{L}{2}\eps\right)^2+\left(x_i\pm \eps^{\frac{2}{3}+\delta}\right)^2} $ and the angles $ \theta_i^\pm(x) $ given by $ \cos\left(\theta_i^\pm \right)=\frac{1}{\rho_i^\pm(x)}\left(x_1+\frac{L}{2}\eps\right) $. We construct then the function $ W_{\eps,L} $ using now these definitions for $ \rho_i^\pm(x) $ and $ \theta_i^\pm(x) $  analogously to Section 4.3
\begin{equation}\label{Welx}
W_{\eps,L}(x)=\sum_{i=2}^3 \left(W^+_i(x)+W^-_i(x)\right)+\frac{\tilde{C}}{\sqrt{L}}\phi_{\frac{1}{8},\eps}+
C\eps^{\delta}\phi_{A,\eps},
\end{equation} where $ \phi_{A,\eps}=\Phi^{\frac{\eps}{A}} $ the supersolution defined in Theorem \ref{interior supsol x}, $ C,\tilde{C}>0 $ some suitable constants and $ W^\pm_i=F^\pm_i(x)+G^\pm_i(x)+H^\pm_i(x) $  given by the auxiliary functions defined in \eqref{Fpm},\eqref{Gpm} and \eqref{Hpm} adapted to the new geometrical setting. Analogously to Section 4.3 we define the following subsets of $ \Omega $ for $ i=2,3 $. 
\begin{equation*}
\mathcal{C}^+_{i}:=\left\{x\in\Omega :\; x_i\leq -\eps^{\frac{2}{3}+\delta} \text{ or } |x_i|<\eps^{\frac{2}{3}+\delta}, x_1\geq \eps^{\frac{2}{3}+\delta}\right\};
\end{equation*}
\begin{equation*}
\mathcal{C}^-_{i}:=\left\{x\in\Omega :\; x_i\geq \eps^{\frac{2}{3}+\delta}\text{ or } |x_i|<\eps^{\frac{2}{3}+\delta}, x_1\geq \eps^{\frac{2}{3}+\delta}\right\};
\end{equation*}
\begin{equation*}
\mathcal{C}_{\delta}:=\left\{x\in\Omega :\; x_1< \eps^{\frac{2}{3}+\delta} \text{ and } |x_i|<\eps^{\frac{2}{3}+\delta} \text{ for } i=2,3\right\};
\end{equation*}
\begin{equation*}
\mathcal{C}_{i,2\delta}:=\left\{x\in\Omega :\;|x_i|<\eps^{\frac{2}{3}+2\delta} \text{ and } x_1< \eps^{\frac{2}{3}+2\delta}\right\}.
\end{equation*}
Also Lemma \ref{lemma supsol} can be extended. We can prove using the geometrical setting above and the defined functions and sets the following Lemma.
\begin{lemma}\label{lemma supsol x}
	Assume $ p\in\Omega $, $ 0<\eps<1 $, $ L,\delta $ as indicated in Proposition \ref{supsol ou-u eps x}. Let $ x_i=\Rot_p(x)\cdot e_i $ for $ i=1,2,3 $. Let $ W^\pm_i $ as above. Then there exist a constant $ \alpha>0 $ depending only on $ \delta $ and a constant $ C>0 $ depending on $ \Omega $, $ g_\nu $, $ c_0 $ and $ c_1 $ but independent of $ \eps $ and $ p $ and suitable $ b>0 $ and $ L>0 $ such that for $ i=2,3 $
	\begin{numcases}{}
	W^\pm_i(x)\geq 0 & in $ \Omega $\label{Wpmlemma1x},\\
	W^\pm_i(x)\geq \frac{\pi}{2}-\arctan(2) & in $\mathcal{C}^\pm_{i}$\label{Wpmlemma2x},\\
	W^\pm_i(x)\leq C\eps^{\alpha} & in $\mathcal{C}_{i,2\delta}$\label{Wpmlemma4x},\\
	\LL_{\Omega}^\eps\left(W^{\pm}_i\right)(x)\geq-\frac{C}{\sqrt{L}}e^{-\frac{c_0\deps{x}}{8}} & in $\mathcal{C}_{\delta}$\label{Wpmlemma5x}.
	\end{numcases}
	\begin{proof}
	As in the proof of Lemma \ref{lemma supsol} it is enough to prove it for $ W=W^-_2 $. Again we consider $ \rho=\rho_2^- $, $ \theta=\theta_2^- $, $ F=F^-_2 $, $ G=G^-_2 $, $ H=H^-_2 $, $ \mathcal{C}_{2\delta}=\mathcal{C}_{2,2\delta} $ and $ \mathcal{C}=\mathcal{C}_{2} $. We only have to show claim \eqref{Wpmlemma5x}, since all other claims work exactly in the same way. The only thing that is needed is that $ \eps^{\frac{2}{3}+2\delta}=\eps^\delta \eps^{\frac{2}{3}+\delta}=\eps^{2\delta}\eps^{\frac{2}{3}} $. In this case we have $ \beta= \frac{1-12\delta}{3}>0 $ and $ \eps<\left(\frac{1}{L}\right)^{\frac{1}{\beta}} $.\\
	
	In order to prove \eqref{Wpmlemma5x} we follow the strategy of Lemma \ref{lemma supsol}. We expand hence by Taylor the function $ W $ putting together equations \eqref{taylor+error F}, \eqref{taylor+error H} and \eqref{taylor+error G}. Exactly as in Section 4.3 we consider for $ x\in\mathcal{C}_\delta $ the three cases $ \rho(x)<L\eps $, $ \rho(x)\geq L\eps $ with $ d(x)<\eps $ and $ \rho(x)\geq L\eps $ with $ d(x)\geq \eps $. For each of these situations the same estimates of the error term $ E^3(\eta,x) $ as in \eqref{error F rho small} holds. We substitute for $ W(\eta) $ in the formulation of $ \LL^\eps_\Omega(W)(x) $ the Taylor expansion. For all terms containing the first, second and third derivatives of $ W $ we argued in Lemma \ref{lemma supsol} by the symmetry of the kernel. In this case is not possible anymore. Hence, for that terms, we decompose the kernel $ K_\eps(x;\eta)=K^{\alpha(x)}_\eps(x-\eta)+R_\eps^x(x;\eta) $ according to claim (i) of Proposition \ref{prop K eps}. For $ K_\eps^{\alpha(x)} $ we use the same arguments as in Section 4. For the terms with the remainder $ R_\eps^x $ we estimate in a different way. We notice first of all that if $ x\in\mathcal{C}_\delta $, then \begin{equation}\label{2over3}
	\rho(x)\leq 2\eps^{\frac{2}{3}+\delta}<\eps^{\frac{2}{3}}
	\end{equation} taking $ \eps>0 $ sufficiently small. Since $ \rho>\frac{L}{2}\eps $ and $ \cos(\theta(x))\rho(x)\geq \frac{L}{2}\eps  $ we recall also that for $ n\geq 1 $ \begin{equation*}\begin{split}	
\eps^n\left|\nabla_x^nW(x)\right|\leq& C_F\frac{\eps^n}{\rho^n(x)}+ bC_H\frac{\eps^{n+2}}{\rho^{n+2}(x)}+ aC_G\frac{\eps}{L^{n-\frac{1}{2}}\eps^n\rho(x)} .
	\end{split}
	\end{equation*} This implies, using the estimates \eqref{R x eps} and \eqref{2over3} and $ \rho>\frac{L}{2}\eps $ that
	\begin{equation*}
	\begin{split}
	\left|\nabla_x^nW(x)\right|\int_{\RR^3} d\eta\; |R^x_\eps(x;\eta)||x-\eta|^n\leq& \frac{C(c_0,c_1)\eps^2}{\rho(x)}\left(\frac{C_F\eps^{n-1}}{\rho^{n-1}}+bC_H\frac{\eps^{n+1}}{\rho^{n+1}(x)}+\frac{aC_G}{L^{n-\frac{1}{2}}}\right)\\=& \frac{C(c_0,c_1)\eps^4}{\rho^4(x)}\left(\frac{C_F\eps^{n-1}}{\rho^{n-1}}+bC_H\frac{\eps^{n+1}}{\rho^{n+1}(x)}+\frac{aC_G}{L^{n-\frac{1}{2}}}\right)\eps \left(\frac{\rho(x)}{\eps}\right)^3\\\leq& \frac{C(c_0,c_1)\eps^4}{\rho^4(x)}\left(\frac{C_F\eps^{n-1}}{\rho^{n-1}}+bC_H\frac{\eps^{n+1}}{\rho^{n+1}(x)}+\frac{aC_G}{L^{n-\frac{1}{2}}}\right)\\\leq& \frac{C(c_0,c_1)\eps^4}{\rho^4(x)}\left(C_F+\frac{bC_H}{L^2}+\frac{aC_G}{\sqrt{L}}\right).
	\end{split}
	\end{equation*}
	Hence, since $ \frac{\eps^4}{\rho^4(x)}= \left(\frac{\eps}{\rho(x)}\right)^{\frac{5}{2}} \left(\frac{\eps}{\rho(x)}\right)^{\frac{3}{2}}$ all arguments and estimates we had in the proof of Lemma \ref{lemma supsol} for the first case when $ \rho(x)<L\eps $ can be obtained also for this function $ W $ and for $ \alpha\in C^3\left(\Omega\right) $. For the other two cases, when $ \rho(x)\geq L\eps $ with $ d(x)<\eps $ or $ d(x)\geq \eps $ we need also to add to the assumption $ a>b $ the assumption $ a=2b $ (or even $ a=b+1 $). Then the term coming from the integral
	\begin{equation*}
	-\left|\nabla_x^nG(x)\right|\int_{\RR^3} d\eta \;|R^x_\eps(x;\eta)||x-\eta|^n\geq - \frac{C(c_0,c_1)2b C_G\eps^4}{\sqrt{L}\rho^4(x)}
	\end{equation*}
	can be always absorbed taking $ L $ large enough by the term coming from the Laplacian of $ H $ when integrating the second derivative term of the Taylor expansion with the kernel $ K^{\alpha(x)}_\eps $, i.e. it is absorbed by $ \frac{2b}{3\alpha^2(x)}\frac{\eps^4}{\rho^4(x)}\geq  \frac{2b}{3c_1^2}\frac{\eps^4}{\rho^4(x)} $. The remaining arguments and computation are similar to the one in the proof of Lemma \ref{lemma supsol}.
	 We thus refer to the that proof which implies Lemma \ref{lemma supsol x}.
	\end{proof}
\end{lemma}
Arguing as in Section 4.3, Lemma \ref{lemma supsol x} implies now Proposition \ref{L ou-u x} for the supersolution $ W_{\eps,L} $ as given in equation \eqref{Welx}. In this case we have to use Theorem \ref{interior supsol x} instead of Theorem \ref{interior supsol}.
\subsection{Rigorous proof of the convergence of $ u^\eps $ to the solution of the new boundary value problem}
We are now ready conclude the proof of the convergence of $ u^\eps $ to the function $ v $, solution of the boundary value problem
\begin{equation}\label{final bvp x}
\begin{cases}
-\Div\left(\frac{1}{\alpha(x)}\nabla_x v(x)\right) =0 & x\in\Omega,\\v(p)=\ou_\infty(p) & p\in\bnd.
\end{cases}
\end{equation}
To this end, we generalize Section 4.4 for the case $ \alpha\in C^3\left(\overline{\Omega}\right) $. Again, we decompose $ \Omega $ in new regions. In this case though, their distance from the boundary will be of the order $ \eps^{\frac{2}{3}} $. Since the results in this last section are analogous to those we obtained in Section 4.4 we use the same notation. Hence, we define in this case $ \hat{\Omega}_\eps:=\left\{x\in\Omega: d(x)>\eps^{\frac{2}{3}+2\delta}\right\} $, $ \Sigma_\eps:=\left\{x\in\Omega: \eps^{\frac{2}{3}+4\delta}<d(x)\leq\eps^{\frac{2}{3}+2\delta}\right\} $ and their union $ \Omega_\eps=\hat{\Omega}_\eps\cup\Sigma_\eps $. We also define for $ 0<\sigma<1 $ sufficiently small independent of $ \eps $ the set $ \Omega^\sigma:=\Omega\cup\left\{x\in\Omega^c: d(x)<\sigma\right\} $. Recall the continuous projection $ \pi_{\bnd} $ as given in \eqref{proj} and the estimate $ \left|\Rot_{\pi_{\bnd}(x)}(x)\cdot e_1\right|\geq d(x) $ as we have seen in \eqref{Rotpx distance}. Then as in Lemma \ref{estimate in sigma  eps} we can prove the following result.
\begin{lemma}\label{estimate in sigma eps x}
Let $ 0<\eps<1 $ sufficiently small, $ C,\; \alpha,\; L,\; 0<\delta<\frac{1}{12} $ according to Corollary \ref{good estimatex}. Then 
\begin{equation*}\label{estimate sigma eps x}
\sup_{x\in\Sigma_\eps}\left|\ou_\infty\left(\pi_{\bnd}(x)\right)-u^\eps(x)\right|\leq C\left(\eps^\alpha+\frac{1}{\sqrt{L}}\right)+C(c_0,c_1) \eps^{\frac{1}{3}-4\delta}.
\end{equation*}
\begin{proof}
We combine the estimate in \eqref{Rotpx distance} with Lemma \ref{exp rate}, we use the Corollary \ref{good estimatex}, the fact that $ 0<\delta<\frac{1}{12} $ and the following estimate
\begin{equation*}
\sup_{x\in\Sigma_\eps}\left|\ou_\infty\left(\pi_{\bnd}(x)\right)-u^\eps(x)\right|\leq \sup_{x\in\Sigma_\eps}\left|\ou_\infty\left(\pi_{\bnd}(x)\right)-\oU_\eps\left(x,\pi_{\bnd}(x)\right)\right|+\sup_{x\in\Sigma_\eps}\left|\oU_\eps\left(x,\pi_{\bnd}(x)\right)-u^\eps(x)\right|.
\end{equation*}
\end{proof}
\end{lemma}
Similarly to Section 4.4 we consider the function $ v_\sigma $, solution to the boundary value problem
\begin{equation}\label{bvp sigma x}
\begin{cases}
-\Div\left(\frac{1}{\overline{\alpha}(x)}\nabla_x v_\sigma(x)\right) =0 & x\in\Omega^\sigma,\\v(x)=\ou_\infty(\pi_{\bnd}(x)) & x\in\bnd^\sigma,
\end{cases}
\end{equation}
where we consider the smooth extension of $ \alpha $ as defined at the beginning of Section 5.1. Since $ \overline{\alpha}\in C^3(\RR^3) $, $ \pi_{\bnd} $ is a continuous bijection and $ \ou_\infty $ is Lipschitz, the theory on elliptic regularity assures that $ v_\sigma $ uniquely exists and it is also three times continuously differentiable, i.e. $ v_\sigma\in C^3\left(\Omega^\sigma\right)\cap C\left(\overline{\Omega^\sigma}\right) $. For the same reason also the function $ v $ defined in \eqref{final bvp x} belongs to $ C^3(\Omega)\cap C(\overline{\Omega}) $. We denote again by $ \omega $ the modulus of continuity of $ v_\sigma $ and by $ \omega_2 $ the one of $ v $. Moreover, the elliptic equation satisfies the maximum principle, hence for all $ x\in\Omega $ we can estimate $$ \left|v(x)-v_\sigma(x)\right|\leq \max\limits_{x\in\bnd^\sigma}\left|v_\sigma\left(\pi_{\bnd}(x)\right)-v_\sigma(x)\right|\leq \omega(\sigma) .$$
We can now prove a suitable new version of Lemma \ref{estimate op vsigma-ueps}. 
\begin{lemma}\label{estimate op vsigma-ueps x}
Let $ x\in\hat{\Omega}_\eps $, $ 0<\delta<\frac{1}{12} $ and $ \beta=\frac{1-6\delta}{3}>0 $. Then
\begin{equation*}\label{op vsigma-ueps x}
\left|\LL_{\Omega}^\eps\left(v_\sigma-u^\eps\right)(x)\right|\leq C(\Omega, g_\nu, c_0,c_1) e^{-\frac{c_0d(x)}{2\eps}}\left(\eps^\beta+C_\sigma\eps\right)+ C(\Omega, \sigma, c_0,c_1)\eps^3,
\end{equation*}
for some constants $  C(\Omega, g_\nu, c_0,c_1)>0 $ and $ C(\Omega, \sigma, c_0,c_1) $ and $ 0<\eps<1 $ sufficiently small.
\begin{proof}
Since $ x\in\hat{\Omega}_\eps $ then $ d(x)>\eps^{\frac{2}{3}+2\delta} $. Hence, as we have seen at the beginning of Section 5.3 we obtain for $ \beta $ as in the Lemma
\begin{equation*}
\left|\LL_{\Omega}^\eps\left(u^\eps\right)(x)\right|\leq \Arrowvert g\Arrowvert_1 e^{-\frac{c_0d}{\eps}}\leq C(g_\nu, c_0) e^{-\frac{c_0d}{2\eps}} \eps^\beta.
\end{equation*}
We consider now the operator acting on $ v_\sigma $. Since $\Omega \subsetneq \Omega^\sigma $ there exists a constant $ c_\sigma>0 $ depending on $ v_\sigma $ such that $ \sup\limits_{0\leq n\leq3}\sup\limits_{x\in\Omega} \Arrowvert \nabla_x^n v_\sigma \Arrowvert_{\infty}\leq c_\sigma $. We expand $ v_\sigma $ in $ \Omega $ with Taylor as $ v_\sigma(\eta)=v_\sigma(x)+\nabla_x v_\sigma(x)\cdot (\eta-x)+\frac{1}{2}(\eta-x)^\top\nabla_x^2 v_\sigma(x)(\eta-x)+E^3(\eta,x)$, where $ \left|E^3(\eta,x)\right|\leq c_\sigma |x-\eta|^3 $. We want to use the expansion of $ v_\sigma $ together with the fact that this function solves the elliptic equation as given in \eqref{bvp sigma x}. In the case of constant coefficient, the estimate on the operator was the result of the symmetry of the kernel. Indeed, the integral in $ \RR^3 $ of the term with the first derivative in the Taylor expansion was zero and for the term with the second derivative we obtained the Laplacian, which was in that case also zero. As we have already noticed, when the absorption coefficient $ \alpha $ is space dependent the kernel $ K_\eps $ is no longer symmetric and moreover $ v_\sigma $ solves a more general elliptic equation. 

Our strategy now is to find a decomposition of the kernel $ K_\eps $ in such a way that we can recover the elliptic equation \eqref{bvp sigma x} and with a remainder which gives errors of the order $ \eps^3 $. This decomposition is given by claim (iii) in Proposition \ref{prop K eps}. We recall $ K_\eps(x;\eta)=  K^{\alpha(x)}_\eps(x-\eta)+\mathcal{K}_\eps^1(x-\eta)+\mathcal{K}^2_\eps(x-\eta)+\tilde{R}^x_\eps(x;\eta)$. Analogously as we have computed in \eqref{R x eps} we see in this case
\begin{equation}\label{R3 x eps}
\int_{\RR^3} \tilde{R}^x_\eps(x;\eta) |x-\eta|^n \;d\eta\leq C(c_0,c_1)\eps^{n+2}.
\end{equation}
With this decomposition we recover the elliptic equation. Indeed, using equation \eqref{laplacian taylor} we compute
\begin{equation}\label{2term}
\int_{\RR^3}d\eta \left(K^{\alpha(x)}_\eps(x-\eta)+\mathcal{K}_\eps^1(x-\eta)+\mathcal{K}^2_\eps(x-\eta)\right)\frac{1}{2}(\eta-x)^\top\nabla_x^2 v_\sigma(x)(\eta-x)=\frac{\eps^2}{3\alpha^2(x)}\Delta v_\sigma(x),
\end{equation}
where we used that $ \mathcal{K}_\eps^1 $ and $ \mathcal{K}_\eps^2 $ are antisymmetric while $ (\eta-x)^\top\nabla_x^2 v_\sigma(x)(\eta-x) $ is symmetric.

Before moving to the term containing $ \nabla_x v_\sigma(x)\cdot (\eta-x) $ we notice that for any symmetric function $ F(x-\eta) $ we can compute
\begin{equation}\label{1term1}
\begin{split}
\int_{\RR^3}d\eta\;& F(x-\eta) \left(\nabla_x \alpha(x)\cdot (\eta-x)\right)\left(\nabla_x v_\sigma(x)\cdot (\eta-x)\right)\\=&
\sum_{i\ne j}\partial_i\alpha(x)\partial_jv_\sigma(x)\int_{\RR^3}d\eta\; F(x-\eta) (\eta-x)_i(\eta-x)_j\\&+\sum_{i=1}^3 \partial_i\alpha(x)\partial_iv_\sigma(x)\int_{\RR^3}d\eta\; F(x-\eta)(\eta-x)_i^2\\=&\frac{1}{3}\nabla_x\alpha(x)\cdot \nabla_xv_\sigma(x)\int_{\RR^3}d\eta\; F(x-\eta)|\eta-x|^2,
\end{split}
\end{equation}
where we used the symmetry of $ F $. Hence, since $ K^{\alpha(x)}_\eps $ is symmetric using the definition of $ \mathcal{K}_\eps^1 $ and $ \mathcal{K}_\eps^2 $ and equation \eqref{1term1} we conclude

\begin{equation}\label{1term2}
\begin{split}
\int_{\RR^3}d\eta\;& \left(K^{\alpha(x)}_\eps(x-\eta)+\mathcal{K}_\eps^1(x-\eta)+\mathcal{K}^2_\eps(x-\eta)\right)\left(\nabla_x v_\sigma(x)\cdot (\eta-x)\right)\\=&
\int_{\RR^3}d\eta\; \frac{\alpha(x)e^{-\frac{\alpha(x)|x-\eta|}{\eps}}}{4\pi\eps|x-\eta|^2}\left(\frac{1}{\alpha(x)}-\frac{|x-\eta|}{2\eps}\right)\left(\nabla_x \alpha(x)\cdot (\eta-x)\right)\left(\nabla_x v_\sigma(x)\cdot (\eta-x)\right)\\=&\frac{1}{3}\nabla_x\alpha(x)\cdot \nabla_xv_\sigma(x)\int_{\RR^3}d\eta\; \frac{\alpha(x)e^{-\frac{\alpha(x)|x-\eta|}{\eps}}}{4\pi\eps|x-\eta|^2}\left(\frac{1}{\alpha(x)}-\frac{|x-\eta|}{2\eps}\right)|\eta-x|^2\\=&\frac{\eps^2}{3\alpha^3(x)}\nabla_x\alpha(x)\cdot \nabla_xv_\sigma(x)\int_{\RR^3}d\eta\; \frac{e^{-|\eta|}}{4\pi|\eta|^2}\left(1-\frac{|\eta|}{2}\right)|\eta|^2=-\frac{\eps^2}{3\alpha^3(x)}\nabla_x\alpha(x)\cdot \nabla_xv_\sigma(x).
\end{split}
\end{equation}
Hence, equations \eqref{2term} and \eqref{1term2} imply 
\begin{equation}\label{1+2term}
\begin{split}
\int_{\RR^3}d\eta\;& \left(K^{\alpha(x)}_\eps(x-\eta)+\mathcal{K}_\eps^1(x-\eta)+\mathcal{K}^2_\eps(x-\eta)\right)\left(\nabla_x v_\sigma(x)\cdot (\eta-x)+\frac{1}{2}(\eta-x)^\top\nabla_x^2 v_\sigma(x)(\eta-x)\right)\\=&\frac{\eps^2}{3\alpha(x)}\left(\frac{\Delta v_\sigma(x)}{\alpha(x)}-\frac{\nabla_x\alpha(x)}{\alpha(^2x)}\cdot \nabla_xv_\sigma(x)\right)=\frac{\eps^2}{3\alpha(x)}\Div\left(\frac{1}{\alpha(x)}\nabla_xv_\sigma(x)\right)=0.
\end{split}
\end{equation}
We are ready now to conclude the estimate of the operator acting on $ v_\sigma $. Using indeed that $ |x-\eta|>d(x) $ for $\eta\in \Omega^c $ we compute 
\begin{equation*}
\begin{split}
\left|\LL_\Omega^\eps\right.&\left.\left(v_\sigma\right)(x)\right|\leq \left|v_\sigma(x)-\int_\Omega \Epskern{\eta-x}\left[v_\sigma(x)+\nabla_x v_\sigma(x)\cdot (\eta-x)+\frac{1}{2}(\eta-x)^\top\nabla_x^2 v_\sigma(x)(\eta-x)\right]\;d\eta\right|\\
&+\left|\int_\Omega \Epskern{\eta-x} E^3(\eta,x)\;d\eta\right|\\\leq & v_\sigma(x)\int_{B^c_{d(x)}(x)}\Epskern{\eta-x}\;d\eta+C(c_\sigma)\int_{\RR^3}d\eta\;\left|\tilde{R}^x_\eps(x;\eta)\right|\left(|x-\eta|+|x-\eta|^2\right)\\&+ C(c_0,c_1,\sigma)(\eps+\eps^2+\eps^3)\int_{B^c_{d_\eps(x)}(0)}\frac{e^{-|y|}}{4\pi |y|^2}\left(|y|+|y|^2+|y|^3+|y|^4\right)\;dy\\&+ C(c_0,c_1,\sigma)\eps^3\int_{\RR^3}\frac{e^{-\frac{|y|}{\eps}}}{4\pi \eps^3}\frac{|y|}{\eps}\;dy\leq C(c_0,c_1,\sigma,\Omega)e^{-\frac{c_0d(x)}{2\eps}}\left(\eps^\beta+\eps\right)+ C(c_0,c_1,\sigma)\eps^3,
\end{split}
\end{equation*}
where we first used $ d(x)>\eps^{\frac{2}{3}+2\delta} $, then we decomposed the kernel according to claim (iii) in Proposition \ref{prop K eps}, we applied the result in \eqref{1+2term}, the estimate for the remainder $ \tilde{R}^x_\eps $ as given in \eqref{R3 x eps} and finally the estimate $ K_\eps\leq C(c_0,c_1)K^{c_0}_\eps $. This ended the proof of Lemma \ref{estimate op vsigma-ueps x}.
\end{proof}
\end{lemma}
Similarly as in Lemma \ref{estimate sigma eps}, Lemma \ref{estimate in sigma eps x}, the maximum principle for elliptic operators and the uniform continuity of $ v $ imply the following result
\begin{lemma}\label{sigma eps x}
Let $ x\in\Sigma_\eps $ and $ \eps>0 $ small enough. Then the following uniform bound holds
\begin{equation*}
\left|v_\sigma(x)-u^\eps(x)\right|\leq \omega(\sigma)+\omega_2\left(\eps^{\frac{2}{3}+2\delta}\right)+ C\left(\eps^\alpha+\frac{1}{\sqrt{L}}\right)+C(c_0,c_1) \eps^{\frac{1}{3}-4\delta}.
\end{equation*}
\end{lemma}
We have now all elements for completing the proof of the convergence of $ u^\eps $ to $ v $. 
\begin{theorem}\label{convergence up to the boundary x}
	$ u^\eps $ converges to $ v $ uniformly in every compact set.
\begin{proof}
We argue exactly as in Theorem \ref{convergence up to the boundary} applying the maximum principle to the operator $ \LL_{\Omega_\eps}^\eps $. To this end we see first of all that 
$$ \int_{\Omega\setminus\Omega_\eps}d\eta\;K_\eps(x\;\eta)\left|v_\sigma(\eta)-u^\eps(\eta)\right|\leq \frac{C(c_0,c_1,g)}{\eps}\exp\left(-\frac{\eps^{\frac{6\delta-1}{3}}}{2}\right)\int_{\Omega\setminus\Omega_\eps}d\eta\;\frac{1}{|x-\eta|^2}\leq C(c_0,c_1,\Omega,g)\eps^3,  $$
where we used $ |x-\eta|>\frac{\eps^{\frac{2}{3}+2\delta}}{2} $ for $ \eta\in\Omega\setminus\Omega_\eps $ and $ \eps $ sufficiently small and the well-known estimate $ |x|^ne^{-|x|}\leq C_n $ with $ n=13 $ and $ \delta<\frac{1}{78} $. Hence, Lemma \ref{estimate op vsigma-ueps x} implies for $ x\in\hat{\Omega}_\eps $ and $ \beta=\frac{1-6\delta}{3} $
$$ \left|\LL_{\Omega_\eps}^\eps\left(v_\sigma-u^\eps\right)(x)\right|\leq C(\Omega, g_\nu, c_0,c_1) e^{-\frac{c_0d(x)}{2\eps}}\left(\eps^\beta+C_\sigma\eps\right)+ C(\Omega, \sigma, c_0,c_1)\eps^3. $$
Moreover, Lemma \ref{sigma eps x} assures that $$ \left|v_\sigma(x)-u^\eps(x)\right|\leq \omega(\sigma)+\omega_2\left(\eps^{\frac{2}{3}+2\delta}\right)+ C\left(\eps^\alpha+\frac{1}{\sqrt{L}}\right)+C(c_0,c_1) \eps^{\frac{1}{3}-4\delta} $$ for $ x\in\Sigma_\eps $.

As we have seen in \eqref{Phi} the supersolution $ \Phi^\eps(x) $ satisfies also $ \LL_{\Omega}^\eps\left(\Phi^\eps\right)(x)\geq C \eps^2 $. Hence, we refer now to the proof Theorem \ref{convergence up to the boundary}, which works here in the same way just replacing the supersolution with the suitable $ \Phi^\eps $ defined in Theorem \ref{interior supsol x}. Since the arguments are the same we omit the details.
\end{proof}
\end{theorem}
\bibliographystyle{siam}
\bibliography{bibliography}
\end{document}